\documentclass[lettersize,journal]{IEEEtran}

\usepackage{mypackage}
\usepackage{mymacros2}
\usepackage{mymacros_brackets}
\usepackage{myenv}
\usepackage{convex_opti}
\usepackage{mymacros_support}
\begin{document}

\title{On Support Recovery with Sparse CCA: Information Theoretic and Computational Limits}

\author{Nilanjana~Laha, 
Rajarshi~Mukherjee
        \thanks{Nilanjana Laha is with the Department of Statistics, Texas A{\&}M University,  College Station, TX 77843, (email: nlaha@tamu.edu)}
\thanks{Rajarshi Mukherjee is with the Department of Biostatistics, Harvard T. H. Chan School of Public Health,   677 Huntington Ave, Boston, MA 02115, (email: ram521@mail.harvard.edu)}

\thanks{This work was supported by National Institutes of Health grant P42ES030990.}}
\markboth{Journal of \LaTeX\ Class Files,~Vol.~14, No.~8, August~2021}%
{Laha \MakeLowercase{\textit{et al.}}: Support Recovery}


\ifCLASSINFOpdf
\else
\fi
%
%

%
%

\hyphenation{op-tical net-works semi-conduc-tor}

\maketitle

\begin{abstract}
In this paper, we consider asymptotically exact support recovery in the context of high dimensional and sparse Canonical Correlation Analysis (CCA). Our main results describe four regimes of interest based on information theoretic and computational considerations. In regimes of ``low" sparsity we describe a simple, general, and computationally easy method for support recovery, whereas in a regime of ``high" sparsity, it turns out that support recovery is  information theoretically impossible. For the sake of information theoretic lower bounds, our results also demonstrate a non-trivial requirement on the ``minimal" size of the nonzero elements of the canonical vectors that is required for asymptotically consistent support recovery. Subsequently, the regime of ``moderate" sparsity is further divided into two sub-regimes. In the lower of the two sparsity regimes, we show that polynomial time support recovery is possible by using a sharp analysis of a co-ordinate thresholding \cite{deshpande2014} type method. In contrast, in the higher end of the moderate sparsity regime, appealing to the ``Low Degree Polynomial" Conjecture \cite{kunisky2019}, we  provide evidence  that polynomial time support recovery methods are  inconsistent.  Finally, we carry out numerical experiments to compare the efficacy of  various methods discussed.
\end{abstract}


\section{Introduction}
\IEEEPARstart Canonical Correlation Analysis (CCA) is a highly popular technique to perform initial dimension reduction while exploring relationships between two multivariate objects.  Due to its natural interpretability and success in finding latent information, CCA has found enthusiasm across  vast canvas of disciplines, which include, but are not limited to  psychology and agriculture, information
retrieving \cite{hardoon2004canonical,rasiwasia2010new,gong2014multi}, brain-computer interface \cite{bin2009online}, neuroimaging
\cite{avants2010dementia}, genomics \cite{witten2009}, organizational research
\cite{bagozzi2011measurement}, natural language processing \cite{dhillon2011multi,faruqui2014improving}, fMRI  data analysis \cite{friman2003adaptive}, computer vision \cite{kim2007tensor}, and speech recognition \cite{arora2013multi,wang2015deep}.

 Early developments in the theory and applications of CCA have now been well documented in the statistical literature, and we refer the interested reader to \cite{anderson2003} and references therein for further details. 
However, the modern surge in interest for CCA, often being motivated by \textcolor{black}{data from high throughput biological experiments} \cite{le2009, lee2011, waaijenborg2008}, requires re-thinking several aspects of the traditional theory and methods.  A natural structural constraint that has gained popularity in this regard, is that of sparsity, i.e.,  the phenomenon of an (unknown)  collection of variables being related to each other.  In order to formally introduce the framework of sparse CCA, we present our statistical setup next.
\textcolor{black}{
We shall consider $n$-i.i.d. samples $(X_i,Y_i)\sim \mathbb{P}$ with $X_i\in \mathbb{R}^p$ and $Y_i\in \mathbb{R}^q$ being multivariate mean zero random variables with joint variance covariance matrix 
\begin{equation}\label{def: sigma}
 \Sigma=\begin{bmatrix}
\Sx & \Syx\\
\Syx & \Sy\\
\end{bmatrix}.   
\end{equation}
The first canonical correlation $\Lambda_1$ is then defined as the maximum possible correlation between two linear combinations of $X$ and $Y$. This definition interprets   $\Lambda_1$ as the  optimal value of the following maximization problem:
\begin{maxi}[4]
    {u\in\mathbb R^{p},v\in\mathbb R^{q}}{u^T\Sxy v}
     {\label{opt: sparse CCA}}{}
     \addConstraint{u^T\Sx u=v^T\Sy v}{= 1.}
       \end{maxi}
     The solutions  to \eqref{opt: sparse CCA}
     are the vectors that maximize the correlation of the projections of $X$ and $Y$ in those respective directions. Higher order canonical correlations can thereafter be defined in  a recursive fashion  (cf. \cite{anderson1999asymptotic}).
    In particular,  for $j\geq 1$, we define the $j^{\rm th}$ canonical correlation  $\Lambda_j$ and the corresponding directions $u_j$ and $v_j$ by maximizing \eqref{opt: sparse CCA} with the additional constraint
    \begin{align}\label{def: uj vj}
      u^T\Sx u_l=v^T\Sy v_l=0,\quad 0\leq l\leq j-1.  
    \end{align}
    } 

As mentioned earlier, in many modern data examples, the sample size $n$ is typically at most comparable to or much smaller than $p$ or $q$ -- rendering the classical CCA inconsistent and inadequate without further structural assumptions \cite{cai2018rate,ma2020subspace,bao2019}. The framework of Sparse Canonical Correlation Analysis (SCCA) (cf. \cite{witten2009, mai}), where the $u_i$'s and the $v_i$'s are sparse vectors, was subsequently developed to target low dimensional structures (that allows consistent estimation) when $p,q$ are potentially larger than $n$. 
The corresponding sparse estimates of the leading canonical directions   naturally perform variable selection, thereby  leading to the recovery of their  support (cf. \cite{witten2009, mai, waaijenborg2008, solari2019}). It is unknown, however, under what  settings, this \naive\ method of support recovery, or any other method for the matter, is consistent. 
The support recovery of the leading canonical directions  serves an important purpose of identifying groups of variables that explain the most linear dependence among high dimensional random objects ($X$ and $Y$) under study -- and thereby renders crucial interpretability. Asymptotically optimal support recovery  is yet to be explored systematically in the context of SCCA -- both theoretically, and from the computational viewpoint. In fact, despite the renewed enthusiasm for CCA,  both the theoretical and applied communities  have mainly focused on the estimation  of the leading canonical directions, and relevant scalable algorithms -- see, e.g., \cite{gao2013,gao2015,gao2017,ma2020subspace,mai}. This paper  explores the crucial question of support recovery in the context of SCCA. 
\footnote{In this paper, by support recovery, we refer to the exact recovery of the combined support of the $u_i$'s (or the $v_i$'s) corresponding to nonzero $\Lambda_i$'s.}

The problem of support recovery for SCCA naturally connects to a vast class of variable selection problems (cf. \cite{wainwright2009, amini2009, butucea2015, butucea2017adaptive, meinshausen2010stability}). The problem closest in terms of complexity turns out to be the sparse PCA (SPCA) problem \cite{johnstone2009}. Support recovery  in the latter problem is known to present interesting information theoretic  and computational  bottlenecks  (cf. \cite{boaz2015, amini2009, ding2019, arous2020}). Moreover,  information theoretic and computational issues  also arise in context of SCCA estimation problem   (cf. \cite{gao2013,gao2015,gao2017,mai}). In view of the above, it is natural to expect that such information theoretic and computational issues exist in context of   SCCA support recovery problem as well. However, the techniques used in  SPCA support recovery analysis are  not directly applicable to the SCCA problem, which poses  additional challenges due to the presence of high dimensional nuisance parameters $\Sx$ and $\Sy$. 
  The main focus of our work is therefore  retrieving the complete picture of the information theoretic and computational limitations  of SCCA support recovery. Before going into further details,  we  present a brief summary of our contributions, and defer the discussions on the main subtleties
  to Section \ref{sec:main_results}. Our methods can be implemented using the R package \texttt{Support.CCA} \cite{github_laha}.

\subsection{Summary of Main Results} 

We say a method successfully recovers the support if it achieves exact recovery with probability tending to one uniformly over the  sparse parameter spaces defined in Section \ref{sec:math_formalism}.  In the sequel, we denote the cardinality of the combined support of the $u_i$'s and the $v_i$'s by $s_x$ and $s_y$, respectively. Thus $s_x$ and $s_y$ will be our respective sparsity parameters. Our main contributions are listed below.

\subsubsection{General methodology}
In Section~\ref{sec:easy_regime_general}, we construct a general algorithm called \RecoverSupp,  which leads to successful support recovery whenever the latter is information theoretically tractable. This also serves  as the first step in creating a polynomial time procedure for recovering support in one of the difficult regimes of the problem -- see e,g. Corollary~\ref{cor: CT support recovery}, which shows that  \RecoverSupp\  accompanied by a co-ordinate thresholding type method recovers the support  in polynomial time in a regime that requires subtle analysis. Moreover, Theorem~\ref{thm: support recovery: sub-gaussians} shows that the minimal signal strength required by \RecoverSupp\  matches the information theoretic limit whenever the nuisance precision matrices $\Sx^{-1}$ and $\Sy^{-1}$ are sufficiently sparse.

\subsubsection{Information theoretic and  computational hardness as a function of sparsity} 
 As the sparsity level increases, we show that the CCA support recovery problem transitions from being  efficiently solvable, to NP hard (conjectured), and  to information theoretically impossible.  According to this hardness pattern, the sparsity domain can be partitioned into the following three regimes:
  (i) $s_x,s_y\lesssim \sqn$, (ii) $\sqn\lesssim s_x,s_y\lesssim n/\log(p+q)$, and (iii) $s_x,s_y\gtrsim n/\log(p+q)$. We describe below the distinguishing  behaviours of  these three regimes, which is consistent with the sparse PCA scenario.
\begin{itemize}
    \item We show that when $s_x,s_y\lesssim \sqrt{n/\log(p+q)}$ (``easy regime"), polynomial time support recovery is possible, and well-known consistent estimators of the canonical correlates (cf. \cite{mai, gao2017}) can be utilized to that end. When $\sqrt{n/\log(p+q)}\lesssim s_x,s_y\lesssim\sqn$ (``difficult regime"), we show that a co-ordinate thresholding type algorithm (inspired by \cite{deshpande2014})  succeeds  provided $p+q\asymp n$. We call the last regime ``difficult" because it is unknown whether existing estimation methods like COLAR \cite{gao2017} or SCCA \cite{mai}  have valid statistical guarantees in this regime -- see Section~\ref{sec:easy_regime_general} and Section~\ref{sec: CT} for more details.
    \item  In Section~\ref{sec: low deg}, we show that when $\sqn\lesssim s_x,s_y\lesssim n/\log(p+q)$ (``hard regime"),  support recovery is computationally hard subject to the so called ``low degree polynomial conjecture" recently popularized by \cite{hopkins2017},  \cite{hopkins2018}, and  \cite{kunisky2019}. Of course, this phenomenon is  observable only when $p,q\gtrsim n$, because otherwise, the problem would be solvable by the ordinary CCA analysis (cf. \cite{bao2019, ma2021}). Our findings are consistent with the conjectured computational barrier in context of  SCCA estimation problem  \cite{gao2017}.
    \item When $s_x,s_y\gtrsim n/\log(p+q)$, we show that support recovery is information theoretically impossible (see Section~\ref{sec: IT lower bound}).
\end{itemize}

\subsubsection{Information theoretic hardness as a function of minimal signal strength}
In context of support recovery, the signal strength is quantified by \[\Sig_x=\min_{k\in[p]}\max_{i\in[r]}|(u_i)_k|\quad\text{and}\quad\Sig_y=\min_{k\in[q]}\max_{i\in[r]}|(v_i)_k|.\] 
Generally, support recovery algorithms require  the  signal strength to lie above some threshold. As a concrete example, the detailed analyses provided in \cite{amini2009, deshpande2014}, and \cite{boaz2015} are all based on the nonzero principal component elements being of the order $\pm 1/\sqrt{\rm sparsity}$. To the best of our knowledge, prior to our work, there was no result in the PCA/CCA literature on the information theoretic limit of the minimal signal strength. \textcolor{black}{Generally, PCA studies assume that the top eigenvectors are de-localized, i.e., the principal components have elements of the order $O(1/\sqrt s)$ and thereby mostly considered the cases of de-localized eigenvectors. We do not make any such assumption on the canonical covariates, and thereby we believe that our study paints a more complete picture of the support recovery.}
\begin{itemize}
    \item In Section~\ref{sec: IT lower bound}, we show that $\Sig_x\gtrsim\sqrt{\log(p-s_x)/n}$    (or $\Sig_y\gtrsim\sqrt{\log(q-s_y)/n}$)
    is a necessary requirement for  successful support recovery by $U$ (or $V$).
\end{itemize}

\subsection{Notation}
  For a vector $x\in\RR^p$, we denote its support  by
$D(x)=\{i: x_i\neq 0\}$.
We will overload notation, and for a matrix $A\in\RR^{p\times q}$, we will denote by $D(A)$ the indexes of the nonzero rows of $A$. By an abuse of notation, sometimes we will refer to $D(A)$ as the support of $A$ as well. When $A\in\RR^{p\times q}$ and $\alpha\in\RR^p$ are unknown parameters, generally, the estimator of their supports will be denoted by $\widehat D(A)$ and $\widehat D(\alpha)$, respectively. 
 We let $\NN$ denote the set of all positive numbers, and write $\ZZ$ for the set of all natural numbers $\{0,1,2,\ldots,\}$. For any $n\in\NN$,    We let $[n]$ denote the set $\{1,\ldots,n\}$.  We define the projection of $A$ onto $D\subset [p]\times[q]$ by
  \begin{equation}\label{def:projection operator}
     \lb\mP_D\{A\}\rb_{i,j}=\begin{cases}
A_{i,j} & \text{ if }(i,j)\in D,\\
0 & \text{otherwise.}
\end{cases}
  \end{equation} For any finite set $\mathcal A$, we denote its cardinality by $|\mathcal A|$. Also, for any event $\mathcal B$, we let $1\{ \mathcal B\}$ be the indicator of the event $\mathcal B$. For any $p\in\NN$, we let $\mathbb S^{p-1}$ denote the unit  sphere in $\RR^p$.

 We let $\|\cdot\|_k$ be the usual $l_k$ norm in $\RR^k$ for $k\in\ZZ$.  In particular, we let $\|x\|_0$ denote the number of nonzero elements of a vector $x\in\RR^p$.  For any probability measure $\PP$ on the Borel sigma field of $\RR^p$, we let  $L_2(\PP)$ to be the set of all measurable functions $f:\RR^p\mapsto\RR$ such that $\|f\|_{L_2(\PP)}=\sqrt{\int f^2d\PP}<\infty$. The corresponding $L_2(\PP)$ inner product will be denoted by $\langle \cdot,\cdot\rangle_{L_2(\PP)}$. 
We denote the operator norm and the Frobenius norm of a matrix $A\in\RR^{p\times q}$ by $\|A\|_{op}$ and $\|A\|_F$, respectively. We let $A_{i*}$ and $A_j$ denote the i-th row and $j$-th column of $A$, respectively. For $k\in\NN$, we define the norms $\|A\|_{k,\infty}=\max_{j\in[q]}\|A_j\|_k$ and $\|A\|_{\infty, k}=\max_{i\in[q]}\|A_{i*}\|_k$.
The maximum and minimum eigenvalues of a square matrix $A$   will be denoted  by $\Lambda_{max}(A)$ and $\Lambda_{min}(A)$, respectively.   Also, we let $s(A)$  denote the  maximum number of nonzero entries in any column of $A$, i.e., $s(A)=\max_{j\in[q]}\|A_j\|_0$.

The results in this paper are mostly asymptotic (in $n$) in nature and thus require some standard asymptotic notations.  If $a_n$ and $b_n$ are two sequences of real numbers then $a_n \gg b_n$ (and $a_n \ll b_n$) implies that ${a_n}/{b_n} \rightarrow \infty$ (and ${a_n}/{b_n} \rightarrow 0$) as $n \rightarrow \infty$, respectively. Similarly $a_n \gtrsim b_n$ (and $a_n \lesssim b_n$) implies that $\liminf_{n \rightarrow \infty} {{a_n}/{b_n}} = C$ for some $C \in (0,\infty]$ (and $\limsup_{n \rightarrow \infty} {{a_n}/{b_n}} =C$ for some $C \in [0,\infty)$). Alternatively, $a_n = o(b_n)$ will also imply $a_n \ll b_n$ and $a_n=O(b_n)$ will imply that $\limsup_{n \rightarrow \infty} \ a_n / b_n = C$ for some $C \in [0,\infty)$. \textcolor{black}{We write $a_n\asymp b_n$ if there are positive constants $C_1$ and $C_2$ such that $C_1b_n\leq a_n\leq C_2 b_n$ for all $n\in\NN$.} We will write $a_n=\tilde\Phi(b_n)$ to indicate $a_n$ and $b_n$ are asymptotically of the same order up to a poly-log term. Finally, in our mathematical statements, $C$ and $c$ will be two different generic constants which can vary from line to line.

\section{Mathematical Formalism} \label{sec:math_formalism}

     
    
    We denote the rank of $\Sxy$ by $r$. 
   It can be shown that exactly $r$ canonical correlations are positive and the rest are zero in the model \eqref{opt: sparse CCA}. We will consider the matrices
    $U=[u_1,\ldots,u_r]$ and $V=[v_1,\ldots,v_r]$. From \eqref{opt: sparse CCA} and \eqref{def: uj vj}, it is not hard to see that $U^T\Sx U=I_p$ and $V^T\Sy V=I_q$. The indexes of the nonzero rows of $U$ and $V$, respectively, are the combined support of the $u_i$'s and the $v_i$'s. Since we are interested in the recovery of the latter, it will be useful for us to study of the properties of $U$ and $V$. To that end, we often make use of the following representation connecting $\Sxy$ to $U$ and $V$ \cite{anderson2003}: \begin{equation}\label{representation: preliminary}
        \Sxy=\Sx U\Lambda V^T\Sy=\Sx\left(\sum_{i=1}^r\Lambda_iu_iv_i^T\right)\Sy.
    \end{equation}

 To keep our results straightforward, we restrict our attention to a particular model $\mP(r,s_x,s_y,\B)$ throughout, defined as follows.
\begin{definition}
\label{definition: model}
Suppose $(X,Y)\sim \PP$. Let
 $\B>1$ be a constant. We say
   $\mathbb P\in\mP( r, s_x,s_y, \B)$ if
  \begin{itemize}
  \item[A1] (Sub-Gaussian) $X$ and $Y$ are  sub-Gaussian random vectors (cf.  \cite{vershynin2020}), with joint covariance matrix $\Sigma$ as defined in \eqref{def: sigma}.  
  Also $\text{rank}(\Sxy)=r$.
  \item[A2] Recall the definition of the canonical correlation $\Lambda_i$'s from \eqref{def: uj vj}. Note that by definition, $0\leq \Lambda_r\leq\cdots\leq\Lambda_1$. For $\PP\in\mod$, $\Lambda_r$ additionally satisfies $\Lambda_r\geq 1/\B$. 
  \item[A3] (Sparsity) The number of nonzero rows of $U$ and $V$ are $s_x$ and $s_y$, respectively, that is $s_x=|\cup_{i=1}^rD(\ai)|$ and 
  $s_y=|\cup_{i=1}^rD(\bi)|$. Here $U$ and $V$ are as defined in \eqref{representation: preliminary}.
  \item[A4] (Bounded eigenvalue) 
     \[1/\B<\Lambda_{min}(\Sy), \Lambda_{min}(\Sy), \Lambda_{max}(\Sx), \Lambda_{max}(\Sy)<\B.\]
     \item[A5] (Positive eigen-gap) 
    $\Lambda_i-\Lambda_{i-1}\geq \B^{-1}$ for $i=2,\ldots,r$.
  \end{itemize}
  \end{definition}
  
  Sometimes we will consider a sub-model of $\mod$ where each $\PP\in\mod$ is Gaussian. This model will be denoted by $\modG$, where ``$G$" stands for the Gaussian assumption. Some remarks on the modeling assumptions A1---A5 are in order, which we provide next.
  
  
  \begin{enumerate}
      \item[A1.] We begin by noting that we do not require  $X$ and $Y$ to be jointly sub-Gaussian.  Moreover, the individual sub-Gaussian assumption itself is common in the $s_x,s_y\lesssim\sqrt{n}/\log(p+q)$ regime in the SCCA literature (cf. \cite{mai, gao2017,laha2021}). \textcolor{black}{Our proof techniques depend crucially on the sub-Gaussian assumption. We also anticipate that the  results derived in this paper will change under the violation of this assumption.}
      For the sharper analysis in the difficult regime ($\sqrt{n/\log{(p+q)}}\lesssim s_x,s_y\lesssim\sqrt{n}$), our proof techniques require the Gaussian model $\mP_G$ --  which is in parallel with  \cite{deshpande2014}'s treatment of the sparse PCA in the corresponding difficult regime. In general, the Gaussian spiked model assumption in sparse PCA goes back to \cite{johnstone2001}, and is common in the PCA literature (cf. \cite{amini2009, boaz2015}). 
        \item[A2-A4.] These assumptions are  standard in the analysis of canonical correlations (cf. \cite{mai, gao2017}).
        \item[A5.] This assumption concerns the gap between consecutive canonical correlation strengths. However, we refer to this gap as ``Eigengap" because of its similarity with the Eigengap  in the sparse PCA literature (cf. \cite{deshpande2014, jankova2018}).
This assumption is necessary for the estimation  of the $i$-th canonical covariates. Indeed, if $\lambda_i=\lambda_{i+1}$ then there is no hope of estimating the $i$-th canonical covariates because they are not identifiable, and so support recovery also becomes infeasible. This assumption can be relaxed to requiring only $k$ many $\lambda_i$'s to be strictly larger than $\lambda_{i-1}$'s where $k\leq r$. In this case, we can recover the support of only the first $k$ canonical covariates.
  \end{enumerate}
    
In the following sections, we will denote the preliminary estimators of $U$ and $V$ by $\widehat U$ and $\widehat V$, respectively. The  columns of $\widehat U$ and $\widehat V$ will be denoted by $\hai$ and $\hbi$ ($i\in[r]$), respectively. Therefore $\hai$ and $\hbi$ will stand for the corresponding preliminary estimators of $\ai$ and $\bi$. In case of CCA, the $\ai$'s and $\bi$'s are identifiable only up to a sign flip. Hence, they are also  estimable only up to a sign flip.
Finally, we denote the empirical estimates of $\Sx$, $\Sy$, and $\Sxy$, by $\hSx$, $\hSy$, and $\hSxy$, respectively -- which will often be appended with superscripts to denote their estimation through suitable sub-samples of the data \footnote{e.g., $\hSx^{(j)}$, $\hSy^{(j)}$, and $\hSxy^{(j)}$ will stand for the empirical estimators created from the $j^{th}$-equal split of the data.}. Finally, we  let $C_{\B}$ denote a positive constant which depends on $\PP$ only through $\B$, but can vary from line to line.

\section{Main Results} \label{sec:main_results}

We divide our main results into the following parts based on both statistical and computational difficulties of different regimes. First, in Section \ref{sec:easy_regime_general} we present a general method and associated sufficient conditions for support recovery. This allows us to elicit a sequence of questions regarding necessity of the conditions and remaining gaps both from statistical and computational perspectives. Our subsequent sections are devoted to answering these very  questions.  In particular, in Section \ref{sec: IT lower bound} we discuss information theoretic lower bounds followed by evidence for statistical-computational gaps in Section \ref{sec: low deg}. Finally, we close  a final computational gap in asymptotic regime through sharp analysis of a special co-ordinate-thresholding type method in Section \ref{sec: CT}.


\subsection{\bf A Simple and General Method:}\label{sec:easy_regime_general}

We begin with a simple method for estimating the support, which readily establishes the result for the easy regime,
and sets the directions for the investigation into other more subtle regimes. Since the estimation of $D(U)$ and $D(V)$ are similar, we focus only on the estimation of $D(V)$ for the time being.

 Suppose $\widehat V$ is a row sparse estimator of $V$.
The nonzero indexes of  $\widehat V$ is the most intuitive estimator of $D(V)$. Such an $\widehat V$ is also easily attainable  because most estimators of the canonical directions in the high dimension are sparse (cf. \cite{gao2013, gao2017, mai} among others). \textcolor{black}{Although we have not yet been able to show the validity of this apparently ``\naive" method, we provide numerical results in Section \ref{sec:simulations} to explore its finite sample performance.} However,  a simple method can refine these initial estimators, to often optimally recover the support $D(V)$. We now provide the details of this method and derive its asymptotic properties.

To that end, suppose  we have at our disposal an estimating procedure for $\Sy^{-1}$, which we generically denote by $\hf$ and an estimator $\widehat U\in\RR^{p\times r}$ of $U$. 
We split the sample in two equal parts, and compute  $\widehat U^{(1)}$ and  $\hf^{(1)}$ from the first part of the sample, and the estimator $\hSxy^{(2)}$ from the second part of the sample. Define 
$\widehat V^{clean}= \hf^{(1)}\hSyx^{(2)} \widehat U^{(1)}$.
Our estimator of $D(V)$ is then given by
\begin{equation}\label{def: general method}
    \widehat D(V):= \{i\in[q]: |\widehat V^{clean}_{ik}|>\cut\text{ for some }k\in[r]\},
\end{equation}
where $\cut$ is a pre-specified cut-off  or threshold. We will discuss more on  $\cut$ later. The resulting algorithm  will be referred as \RecoverSupp\  from now on. \textcolor{black}{Algorithm \ref{algo: simple algorithm} gives the algorithm for the support recovery of $V$, but the full version of \RecoverSupp, which estimates  $D(U)$ and $D(V)$ simultaneously, can be found  in Appendix \ref{sec: full version of recoversupp}; see Algorithm  \ref{algo: recoversupp full} there.} 
\RecoverSupp\  is similar in spirit to the ``cleaning" step in the  sparse PCA support recovery literature (cf. \cite{deshpande2014}). 
\textcolor{black}{One thing to remember here is that $\widehat V^{clean}$ is not an estimator $V$. In fact, the $(i,j)$-th element of $\widehat V^{clean}$ is an estimator of $\Lambda_i(v_i)_j$. }

\textcolor{black}{\begin{remark}
In many applications, the rank $r$ may be unknown. \cite{meloun2011} (see Section 4.6.5 therein) suggests to use the screeplot of the  canonical correlations to estimate $r$. Screeplot is also a popular tool to estimate the number of nonzero principal components in PCA analysis \cite{deshpande2014}. For CCA, the screeplot is the plot of the estimated canonical correlations versus their orders. If there is a clear gap between two successive correlations, \cite{meloun2011} suggests taking the larger correlation as the estimator of $\Lambda_r$. One can use \cite{mai}'s SCCA method to estimate the canonical correlations to obtain the screeplot. There can be other ways of estimating $r$. For example, in their trans-eQTL study, \cite{Dutta2021} uses a resampling technique on a holdout dataset to generate observations from the null distribution of the $i$-th canonical correlation estimate under the hypothesis $H_0:\Lambda_i=0$, where $i\in[\min(p,q)]$. The largest $i$, for which the test is rejected, is taken as the estimated rank. A similar  technique has been used by \cite{lock2013} to select the ranks for a related method JIVE. 
\end{remark}}

\begin{algorithm}[h]
\caption{\RecoverSupp\ $(\widehat U^{(1)}, \hf^{(1)},\hSxy^{(2)},\cut,r)$: support recovery of $V$}
\label{algo: simple algorithm}

\begin{algorithmic} 
\REQUIRE \begin{enumerate}
\item Preliminary estimators $\widehat U^{(1)}$  and $\hf^{(1)}$ of $U$ and $\Sy^{-1}$, respectively, based on sample $O_1=(x_i,y_i)_{i=1}^{[n/2]}$.
\item Estimator $\hSxy^{(2)}$ of $\Sxy$ based on sample  $O_2=(x_i,y_i)_{i=[n/2]+1}^{n}$.
\item Threshold level $\cut>0$ and rank $r\in\NN$.
\end{enumerate}
\ENSURE $\widehat D(V)$, an estimator of ${D(V)}$.
\STATE 
\begin{enumerate}
\item \textbf{Cleaning}: $\widehat V^{clean}\leftarrow \hf^{(1)}{{\hSyx}}^{(2)} \widehat U^{(1)}$.
\item \textbf{Threshold:} Compute $ \widehat D(V)$ as in \eqref{def: general method}.
\end{enumerate}

\RETURN $ \widehat D(V)$.
 \end{algorithmic}
\end{algorithm}

It turns out that, albeit being so simple, \RecoverSupp\   has desirable statistical guarantees provided $\widehat U^{(1)}$ and $\hf^{(1)}$ are reasonable estimators of $U$ and $\Sy^{-1}$, respectively. These theoretical properties of \RecoverSupp\ , and the hypotheses and queries generated thereof, lay out  the roadmap for the rest of our paper. However, before getting into  the detailed theoretical analysis of \RecoverSupp\ , we state a $l_2$-consistency condition on $\hai$ and $\hbi$'s, where we remind the readers that we let $\hai$ and $\hbi$ denote the $i$-th columns of $\widehat V$ and $\widehat U$, respectively. Recall also that  the $i$-th columns of $U$ and $V$ are denoted by $u_i$ and $v_i$, respectively.

\begin{cond}[$l_2$ consistency ] \label{Assumption on estimators}
 There exists a function  $\errs\equiv \errs:(n,p,q,s_x,s_y,\B)\mapsto\RR$   so that $|\errs|<1/(2\B\sqrt r)$ and 
the estimators $\hai$ and $\hbi$  of $\ai$ and $\bi$ satisfy
 \[\max_{i\in[r]}\min_{\myv\in\{\pm 1\}}\bl (\myv\hai-\ai)^T\Sx(\myv\hai-\ai)\bl<\errs^2,\]
\[\max_{i\in[r]}\min_{\myv\in\{\pm 1\}}\bl (\myv\hbi-\bi)^T\Sy(\myv\hbi-\bi)\bl<\errs^2\]
with $\PP$ probability $1-o(1)$ uniformly over $\PP\in\mP(r,s_x,s_y,B)$. 
\end{cond}
 We will  discuss  the estimators which satisfy Condition~\ref{Assumption on estimators} later.    
Theorem~\ref{thm: support recovery: sub-gaussians} also requires the  signal strength $\Sig_y$ to be at least of the  order
$\e_n=\xi_n \sqrt{\log(p+q)\sSy/n}$, where the parameter $\xi_n$ depends on  the type of $\hf$ as follows:
\begin{itemize}
    \item[A.] 
    $\hf$ is of type A if there exists $\Cpr>0$ so that $\hf$ satisfies $\|\hf-\Sy^{-1}\|_{\infty,1}\leq \Cpr \sSy\sqrt{(\log q)/n}$ with $\PP$ probability $1-o(1)$ uniformly over $\PP\in\mod$. Here we remind the readers that $\sSy=\max_{j\in[q]}\|(\Sy^{-1})_{j}\|_0$. In this case, $\xi_n=\Cpr\sqrt{\sSy}$.
    \item[B.] $\hf$ is of type B if  $\|\hf-\Sy^{-1}\|_{\infty,2}\leq \Cpr \sqrt{\sSy\log(q)/n}$ with $\PP$ probability $1-o(1)$ uniformly over $\PP\in\modG$ for some $\Cpr>0$. In this case, $\xi_n=\Cpr\max\{\sqrt{r(\log q)/n},1\}$.
    \item[C.] $\hf$ is of type C if $\hf=\Sy^{-1}$. In this case, $\xi_n=1$.
\end{itemize}
The estimation error of $\hf$ clearly decays from type A to C, with the error being zero at type C.
Because   $\sqrt{r(\log q)/n}$ is generally much smaller than $\sSy$, 
 $\xi_n$  shrinks from Case A to Case C monotonously as well. Thus it is fair to say that $\xi_n$  reflects the precision of  the estimator $\hf$ in that $\xi_n$ is smaller if $\hf$ is a sharper estimator. We are now ready to state Theorem~\ref{thm: support recovery: sub-gaussians}. This theorem is  proved in Appendix~\ref{sec: Proof of support rec: subgaussian}.

 \begin{theorem}\label{thm: support recovery: sub-gaussians}
Suppose $\log(p\vee q)=o(n)$ and the estimators $\hai$'s satisfy Condition  \ref{Assumption on estimators}. Further suppose $\hf$ is of  type A, B, or C, which are  stated above. Let $\e_n=\xi_n \sqrt{\log(p+q)\sSy/n}$
 where $\xi_n$ depends on the type of $\hf$ as outlined above.
 Then there exists  a constant $C'_\B>0$, depending only on $\B>0$, so that if
 \begin{equation}\label{intheorem: statement: lower bound on signal}
   \Sig_y>2 C_\B'\e_n,
 \end{equation}
 and 
 $\cut\in[C_\B'\e_n/(2\B),{(\ratio-1)}C_\B'\e_n/(2\B)]$ with $\theta_n=\Sig_y/(C_\B'\e_n)$,
 then  the algorithm \RecoverSupp\  fully recovers $D(V)$ with $\PP$ probability $1-o(1)$ uniformly over $\PP\in\mod$ (for $\hf$ of type A and C), or uniformly over  $\PP\in\modG$ (for $\hf$ of type B).
\end{theorem}

The assumption that $\log p$ and $\log q$ are $o(n)$ appears in all theoretical works of CCA (cf. \cite{gao2017, mai}). A requirement of this type is generally unavoidable. Note that
Theorem~\ref{thm: support recovery: sub-gaussians} implies a more precise estimator $\hf$ requires smaller signal strength for full support recovery. 
\textcolor{black}{
\subsubsection*{Main idea behind the proof of Theorem \ref{thm: support recovery: sub-gaussians}} Because $\Lambda_i(\bi)_k=e_k^T\Sy^{-1}\Syx\ai$, $\widehat V^{clean}_{ik}$ is an estimator of $\Lambda_i(\bi)_k$ for $i\in[q]$ and $k\in[r]$. If $k\notin D(V)$, then $(\bi)_k=0$ for all $i\in[r]$. Therefore, in this case, we expect  $\widehat V^{clean}_{ik}$ to be small for all $i\in[q]$. We will show that  whenever $k\notin D(V)$,  
$|\widehat V^{clean}_{ik}|$ is uniformly bounded by  $C_1\epsilon_n$ for $i\in[q]$ and $k\in[r]$ with high probability. Here $C_1>0$ is a constant. Second, when  $(\bi)_k\neq 0$, we will show that $ \max_{i\in[r]}|\widehat V^{clean}_{ik}|$ can not be too small. In fact, we will show that  \begin{equation}\label{inroadmap:Th1}
    \max_{i\in[r]}\abs{\widehat V^{clean}_{ik}}>C_2\max_{i\in[r]}\abs{\Lambda_i(\bi)_k}-C_1\epsilon_n,\quad k\in[r]
\end{equation}
for some $C_2>0$ with high probability in this case. The lower bound in the above inequality is bounded below by $C_2\Sig_y-C_1\epsilon_n$.  Thus, if the minimal signal strength $\Sig_y$ is bounded below by a large enough multiple of $\epsilon_n$, then the lower bound $C_2\Sig_y-C_1\epsilon_n$ will be larger than the upper bound $C_1\e_n$ in the $k\notin D(V)$ case. Therefore, in this scenario, 
we can choose $C>0$  so that
\[C_1\epsilon_n < C\epsilon_n < C_2\Sig_y-C_1\epsilon_n.\]
If we set \texttt{cut}$=C\epsilon_n$, then the above 
inequality leads to 
\[\sup_{i\notin D(V)}\abs{\widehat V^{clean}_{ik}}< \texttt{cut}< \inf_{i\in D(V)}\abs{\widehat V^{clean}_{ik}}.\]
These $C_1$ and $C_2$ are behind the  constant $C_{\mathcal B}'$ in \eqref{intheorem: statement: lower bound on signal} and our choice of $\theta_n$.}

\textcolor{black}{
Thus the key step in the proof of Theorem \ref{thm: support recovery: sub-gaussians} is analyzing the bias of $\widehat V^{clean}_{ik}$, which hinges on the following bias decomposition:
\begin{align}\label{inroadmap: bias decomposition}
\MoveEqLeft |\widehat V^{clean}_{ik}-\Lambda_i(\bi)_k|\leq  \underbrace{  |e_k^T(\hf-\Phi_0)\hSyx \hai|}_{T_1(i,k)}\nn\\
&\ + \underbrace{  |e_k^T\Phi_0(\hSyx-\Syx) \hai|}_{T_2(i,k)} + \underbrace{ |e_k^T\Phi_0\Syx(\hai-\ai)|}_{T_3(i,k)}.
\end{align}
Note that the term $T_1(i,k)$ corresponds to the bias in estimating $\Phi_0$. Similarly, the error terms $T_2(i,k)$ and $T_3(i,k)$ incur due to the bias in estimating  $\Syx$ and $u_i$, respectively. The main contributing term in  the upper bound in \eqref{inroadmap: bias decomposition} is $T_1(i,k)$. 
 One can use the consistency property of $\hf$ to show that $T_1(i,k) $ is  of the order $O_p(\epsilon_n)$. Since $\hf$ has different rates and modes of convergence in  cases A, B, and C, $T_1(i,k) $ has different orders in cases A, B, and C, which explains why $\epsilon_n$ is of different order in these cases.}

\textcolor{black}{
The term  $T_2(i,k)$ is  much smaller -- it is of the order  $(s(\Sx^{-1})\log (p+q))/n)^{1/2}$. The proof bases on the fact that the $l_\infty$ error of estimating  $\Sxy$ by $\hSxy$ is of the order $(\log (p+q)/n)^{1/2}$ for 
subgaussian $X$ and $Y$. The error term $T_3(i,k) $ is exactly zero for $i\notin D(V)$, and hence does not contribute. Thus only $T_1(i,k)$ and $T_2(i,k)$ contribute to the bias of $\widehat V^{clean}_{ik}$ for $i\notin D(V)$, which is therefore bounded by  $C_1\epsilon_n$ for some $C_1>0$ with high probability in this case.
The term $T_3(i,k)$ does contribute to the bias  of $\widehat V^{clean}_{ik}$ for $i\in D(V)$, however, and it is of the order $\sqrt{r}\max_{j\in[r]}|(v_j)_k|\texttt{Err}$ in this case. Because $\texttt{Err}$ is small by Condition \ref{Assumption on estimators}, we can show that  $T_3(i,k) $ is  smaller than $\max_{i\in[r]}\Lambda_i |(v_i)_k|$, which eventually leads to the relation in \eqref{inroadmap:Th1}, thus completing the proof. }\textcolor{black}{We have already mentioned that \RecoverSupp\ is analogous to the cleaning step in sparse PCA. Therefore the proof of Theorem \ref{thm: support recovery: sub-gaussians} has similarities with some analogous results in sparse PCA. See for example Theorem 3 of \cite{deshpande2014}, which proves the consistency of a ``cleaned" estimator of the joint support of the spiked principal components. However, the proof in the CCA case is a bit more involved because of the presence of $\Sy^{-1}$, which needs to be estimated for the  cleaning step. Different estimators of $\Sy^{-1}$ can have different rates of convergence, which leads to the different types of the estimators. This ultimately leads to different requirements on the order of the threshold \texttt{cut} and the minimal signal strength $\Sig_y$. }

Next we will discuss the implications of Theorem~\ref{thm: support recovery: sub-gaussians}, but before getting into that detail, we will make two important remarks. 

\begin{remark}
Although the estimation of the high dimensional precision matrix $\Sy^{-1}$ is potentially complicated,  it is often unavoidable owing to the inherent subtlety of the CCA framework  due to the presence of high dimensional nuisance parameters $\Sx$ and $\Sy$.   \cite{gao2013} also used precision matrix estimator for partial recovery of the support. 
In case of sparse CCA, to the best of our knowledge, there does not exist an  algorithm that can recover the support, partially or completely, without estimating the precision matrix.
However, our requirements on $\hf$ are not strict in that many common precision matrix estimators, e.g., the nodewise Lasso \cite[Theorem 2.4]{vandegeer2014}, the thresholding estimator  \cite[Theorem 1 and Section 2.3]{bickel2008}, and the CLIME estimator  \cite[Theorem 6]{cai2011} exhibit the  decay rate of type A and B  under standard sparsity assumptions on $\Sy^{-1}$. We will not get into the detail of the  sparsity requirements on $\Sy^{-1}$ because they are unrelated to the sparsity of   $U$ or $V$, and hence are  irrelevant to the primary goal of the current paper. 
\end{remark}

 \begin{remark}
In the easy regime $s_y\lesssim  \sqrt{n/(\log(p+q)}$, polynomial time estimators satisfying Condition~\ref{Assumption on estimators} are already available, e.g., COLAR \cite[Theorem 4.2]{gao2017} or SCCA \cite[Condition C4]{mai}.  
Thus it is easily seen that polynomial time support recovery is possible in the easy regime provided   \eqref{intheorem: statement: lower bound on signal} is satisfied. 
\end{remark}

The  implications of Theorem~\ref{thm: support recovery: sub-gaussians} in context of the sparsity requirements on $D(U)$ and $D(V)$ for full support recovery are somewhat implicit through the assumptions and conditions. However, the restriction on the sparsity is indirectly imposed by two different sources -- which we elaborate on now. To keep the interpretations simple, throughout the following discussion, we  assume that (a) $r=O(n/\log q)$, (b) $p$ and $q$ are of the same order, and (c) $s_x$ and $s_y$ are also of the same order. Note that (a)  implies $\xi_n=O(1)$ for a type B estimator of $\hf$.  Since we separate the task of estimating the nuisance parameter $\Sy^{-1}$ from the support recovery of $V$, we also assume that $\sSy=O(1)$, which implies $\xi_n=O(1)$ for a type A estimator of $\hf$. The assumption  $\sSy=O(1)$, combined with (a), reduces the minimal signal strength condition \eqref{intheorem: statement: lower bound on signal}  in Theorem \ref{thm: support recovery: sub-gaussians} to $\Sig_y\geq C_\B \sqrt{\log(p+q)/n}$.
 
 In lieu of the discussion above, the first source of sparsity restriction is the minimal signal strength condition \eqref{intheorem: statement: lower bound on signal} on $\Sig_y$.
\textcolor{black}{To see this, first note that  \[1=v_i^T\Sy v_i\geq \Lambda_{min}(\Sy)\|v_i\|_2^2\]
where $i\in[r]$.
Since $\Lambda_{min}(\Sy)\geq \B^{-1}$,
\[\Lambda_{min}(\Sy)\|v_i\|_2^2\geq \|v_i\|_2^2/\B\geq \Sig_y^2 s_y/\B,\]
implying  $\Sig_y\leq\sqrt{\B} s_y^{-1/2}$.}
Therefore, implicit in Theorem~\ref{thm: support recovery: sub-gaussians} lies the condition
\begin{align}\label{sparsity requirement: thm 1}
 s_y\leq \frac{C_\B^2 n}{\log(p+q)}, 
\end{align}
which is enforced by the minimal signal strength requirement \eqref{intheorem: statement: lower bound on signal}.   
Thus  Theorem~\ref{thm: support recovery: sub-gaussians} does not hold for $s_y\gg  n/\log(p+q)$ even when $\sSy$ and $r$ are small. This regime requires some attention because in case of sparse PCA \cite{amini2009} and linear regression \cite{wainwright2009}, support recovery at $s\gg n/\log (p-s)$ \footnote{here and later, we will use $s$ to generically denote the sparsity of relevant parameter vectors in parallel problems like sparse PCA or sparse linear regression.} is proven to be information theoretically impossible. However, although a parallel result can be intuited to hold for CCA, the details of the nuances of SCCA support recovery in this regime is yet to be explored.  Therefore, the sparsity requirement in \eqref{sparsity requirement: thm 1}  raises the question whether support recovery for CCA is at all possible when  $s_y\gg n/\log(p+q)$, even if $\Sx$ and $\Sy$ is known.
\begin{question}
\label{Q1}
Does there exist any decoder $\widehat D$ such that $\sup_{\PP\in\mod}\PP(\widehat D(V)\neq D(V))\to 0$ when $s_y\gg n/\log(q-s_y)$?
\end{question}

A related question is whether the minimal signal strength requirement \eqref{intheorem: statement: lower bound on signal} is  necessary.  To the best of our knowledge, there is no formal study on the information theoretic limit of the minimal signal strength  even in context of the sparse PCA support recovery. Indeed, as we noted before,  the detailed analyses of support recovery for SPCA provided in \cite{amini2009, deshpande2014}, and \cite{ boaz2015} are all based on the nonzero principal component elements being of the order $O(1/\sqrt s)$. Finally, although this question is not directly related to the sparsity conditions, it indeed probes the sharpness of the results in Theorem~\ref{thm: support recovery: sub-gaussians}.

\begin{question}\label{Question: minimal signal strength}
What is the minimum signal strength required for the recovery of $D(V)$?
\end{question}

We will  discuss Question \ref{Q1} and Question~\ref{Question: minimal signal strength}  at greater length in Section \ref{sec: IT lower bound}. In particular, Theorem \ref{thm: lower bound}(A) shows that there exists $C>0$ so that  support recovery at $s_y\geq  C\B^{-2}n/\log(q-s_y)$ is indeed information theoretically intractable.  On the other hand, in Theorem \ref{thm: lower bound}(B), we show that the minimal signal strength has to be of the order  $\B\sqrt{\log (q-s_y)/n}$ for full recovery of $D(V)$. Thus when $p\asymp q$ ,   \eqref{intheorem: statement: lower bound on signal} is indeed necessary from information theoretic perspectives.

\textcolor{black}{The second source of restriction  on the sparsity} lies in  Condition~\ref{Assumption on estimators}. Condition~\ref{Assumption on estimators} is a $l_2$-consistency condition, which has  sparsity  requirement itself owing the inherent hardness in the estimation of $U$.  Indeed, Theorem 3.3 of \cite{gao2017} entails that it is impossible to estimate the canonical directions $\ai$'s consistently  if $s_x>Cn/(r+\log (ep/s_x))$ for some large $C>0$. Hence, Condition~\ref{Assumption on estimators} indirectly imposes the restriction $s_x\lesssim n/\max\{\log(p/s_x),r\}$. However, when $s_x\asymp s_y$, $p\asymp q$, and $r=O(1)$, the above restriction is already absorbed into the condition  $s_y\lesssim n/\log(q-s_y)$ elicited in the last paragraph. In fact, there exist consistent estimators of $U$   whenever $s_x\lesssim n/\max\{\log(p/s_x),r\}$ and $s_y\lesssim n/\max\{\log(q/s_y),r\}$ (see  \cite{gao2015} or Section 3 of \cite{gao2017}).  Therefore, in the latter regime,  \RecoverSupp\  coupled with the above-mentioned estimators succeeds. In view of the above, it might be tempting to think  that  Condition~\ref{Assumption on estimators} does not impose significant additional  restrictions.
The restriction due to Condition~\ref{Assumption on estimators}, however, is rather subtle and  manifests itself through computational challenges. Note that  when support recovery is information theoretically possible, the computational hardness of recovery by  \RecoverSupp\  \textcolor{black}{will be at least as much as that of the estimation  of $U$.} 
Indeed, the estimators of $U$ which work in  the regime $s_x\asymp n/\log(p/s_x)$, $s_y\asymp n/\log(q/s_y)$ are not adaptive of the sparsity, and they require a search over exponentially many sets of size $s_x$ and $s_y$. 
Furthermore,  under $\mod$, all  polynomial time  consistent estimators of $U$ in the literature, e.g., COLAR \cite[Theorem 4.2]{gao2017} or SCCA \cite[Condition C4]{mai}, require $s_x$, $s_y$ to be of the order $\sqrt{n/\log (p+q)}$. In fact, \cite{gao2017} indicates that estimation of $U$ or $V$ for  sparsity of larger order is  NP hard. 

The above raises the question whether \RecoverSupp\ (or any method as such) can succeed  at polynomial time  when $ \sqrt{n/\log (p+q)}\ll s_x,s_y\lesssim n/\log(p+q)$. 
We turn to the landscape of sparse PCA for intuition. Indeed, in case of  sparse PCA, different scenarios are observed in  the  regime $s\lesssim n/\log p$,  depending on whether $\sqn\ll s\lesssim n/\log p$, or $s\lesssim \sqrt{n}$ (we recall that for SPCA we denote the sparsity of the leading principal component direction generically through $s$). 
We focus on the sub-regime $\sqn\ll s\lesssim n/\log p$ first. In  this case, both estimation and support recovery for sparse PCA  are \textcolor{black}{conjectured to be NP hard}, which means no polynomial time method succeeds; see Section~\ref{sec: low deg}  for  more details. 
The above hints  that   the regime $s_x,s_y\gg \sqrt n$ is  NP hard for sparse CCA as well.

\begin{question}
\label{Question 2}
Is there any polynomial time method that  can recover the support $D(V)$ when $s_x,s_y\gg \sqn$?
\end{question}
We dedicate Section \ref{sec: low deg} to  answering this question.  Subject to the recent advances in the low degree polynomial conjecture, we establish computational hardness of the regime $s_x,s_y\gg \sqn$   (up to a logarithmic factor gap) subject to $n\lesssim p,q$. Our results are  consistent with \cite{gao2017}'s findings in the estimation case and cover a broader regime; see Remark~\ref{remark: chao gao comp barrier} for a comparison.

When the sparsity is of the order $\sqrt n$ and $p\asymp n$, however, polynomial time  support recovery and estimation are  possible for the sparse PCA case. \cite{deshpande2014} showed that a
 co-ordinate thresholding type spectral algorithm works in this regime. 
 Thus the following question is immediate.
 
 \begin{question}
\label{Question 3}
Is there any polynomial time method that  can recover the support $D(V)$ when $s_x,s_y\in[\sqrt{n/\log (p+q)},\sqn]$?
\end{question}
 
  We give an affirmative answer to Question~\ref{Question 3} in Section~\ref{sec: CT}, which is in parallel with the observations for the sparse PCA. In fact,  Corollary~\ref{cor: CT support recovery} shows that when $\Sx$ and $\Sy$ are known, $p+q\asymp n$, and  $s_x,s_y\lesssim \sqn$,  estimation  is possible in polynomial time. Since estimation is possible,  \RecoverSupp\  suffices for polynomial time support recovery in this regime, where $\sqn$ is well below the information theoretic limit of $n/\log (p+q)$. The main tool used in Section~\ref{sec: CT} is co-ordinate thresholding,
which is originally a method for  high dimensional matrix estimation \cite{bickel2008}, and apparently has nothing to do with estimation of canonical directions. However, under our setup, if the covariance matrix is consistently estimated in operator norm, 
by Wedin's Sin $\theta$ Theorem \cite{yu2015}, an SVD  is enough to get a consistent estimator of $U$ and $V$ suitable for further precise analysis.
 
 \begin{remark}
\RecoverSupp\  uses sample splitting, which can reduce the efficiency. One can swap between the samples and compute two estimators of the supports. One can easily show that both the intersection and the union of the resulting supports enjoy the asymptotic guarantees of Theorem~\ref{thm: support recovery: sub-gaussians}.
 \end{remark}

This section can be best summarized by Figure~\ref{fig: sparsity big picture}, which gives the information theoretic and computational landscape of  sparse CCA analysis in terms of the sparsity. \textcolor{black}{In other words, Figure~\ref{fig: sparsity big picture} gives the phase transition plot for SCCA support recovery with respect to sparsity.} It can be seen that our contributions (colored in red) complete the picture, which was initiated by \cite{gao2017}.

 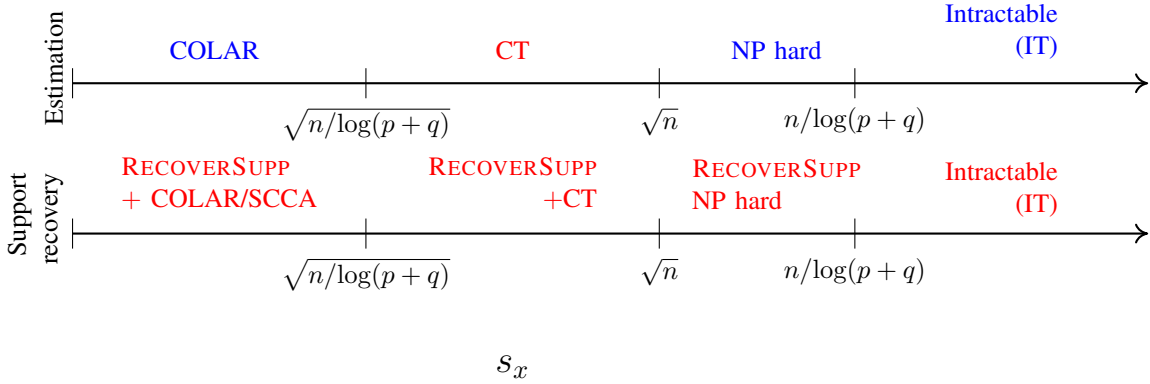
\begin{figure*}
    \begin{tikzpicture}[xscale=1.3]
\draw [ -{>[scale=2.5,
          length=2,
          width=3]}, thick]  (0,0) -- (11,0);
\draw (0,-.2) -- (0, .2);
\draw (3,-.2) -- (3, .2);
\draw (6,-.2) -- (6, .2);
\draw (8, -.2) -- (8, .2);
\node[align=left, above, text=blue] at (1.5,.2)%
    {COLAR };
    \node[align=right, below] at (3,-.2)%
    { $\sqrt{n/\log (p+q)}$};
\node[below] at (6,-.2)%
    {$\sqrt{n}$};
     \node[align=right, above, text=red] at (4.5, .2)%
     {CT};
\node[align=left, above, text=blue] at (7.2,.2)%
    {NP hard};
     \node[align=right, above, rotate=90] at (0,0 .2)%
     {Estimation};
   \node[below] at (8, -.2)%
    {$n/\log (p+q)$};    
     \node[align=right, above, text=blue] at (9.5, .2)%
     {Intractable\\ (IT)}; 
\draw [ -{>[scale=2.5,
          length=2,
          width=3]}, thick]  (0,-2) -- (11,-2);
\draw (0,-2.2) -- (0, -1.8);
\draw (3,-2.2) -- (3, -1.8);
\draw (6,-2.2) -- (6, -1.8);
\draw (8, -2.2) -- (8, -1.8);
\node[align=left, above, text=red] at (1.5,-1.8)%
    {\RecoverSupp\ \\$+$ COLAR/SCCA};
    \node[align=right, below] at (3,-2.2)%
    { $\sqrt{n/\log (p+q)}$};
\node[below] at (6,-2.2)%
    {$\sqrt{n}$};
      
    \node[align=right, below] at (4.5,-3.5)%
    {\Large $s_x$};
     \node[align=right, above, text=red] at (4.5, -1.8)%
     {\RecoverSupp\ \\$+$CT};
\node[align=left, above, text=red] at (7.2,-1.8)%
    {\RecoverSupp\ \\NP hard};
     \node[align=right, above, rotate=90] at (0,-1.8)%
     {Support\\ recovery};
    \node[below] at (8, -2.2)%
    {$n/\log (p+q)$};    
     \node[align=right, above, text=red] at (9.5, -1.9)%
     {Intractable\\ (IT)};   
\end{tikzpicture}
\caption{\textcolor{black}{Phase transition plots for SCCA estimation and support recovery problems with respect to sparsity.} 
We have taken $s_x=s_y$ here. COLAR corresponds to the estimation method of \cite{gao2017}. Our contributions are colored in red. See \cite{gao2017} for more details on the regions colored in blue. }
\label{fig: sparsity big picture}
 \end{figure*}

\subsection{\bf Information Theoretic Lower Bounds: Answers to Question~\ref{Q1} and \ref{Question: minimal signal strength}}
\label{sec: IT lower bound}

     


Theorem~\ref{thm: lower bound} establishes the information theoretic limits on the sparsity levels  $s_x$, $s_y$, and the signal strengths $\Sig_x$ and $\Sig_y$.  The proof of Theorem~\ref{thm: lower bound} is deferred to Appendix~\ref{sec: Proof of IT limit}.
\begin{theorem}\label{thm: lower bound}
  Suppose $\widehat D(U)$ and $\widehat D(V)$ are estimators of $D(U)$ and $D(V)$, respectively.  Let $s_x$, $s_y>1$, and $p-s_x,q-s_y>16$. Then the following assertions hold:
  \begin{itemize}
      \item[A.] If $s_x>16n/\{(\B^2-1)\log(p-s_x)\}$,
      then 
     \[\inf_{\widehat D}\sup_{\PP\in\mP(r,s_x,s_y,\B)}\PP\lb \widehat D(U)\neq D(U)\rb >1/2.\] 
    On the other hand, if $s_y>16n/\{(\B^2-1)\log(q-s_y)\}$, then
 \[\inf_{\widehat D}\sup_{\PP\in\mP(r,s_x,s_y,\B)}\PP\lb  \widehat D(V)\neq D(V)\rb >1/2.\]
 \item[B.] Let  $\modsig$  be the class of distributions $\PP\in\mod$ satisfying \\
 $\Sig^2_x\leq (\B^2-1)(\log(p-s_x))/(8n)$. Then
\[\inf_{\widehat D}\sup_{\PP\in\modsig}\PP\lb \widehat D(U)\neq D(U)\rb >1/2.\]
On the other hand, if \\ $\Sig^2_y\leq(\B^2-1)(\log(q-s_y))/(8n)$, then 
  \[\inf_{\widehat D}\sup_{\PP\in\modsig}\PP\lb  \widehat D(V)\neq D(V)\rb >1/2.\]
  \end{itemize}
In both cases, the infimum is over all possible decoders $\widehat D(U)$ and $\widehat D(V)$.
\end{theorem}

First, we discuss the implications of part A of Theorem~\ref{thm: lower bound}. This part entails that for full support recovery of $V$, the minimum sample size requirement is of the order $s_y\log(q-s_y)$. This requirement is consistent with the traditional  lower bound on $n$  in context of  support recovery  for sparse PCA \cite[Theorem 3]{amini2009}  and $L_1$ regression \cite[Corollary 1]{wainwright2009}. However, when $r=O(1)$, the sample size requirement for estimation of $V$ is slightly relaxed, that is, $n\gg s_y\log(q/s_y)$   \cite[Theorem 3.2]{gao2017}. Therefore, from information theoretic point of view, the task of full support recovery appears to be slightly harder than the task of estimation. The scenario for partial support recovery might be different and we do not pursue it here.
Moreover, as mentioned earlier,  in the regime $s_y\lesssim C_\B n/\log(p+q)$, \RecoverSupp\  works with \cite{gao2017}'s (see Section 3 therein) estimator of $U$. Thus part A of Theorem~\ref{thm: lower bound} implies that $n/\log(p+q)$ is the information theoretic upper bound on the sparsity for the full support recovery of sparse CCA.

Part B of Theorem~\ref{thm: lower bound} implies that it is not possible to push the minimum signal strength  below the level $O(\sqrt{\log(q-s_y)/n})$. Thus the minimal signal strength requirement \eqref{intheorem: statement: lower bound on signal} by Theorem~\ref{thm: support recovery: sub-gaussians} is indeed minimal up to a factor of $\xi_n\sqrt{\sSy}$. The last statement can be refined further. To that end, we remind the readers that for a good estimator of $\Sy^{-1}$, i.e.,  a type B  estimator, $\xi_n=O(1)$   if $r=O(n/\log q)$. However, the latter always holds  if support recovery is at all possible, because in that case
 $s_y\lesssim n/\log (p+q)$, and elementary linear algebra gives $s_y\geq r$. Thus, it is fair to say that, provided a good estimator of $\Sy^{-1}$, the requirement \eqref{intheorem: statement: lower bound on signal} is minimal up to a factor of $\sqrt{\sSy}$. Indeed, this implies that for banded inverses with finite band-width our results are rate optimal.

It is further worth comparing this part of the result to the SPCA literature. In the  SPCA support recovery literature, generally, the lower bound on the  signal strength  is  depicted in terms of the sparsity $s$, and usually a signal strength of order $O(1/\sqrt{s})$ is postulated (cf. \cite{deshpande2014,amini2009,boaz2015}). 
Using our proof strategies, it can be easily shown that for SPCA, the analogous lower bound on the signal strength would be $\sqrt{\log(p-s)/n}$. The latter is generally much smaller than $1/\sqrt s$ and only  when $s\asymp n/\log(p)$, the requirement of $1/\sqrt s$ is close to the lower bound. Thus, in the regime $s\lesssim \sqrt{n/\log p}$,  the  lower bound should rather be of the order $O(1/s)$. Therefore the minimum signal strength requirement of $O(1/\sqrt s)$ typically assumed  in SPCA literature seems  larger than necessary. 

\textcolor{black}{
\subsubsection*{Main idea behind the proof of Theorem~\ref{thm: lower bound}} 
The main device used in this proof is Fano's inequality \cite{yatracos1988}.
Note that for any $\mathcal C\subset \mP(r,s_x,s_y,\B)$,
\begin{equation}
\label{inroadmap: Fano}
\inf_{\widehat D_\alpha}\sup_{\mathbb P\in  \mathcal C}\mathbb P\slb\widehat D_{\alpha}\neq D(\alpha)\srb<\inf_{\widehat D}\sup_{\mathbb P\in  \mP(r,s_x,s_y,\B)}\mathbb P\lb \widehat D_{\alpha}\neq D(\alpha)\rb  . 
\end{equation}
Therefore it suffices to show that the left hand side in the above inequality is bounded away from $1/2$ for some carefully chosen $\mathcal C$.  If $\mathcal C$ is finite, we can lower bound the left hand side of \eqref{inroadmap: Fano} using  Fano's inequality \cite{yatracos1988}, which yields
\begin{align}\label{inroadmap:fano}
     \MoveEqLeft \inf_{\widehat{D}_\alpha}  \sup_{\PP\in  \mathcal C}\PP\slb\widehat D_{\alpha}\neq D(\alpha)\srb\geq  1-\dfrac{\frac{\sum_{\PP_1,\PP_2\in\mathcal C}KL(\PP_1^n|\PP_2^n)}{|\mathcal C|^2}+\log 2}{\log(|\mathcal C|-1)},  
\end{align}
Thus the main task is to choose $\mathcal C$ in a way so that the right hand side (RHS) of \eqref{inroadmap:fano} is large. 
We will choose $ \mathcal C$  so that $X$ and $Y$ are jointly Gaussian. In particular, $X\sim N_p(0,I_p)$, $Y\sim N_q(0,I_q)$, and $\Sigma_{xy}=\rho\bk\alpha^T$ where $\bk\in\mathbb S^{q-1}$ and $\rho\in(0,1)$ are fixed, and $\alpha$ is allowed to vary in a set $\mE\subset\mathbb S^{p-1}$. In this model, $r=1$, $\rho$ is the canonical correlation, and $\alpha$ and $\beta_0$ are the left and right canonical covariates, respectively. Also, $\PP$ varies across $\mathcal C$ as $\alpha$ varies across $\mE$. Moreover, $|\mathcal C|=|\mE|$.   Our main task boils down to choosing $\mE$ carefully.
 }

\textcolor{black}{
The idea behind choosing  $\mathcal E$ is as follows. For any decoder, i.e., an estimator of the support, the chance of making error increases when  $|\mE|$ is large. This can also be seen noting that the right hand side of \eqref{inroadmap:fano} increases as $|\mE|=|\mathcal C|$ increases. However, even if we prefer a larger $\mE$, we need to ensure that  the KL divergence between the  distributions in the resulting $\mathcal C$ is small. The reason is that,  for a large $\mE$, the right hand side of \eqref{inroadmap:fano} can be small  unless the KL divergence between the corresponding distributions in $\mathcal C$ is small. In other words, any decoder will face a challenge detecting the true support of $\alpha$ when there are many distributions to choose from, and these distributions are also close to each other in KL distance.  
}

\textcolor{black}{
For part A of Theorem~\ref{thm: lower bound}, we choose $\mE$ in the following way. Letting \[\alk=(\underbrace{1/\sqrt s_x,\ldots,1/\sqrt s_x}_{s_x\text{ many }},\underbrace{0,\ldots,0}_{p-s_x\text{ many }}),\] 
we let $\mE$ be the class of $\alpha$'s which are obtained by  replacing one of the $1/\sqrt s_x$'s in $\alpha_0$ by $0$, and  one of the zero's in $\alk$ by $1/\sqrt s_x$. 
 A typical  $\alpha$ obtained this way looks like 
   \[\alpha=\slb \underbrace{1/\sqrt s_x,\ldots,\textcolor{red}{\mathbf 0},\ldots 1/\sqrt s_x}_{s_x\text{ many }},\underbrace{0,\ldots,\textcolor{red}{\mathbf{1/\sqrt{s_x}}},\ldots,0}_{p-s_x\text{ many }}\srb.\]
   In this case, it turns out that $|\mE|=s_x(p-s_x)$. Under the conditions of part A of \ref{thm: lower bound}, we can show that the RHS of \eqref{inroadmap:fano} is bounded below by 1/2 for this $\mE$. \textcolor{black}{The proof of part A is similar to its PCA analogue, which is Theorem 3 of \cite{amini2009}. The latter theorem is also based on Fano's lemma and uses a similar construction for $\mE$. However, there is no PCA analogue of part B.}
   For part B of Theorem \ref{thm: lower bound}, we let $\mE$ be the class of all $\alpha$'s so that
   \[\alpha=\slb \underbrace{ b,\ldots, b}_{s_x-1\text{ many }},\underbrace{0,\ldots,0,\textcolor{red}{z},0,\ldots,0}_{p-s_x+1\text{ many }}\srb.\]
   where
   \[z=\sqrt{\frac{1-\rho^2}{4n\rho^2}}\log\lb \frac{p-s_x}{4}\rb\]
   can take any position out of the $p-s_x+1$ positions. Clearly, $|\mE|=p-s_x+1$. It can be shown that  the RHS of \eqref{inroadmap:fano} is bounded below by $1/2$ in this case as well.
}

\subsection{\bf Computational Limits and Low Degree Polynomials: Answer to Question~\ref{Question 2}}
 \label{sec: low deg}

  We have so far explored the information theoretic upper and lower bounds for recovering the true support of leading canonical correlation directions. However, as indicated in the discussion preceding Question~\ref{Question 2}, the statistically optimal procedures in the regime where $\sqrt{n}\lesssim s_x,s_y\lesssim n/\log{(p+q)}$ are computationally intensive and is of exponential complexity (as a function of $p,q$).
 \textcolor{black}{ In particular,
  \cite{gao2017} have already showed that  \textit{when $s_x$ and $s_y$ belong to parts of this regime}, estimation of the canonical correlates  is computationally hard, subject to a  computational complexity based  \textit{Planted Clique Conjecture}. For the case of support recovery, the SPCA has been explored in detail and the corresponding  computational hardness has been established in analogous regimes -- see, e.g., \cite{amini2009, deshpande2014}, and \cite{boaz2015} for details. A similar phenomenon of computational hardness is observed in case of  \textit{SPCA spike detection} problem \cite{berthet2013}.
  In light of the above, it is natural to believe that the SCCA support recovery is also computationally hard in the regime  $\sqrt{n}\lesssim s_x,s_y\lesssim n/\log{(p+q)}$, and, as a result, yields a statistical-computational gap.} 
 Although several paths exist to provide evidence towards such gaps \footnote{e.g., Planted Clique Conjecture \cite{berthet2013, gao2017, brennan2018reducibility}, Statistical Query based lower bounds \cite{kearns1998efficient, feldman2012computational, brennan2020statistical, dudeja2021statistical}, and Overlap Gap Property based analysis \cite{gamarnik2017, susenda2019, arous2020}.}, the recent developments using ``Predictions from Low Degree Polynomials" \cite{hopkins2017, hopkins2018, kunisky2019} is particularly appealing due its simplicity in exposition. In order to show computationally hardness of the SCCA support recovery problem in the  $s_x,s_y\in(\sqrt{n},n/\log(p+q))$ regime, we shall resort to this very style of ideas, which has so far been applied successfully to explore  statistical-computational gaps under sparse PCA \cite{ding2019}, Stochastic Block Models, and  tensor PCA \cite{hopkins2018}, among others. This will allow us to explore the computational hardness of the problem in the entire regime where
   \begin{equation}\label{eq: low deg: sx and sy def}
   s_x+s_y\gtrsim  (\sqrt  n)(\log n)^c,
   \end{equation}
   \textcolor{black}{compared to the somewhat partial results (see Remark \ref{remark: chao gao comp barrier} for detailed comparison)} in earlier literature.

We divide our discussions to argue the existence of a statistical-computational gap in this regime as follows. Starting with a brief background on the statistical literature on such gaps, we first present a natural reduction of our problem to a suitable hypothesis testing problem in Section \ref{subsec:reduction_to_testing_gap}. Subsequently, in Section \ref{subsec:background_gap} we present the main idea of the ``low degree polynomial conjecture" by appealing to the recent developments in \cite{hopkins2017, hopkins2018}, and \cite{kunisky2019}.  Finally, we present our main result for this regime in Section \ref{subsec:main_result_gap},  thereby providing evidence of the aforementioned gap modulo the Low Degree Polynomial Conjecture presented in Conjecture \ref{conjecture: low deg}.

   \subsubsection{Reduction to Testing Problem:}\label{subsec:reduction_to_testing_gap}
  Denote by $\QQ$ the distribution of a  $N_{p+q}(0,I_{p+q})$ random vector. Therefore $(X,Y)\sim\QQ$ corresponds to the case when $X$ and $Y$ are uncorrelated.
  We first show that there is any scope of support recovery in $\mod$ only if  $\mod$ is distinguishable from $\QQ$, i.e., the test $H_0:(X,Y)\sim\QQ$ vs. $H_1:(X,Y)\sim\PP\in\mod$ has asymptotic zero error.

  To formalize the ideas, suppose we observe i.i.d random vectors $\{X_i,Y_i\}_{i=1}^{n}$ which are distributed either as $\PP$ or $\QQ$. We denote the $n$-fold product measures corresponding to $\PP$ and $\QQ$ by $\PP_n$ and $\QQ_n$, respectively. Note that if $\PP\in\mod$, then $\PP_n\in\mod^n$. We overload notation, and denote the combined sample $\{X_i\}_{i=1}^n$ and $\{Y_i\}_{i=1}^n$ by $\X$ and $\Y$ respectively. 
  In this section, $\X$ and $\Y$ should be  viewed as unordered sets.
  The test $\Phi_n:\RR^{pn+qn}\mapsto\{0,1\}$ for testing the null $H_0:(\X,\Y)\sim\QQ_n$ vs. the alternative $H_1:(\X,\Y)\sim\PP_n$  is said to  strongly distinguish $\PP_n$ and $\QQ_n$  if
  \[\lim_n \QQ_{n}(\Phi_n(\X,\Y)=1)+\lim_n \PP_{n}(\Phi_n(\X,\Y)=0)=0.\]
  The above implies that both the type I error and the type II error of $\Phi_n$ converges to zero as $n\to\infty$.
In case of composite alternative  $H_1:(\X,\Y)\sim\PP_n\in \mod^n$, the test strongly distinguishes $\QQ_n$ from $\mod^n$ if
\begin{align*}
  \MoveEqLeft  \liminf_{n\to\infty}\lbs \QQ_{n}(\Phi_n(\X,\Y)=1)\\
    &\ + \sup_{\PP_n\in \mod^n}\PP_{n}(\Phi_n(\X,\Y)=0)\rbs=0.
\end{align*}
Now we explain how support recovery and  the testing framework are connected.  Suppose there exist decoders which exactly recover $D(U)$ and $D(V)$ under $\mod$ for $\B\geq 0$. Then  the trivial test, which  rejects the null  if either of the estimated supports is non-empty, strongly distinguishes $\QQ_n$ from $\mod^n$. The above  can be coined as the following lemma.
 \begin{lemma}\label{lemma: low deg: support to detection}
  Suppose there exist polynomial time decoders $\hat D_x$ and $\hat D_y$ of $D(U)$ and $D(V)$ so that 
\begin{align}\label{inlemma: low deg: support}
  \MoveEqLeft  \liminf_{n\to\infty}\sup_{\PP_n\in \mod^n} \PP_{n}\lb\hat D_x(\X,\Y)=D(U)\nn\\
    &\ \text{ and } \hat D_y(\X,\Y)=D(V)\rb=1  \end{align}
  Further assume, $\QQ_n(\hat D_x(\X,\Y)=\emptyset)\to 1$, and $\QQ_n(\hat D_y(\X,\Y)=\emptyset)\to 1$.
  Then there exists a polynomial time test which strongly distinguishes $\mod^n$ and $\QQ_n$.
  \end{lemma}
Thus, if a regime does not allow any polynomial time test for distinguishing $\QQ_n$ from $\mod^n$, there can be no polynomial time computable consistent decoder for $D(U)$ and $D(V)$. Therefore, it suffices to show that there is no polynomial time test which distinguishes $\QQ_n$ from $\mod^n$ in the  regime $s_x,s_y\gg\sqn$. To be more explicit, we want to show that  if $s_x,s_y\gg\sqn$, then
\begin{align}\label{low-deg: what show}
  \MoveEqLeft  \liminf_{n\to\infty}\lbs \QQ_{n}(\Phi_n(\X,\Y)=1)\nn\\
    &\ +\sup_{\PP_n\in \mod^n} \PP_{n}(\Phi_n(\X,\Y)=0)\rbs>0
\end{align}
for any $\Phi_n$ that is computable in polynomial time.

The   testing problem under concern is commonly known as the CCA detection problem, owing to its alternative formulation as  $H_0:\Lambda_1=0$ vs. $H_1:\Lambda_1>0$. 
In other words, the test tries to detect if there is any signal  in the data.  Note that, Lemma \ref{lemma: low deg: support to detection} also implies that detection is an easier problem than support recovery in that the former is always possible whenever the latter is feasible. The opposite direction may not be true,  however, since  detection does not reveal much information on the support.

  \subsubsection{Background on the Low-degree Framework} \label{subsec:background_gap}
   We shall provide a brief introduction to the  low-degree polynomial conjecture which forms the basis of our analyses here, and  refer the interested reader to \cite{hopkins2017, hopkins2018}, and \cite{kunisky2019} for in-depth discussions on the topic.  We will apply this method in context of the test $H_0:(\X,\Y)\sim\QQ_n$ vs. $H_1:(\X,\Y)\sim\PP_n$.
   The low-degree method centers around  the likelihood ratio $\LL_n$, which takes the form $\frac{d\PP_n}{d\QQ_n}$ in the above framework. Our key tool here will be the Hermite polynomials, which form a  basis system  of $L_2(\QQ_n)$ \cite{szego1939}.
    Central to the low-degree approach lies the projection of $\LL_n$ onto the subspace (of $L_2(\QQ_n)$) formed by the Hermite polynomials of degree at most $D_n\in\NN$.
    The latter projection, to be denoted by $\LL_n^{\leq D_n}$ from now on, is important because it measures  how well  polynomials of degree  $\leq D_n$ can distinguish $\PP_n$ from $\QQ_n$. In particular,  \begin{equation}\label{eq:L-defn}
\|\LL_n^{\le D_n}\|_{L_2(\QQ_n)} := \max_{f\text{ deg }{\leq D_n}} \frac{\E_{ \PP_n}[f(\X,\Y)]}{\sqrt{\E_{\QQ_n}[f(\X,\Y)^2]}},
\end{equation}
where the maximization is over polynomials $f:\RR^{n(p+q)}\mapsto\RR$ of degree at most $D_n$ \cite{ding2019}.

 The $L_2(\QQ_n)$ norm of the untruncated likelihood ratio $\LL_n$ has long  held an important place in the theory hypothesis testing since $\|\LL_n\|_{L_2(\QQ_n)}=O(1)$ implies   $\PP_n$ and $\QQ_n$ are asymptotically indistinguishable. While the untruncated likelihood ratio $\LL_n$ is  connected to the existence of \emph{any} distinguishing test, degree $D_n$ projections of $\LL_n$ are connected to the existence of polynomial time distinguishing tests. 
 The implications of the above heuristics are made precise by the following  conjecture \cite[Hypothesis 2.1.5]{hopkins2018}.
\begin{conjecture}[Informal]
\label{conjecture: low deg}
Suppose $t:\NN\mapsto\NN$. For ``nice" sequences of distributions $\PP_n$ and $\QQ_n$, if $\|\LL_n^{\le D_n}\|_{L_2(\QQ_n)}=O(1)$ as $n\to\infty$ whenever $D_n\leq t(n)\text{polylog}(n)$, then there is no time-$n^{t(n)}$ test $\Phi_n:\RR^{n(p+q)}\mapsto\{0,1\}$ that strongly distinguishes $\PP_n$ and $\QQ_n$.
\end{conjecture}
Thus Conjecture~\ref{conjecture: low deg} implies that the degree-$D_n$ polynomial $\LL_n^{\leq D_n}$ is a proxy for time-$n^{t(n)}$ algorithms  \cite{kunisky2019}. If we can show that  $\|\LL_n^{\le D_n}\|_{L_2(\QQ_n)}=O(1)$ for a  $D_n$  of the order $(\log n)^{1+\e}$ for some $\e>0$, then the low degree Conjecture says that no polynomial time test can strongly distinguish $\PP_n$ and $\QQ_n$ \cite[Conjecture 1.16]{kunisky2019}.


   Conjecture \ref{conjecture: low deg} is informal in the sense that we do not specify  the ``nice" distributions, which are defined in Section 4.2.4 of \cite{kunisky2019}   (see also Conjecture 2.2.4 of \cite{hopkins2018}). Niceness requires $\PP_n$ to be sufficiently symmetric, which is generally guaranteed by naturally occurring high dimensional problems like ours. The condition of ``niceness" is attributed to eliminate pathological cases where the testing can  be made easier by methods like Gaussian elimination. See \cite{hopkins2018} for more details.
\subsubsection{Main Result} \label{subsec:main_result_gap}
Similar to \cite{ding2019}, we will consider a Bayesian framework. It might not be immediately clear how a Bayesian formulation will fit into the low-degree framework, and  lead to \eqref{low-deg: what show}. However, the connection will be clear soon. 
We put independent Rademacher priors $\pi_x$ and $\pi_y$  on $\alpha$ and $\beta$. We say $\alpha\sim\pi_x$ if $\alpha_1,\ldots,\alpha_p$ are i.i.d., and for each $i\in[p]$,
  \begin{align}\label{prior: radamander}
      \alpha_i=\begin{cases}
      1/\sqrt{s_x} & w.p.\quad  s_x/(2p)\\
      -1/\sqrt{s_x} & w.p.\quad  s_x/(2p)\\
      0 & w.p.\quad  1-s_x/p.
      \end{cases}
  \end{align}
  The Rademacher  prior $\pi_y$ can be defined similarly.  We will denote the product measure $\pi_x\times\pi_y$ by $\pi$.
  Let us define
\begin{align}\label{def: sigma in low deg}
    \Sigma(\alpha,\beta,\rho)=
\begin{bmatrix}
I_p & \rho\al\beta^T\\
\rho\beta\al^T & I_q
\end{bmatrix}, \quad \alpha\in\RR^p,\ \beta\in\RR^q,\ \rho>0.
\end{align}
When $\rho\|\alpha\|_2\|\beta\|_2<1$, $\Sigma(\alpha,\beta,\rho)$ is the covariance matrix corresponding to $X\sim N_p(0,I_p)$ and $Y\sim N_q(0,I_q)$ with covariance $\text{cov}(X,Y)=\rho\alpha\beta^T$.
\textcolor{black}{  Hence, for $\Sigma(\alpha,\beta,\rho)$ to be positive definite, $\|\alpha\|_2\|\beta\|_2< 1/\rho$ is a sufficient condition. The priors $\pi_x$ and $\pi_y$ put positive weight on $\alpha$ and $\beta$ that do not lead to a positive definite $\Sigma(\alpha,\beta,\rho)$, and hence calls for extra care during the low-degree analysis. This subtlety is absent in the sparse PCA analogue \cite{ding2019}.}

Let us define
\begin{equation}\label{def: low deg: PP alpha beta}
    \PP_{\alpha,\beta}=\begin{cases}
    N(0,\Sigma(\alpha,\beta,1/\B)) & \text{when }\|\alpha\|_2\|\beta\|_2< \B\\
    \QQ & \text{o.w.}
    \end{cases}
\end{equation}
 We denote the $n$-fold product measure corresponding to $\PP_{\alpha,\beta}$ by $\PP_{n,\alpha,\beta}$.
If $(\X,\Y)\mid\alpha,\beta\sim \PP_{n,\alpha,\beta}$, then  the marginal density of $(\X,\Y)$ is
 $\E_{\alpha\sim\pi_x,\beta\sim\pi_y}d\PP_{n,\alpha,\beta}$. The following lemma, which is proved in Appendix~\ref{sec: add lemmas: low-deg}, explains how the Bayesian framework is connected to \eqref{low-deg: what show}.
  \begin{lemma}\label{lemma: low deg: bayesian}
  Suppose $\B>2$ and $s_x,s_y\to\infty$. Then
  \begin{align*}
   \MoveEqLeft  \liminf_n\sup_{\PP_n\in \mP_G(r,2s_x,2s_y,\B)^n}\PP_n\slb \Phi_n(\X,\Y)=0\srb\\
     \geq &\ \liminf_n \E_{\pi}\PP_{n,\alpha,\beta}\slb \Phi_n(\X,\Y)=0\srb, 
  \end{align*}
where $\E_\pi$ is the shorthand for $\E_{\alpha\sim\pi_x,\beta\sim\pi_y}$.
\end{lemma}
Note that a similar result holds for $\mod$ as well because $\modG\subset\mod$.
  Lemma~\ref{lemma: low deg: bayesian} implies that to show \eqref{low-deg: what show}, it suffices to show
 that a polynomial time computable $\Phi_n$ fails to strongly distinguish  the marginal distribution of $\X$ and $\Y$ from $\QQ_n$. 
However,  the latter falls within the realms of  the low degree framework because the corresponding likelihood ratio takes the form 
\begin{equation}\label{def: def of LL n}
    \LL_n=\frac{\E_{\alpha\sim\pi_x,\beta\sim\pi_y}d\PP_{n,\alpha,\beta}}{d\QQ_n(\X,\Y)}.
\end{equation}
Using priors on the alternative space is a common trick to convert a composite alternative to a simple alternative, which generally yields more easily to various mathematical tools.
  If we can show that $\|\LL_n^{\leq D_n}\|^2_{L_2(\QQ_n)}=O(1)$ for some $D_n=O(\log n)$, then 
   Conjecture~\ref{conjecture: low deg} would indicate that a $n^{\tilde\Theta(D_n)}$-time computable $\Phi_n$ fails to distinguish the distribution of  $\E_{\alpha\sim\pi_x,\beta\sim\pi_y}d\PP_{n,\alpha,\beta}$ from $\QQ_n$. Theorem \ref{thm: low deg} accomplishes the above under some additional conditions on $p$, $q$, and $n$, which we will discuss shortly. Theorem~\ref{thm: low deg} is proved in Appendix~\ref{sec; proof of low deg}.
  \begin{theorem}\label{thm: low deg}
 Suppose $D_n\leq \min(\sqrt p,\sqrt q,n)$, 
  \begin{equation}\label{condition: low deg: spasity}
   s_x,s_y\geq \sqrt{enD_n}/\B  \quad\text{and}\quad p,q\geq 3en/\B^2.
  \end{equation}
  Then $\|\LL_n^{\leq D_n}\|^2_{L_2(\QQ_n)}$ is $O(1)$ where $\LL_n$ is as defined in \eqref{def: def of LL n}.
  \end{theorem}
  The following Corollary results from combining Lemma~\ref{lemma: low deg: bayesian} with Theorem~\ref{thm: low deg}.
  \begin{corollary}\label{cor: low deg corollary}
    Suppose 
    \begin{equation}\label{condition: low deg: spasity 2}
        s_x,s_y\geq 2\sqrt{enD_n}/\B  \quad\text{and}\quad p,q\geq 3en/\B^2.
    \end{equation}
     If Conjecture~\ref{conjecture: low deg} is true, then for $D_n\leq \min(\sqrt p,\sqrt q,n)$, there is no time-$n^{\tilde\Theta(D_n)}$ test that strongly distinguishes $\modG$ and $\QQ_n$.
  \end{corollary}
  Corollary~\ref{cor: low deg corollary} conjectures that polynomial time algorithms can not strongly distinguish  $\modG^n$ and $\QQ_n$  provided $s_x,s_y$, $p$, and $q$ satisfy \eqref{condition: low deg: spasity 2}. Therefore under \eqref{condition: low deg: spasity 2}, Lemma~\ref{lemma: low deg: support to detection}  conjectures  support recovery  to be NP hard.
  
    Now we discuss a bit on  condition  \eqref{condition: low deg: spasity 2}.
    The first constraint in \eqref{condition: low deg: spasity 2} is expected because it ensures $s_x,s_y\gg\sqn$, which indicates that the sparsity is in the hard regime. We need to explain a bit on why the other constraint $p,q>3en/\B^2$ is needed. If $n\gg p,q$, the sample canonical correlations are consistent, and therefore  strong separation is possible in polynomial time without any restriction on the sparsity  \cite{bao2019, ma2021}. Even if $p/n\to c_1\in(0,1)$ and $q/n\to c_2\in(0,1)$, then also strong separation is possible in model \ref{def: sigma in low deg}  provided the canonical correlation $\rho$ is larger than some threshold depending on $c_1$ and $c_2$ \cite{bao2019}. The restriction $p,q>3en/\B^2$ ensures that the problem is hard enough so that the vanilla CCA does not lead to successful detection. The constant $3e$ is not sharp and possibly can be improved. 
    The necessity of the condition $p,q\gtrsim n/\B^2$ is unknown for support recovery, however. Since support recovery is a harder problem than detection, in the hard regime, polynomial time support recovery algorithms may  fail at a weaker condition on $n$, $p$, and $q$.

  
   \begin{remark}
\label{remark: chao gao comp barrier}
[Comparison with previous work:]
As mentioned earlier, \cite{gao2017} was the first to discover the existence of computational gap in context of sparse CCA.  In their seminal work, \cite{gao2017} established the computational hardness of CCA estimation problem at a particular subregime of $s_x,s_y\gg \sqrt{n}/(\B \sqrt{\log(p+q)})$ provided  $\B\to\infty$ is allowed. In view of the above, it was hinted that sparse CCA becomes computationally hard when $ s_x,s_y\gg \sqrt{n}/(\B \sqrt{\log(p+q)})$. However, when $\B$ is bounded, the entire regime $ s_x,s_y\gg \sqrt{n}/(\B \sqrt{\log(p+q)})$ is probably not computationally hard. In Section \ref{sec: CT}, we show that if $p+q\asymp n$, then both polynomial time estimation and support recovery are possible  if  $s_x+s_y\lesssim \sqn$, at least in the known $\Sx$ and $\Sy$ case. The latter sparsity regime can be considerably larger than $ s_x,s_y\lesssim \sqrt{n/\log(p+q)}$. Together, Section~\ref{sec: CT} and the current section  indicate that in the bounded $\B$ case,  the transition of computational hardness for sparse CCA probably happens at the sparsity level $\sqn$, not $\sqrt{n/\log(p+q)}$, which is consistent with  sparse PCA. Also,  the low-degree polynomial conjecture allowed us to explore almost the entire targeted regime $s_x,s_y\gg\sqrt{n}$, where \cite{gao2017}, who used the  planted clique conjecture, considers only a subregime of $s_x,s_y\gg \sqrt{n}/(\B \sqrt{\log(p+q)})$.



\end{remark}
We will end the current section with a brief outline of  the proof of  Theorem \ref{thm: low deg}.
  \textcolor{black}{
  \subsubsection*{The main idea behind the proof of Theorem \ref{thm: low deg}} Let us denote by $\Pi_n^{\leq D_n}$ the linear span of all $n(p+q)$-variate Hermite polynomials of degree at most $D_n$. For each $z\in\ZZ^{m}$ and $y\in\RR^{m}$, we let $\widehat H_z(y)=\prod _{i=1}^{m}\widehat h_{z_i}(y_i)$, where $\widehat h_{z_i}$ is the univariate normalized Hermite polynomial of degree $z_i$. We will discuss the Hermite polynomials in greater detail in Appendix \ref{sec; proof of low deg}. Any normalized $m$-variate Hermite polynomial is of the form $\widehat H_z$, where $z\in \ZZ^{m}$.
   Then  $\Pi_n^{\leq D_n}$ is the linear span of all $\widehat H_w$'s with
   \[w\in \mathcal C_{l}:=\lbs z\in \ZZ^{n(p+q)}: \sum_{i=1}^{n(p+q)}z_i\leq D_n\rbs.\]
  Since $\LL_n^{\leq D_n}$ is the projection of $\LL_n$ on $\Pi_n^{\leq D_n}$, it then holds that
  \[ \|\LL_n^{\leq D_n}\|^2_{L_2(\QQ_n)}=\sum_{w\in \mathcal{C}_l}\langle \LL_n, \widehat H_\myv\rangle_{L^2(\QQ_n)}^2. \] 
  }
  
  \textcolor{black}{The first step of the proof is to find out the expression of $\langle \LL_n, \widehat H_\myv\rangle_{L^2(\QQ_n)}^2$. Since $w\in \ZZ^{n(p+q)}$, we can partition $w$ into  $w=(w_1,\ldots,w_n)$, where $w_i\in\ZZ^{p+q}$ for each $i\in[n]$. Using some algebra, we can show that \begin{align*}
     \langle \LL_n, \widehat H_\myv\rangle_{L^2(\QQ_n)}=
     \ \E_{\pi}\lbt \prod_{i\in[n]}\E_{(X_i,Y_i)\sim \PP_{\alpha,\beta}}\slbt \widehat H_{\myv_i} (X_i,Y_i)\srbt\rbt.
     \end{align*}
Exploiting the properties of Hermite polynomials, it can be shown that
\begin{align*}
  \MoveEqLeft  \E_{(X_i,Y_i)\sim \PP_{\alpha,\beta}}\slbt\widehat H_{\myv_i} (X_i,Y_i)\srbt\\
    =&\ \frac{1\{\|\alpha\|_2\|\beta\|_2<\B\}}{\sqrt{\prod_{j=1}^{p+q}(w_i)_j!}}\\
    &\ \times\partial^{\myv_i}_{t}\lb\exp\left\{\frac{1}{2}t^T\slb \Sigma (\alpha,\beta,1/\B) -I_{p+q}\srb t\right\}\rb\bl_{t=(0,\ldots,0)},
\end{align*}
where for $z\in\ZZ^{p+q}$, $t\in\RR^{p+q}$, and any function $f:\RR^{p+q}\mapsto\RR$, the notation $\partial^z_t(f(t))|_{t=(0,\ldots,0)}$ stands for the $z$-th order partial derivative of $f$ with respect to $t$ evaluated at the origin. }\textcolor{black}{ The rest of the proof is similar to the PCA analogue in \cite{ding2019}, but there is an extra indicator term for the CCA case. Following \cite{ding2019}, we use the common trick of using replicas of $\alpha$ and $\beta$ to simplify the algebra. Suppose  $\alpha_1,\alpha_2\sim \pi_x$ and $\beta_1,\beta_2\sim\pi_y$ are independent. Let $W$ be the indicator function of the event $|(\alpha_1^T\alpha_2)(\beta_1^T\beta_2)|<\B^2$. Denote by $((1-x)^{-n})^{\leq p}$ the $p$-th order truncation of the Taylor series expansion of $(1-x)^{-n}$ at $x=0$.
 Following some algebra, it can be shown that
 \begin{align*}
   \MoveEqLeft  \|\LL_n^{\leq D_n}\|^2_{L_2(\QQ_n)}\\
     =&\ \E_{\pi}\lbt W\lbs \lb 1-\B^{-2}(\alpha_1^T\alpha_2)(\beta_1^T\beta_2)\rb^{-n}\rbs^{\leq \floor{D_n/2}}\rbt.
 \end{align*}
Comparing the above with the analogous result for PCA, namely Lemma 4.2 of \cite{ding2019}, we note that the indicator term $W$ does not appear in the PCA analogue. The indicator term $W$ appears in the CCA case because we had set  $\PP_{\alpha,\beta}$ to be $\QQ$ for  $\|\alpha\|_2\|\beta\|_2>\B$ to tackle  the extra restrictions on $\alpha$ and $\beta$ in this case.
}

\subsection{\textbf {A Polynomial Time Algorithm for $\sqrt{n}/\log{(p+q)}\ll s_x,s_y\ll \sqrt{n}$ Regime : Answer to Question~\ref{Question 3}}}
\label{sec: CT}

In this subsection, we show that  in the difficult regime $s_x+s_y\in[\sqrt{n/\log(p+q)},\sqn]$, using a soft co-ordinate thresholding (CT) type algorithm,  we can estimate the  canonical directions consistently  when $p+q\asymp n$. CT was introduced by the seminal work of \cite{bickel2008} for the purpose of estimating high dimensional covariance matrices. For SPCA, \cite{deshpande2014}'s CT is the only algorithm that  provably recovers the full support in the difficult regime (see also \cite{boaz2015}). In context of CCA, \cite{gao2013} uses CT for partial support recovery in the rank one model under what we referred to as the easy regime. However, \cite{gao2013}'s main goal was the  estimation of the leading canonical vectors, not support recovery.  As a result, \cite{gao2013} detects the support of the relatively large elements of the leading canonical directions, which are subsequently used to obtain consistent preliminary estimators of the leading canonical directions. Our thresholding level and theoretical analysis are different from that of \cite{gao2013} because the analytical tools used in the easy regime do not work in the difficult regime.

\subsubsection{Methodology: Estimation via CT}
By ``thresholding a matrix $A$ co-ordinate-wise",  we will roughly mean the process of assigning the value zero to any element of $A$ which is below a certain threshold in absolute value. Similar to \cite{deshpande2014}, we will consider the soft thresholding operator, which, at threshold level $t$, takes the form
\[\eta(x,t)=\begin{cases}
x-t & x>t\\
0 & |x|<t\\
x+t & x<-t.
\end{cases}\]
It will be worth noting that the soft thresholding operator $x\mapsto\eta(x,t)$ is continuous.

  
 
 
 \begin{algorithm}[h]
\caption{Coordinate Thresholding (CT) for estimating $D(V)$}
\label{algo: CT: CCA}

\begin{algorithmic} 
\REQUIRE \begin{enumerate}
\item Sample covariance matrices $\hSxy^{(1)}$ and $\hSxy^{(2)}$ based on samples  $O_1=(x_i, y_i)_{i=1}^{[n/2]}$ and $O_2=(x_i, y_i)_{i=[n/2]+1}^{n}$, respectively.
\item Variances $\Sx$ and $\Sy$.
\item Parameters $\Thres$ and $\cut$.
\item $r$, i.e., rank of $\Sxy$
\end{enumerate}
\ENSURE  $\widehat D(V)$.
\STATE 
\begin{enumerate}
\item \textbf{Peeling:} calculate $\tSxy=\Sx^{-1}\hSxy^{(1)}\Sy^{-1}$.
\item \textbf{Threshold:} Letting $N=m+n$, perform soft thresholding  $x\mapsto \eta (x;\Thres/\sqrt{N})$ entrywise on $\tSxy$ to obtain thresholded $\eta(\tSxy)$.
\item \textbf{Sandwitch:} $\eta(\tSxy)\mapsto \Sx^{1/2}\eta(\tSxy)\Sy^{1/2}$.
\item \textbf{SVD: } Find $\widehat U_{pre}$, the matrix of the leading $r$ singular vector of  $\Sx^{1/2}\eta(\tSxy)\Sy^{1/2}$. 
\item \textbf{Premultiply:} Set $\widehat{U}^{(1)}=\Sx^{-1/2}\widehat U_{pre}$.
\end{enumerate}
\RETURN  \RecoverSupp\ $(\widehat U^{(1)}, \cut, \Sy^{-1},\hSxy^{(2)},r)$ where \RecoverSupp\  is given by Algorithm~\ref{algo: simple algorithm}.
 \end{algorithmic}
\end{algorithm}

We will also assume that  the covariance matrices $\Sx$ and $\Sy$ are known.  To understand the difficulty of unknown $\Sx$ and $\Sy$, we remind the readers that $\Sxy=\Sx U\Lambda V^T\Sy$. Because the matrices $U$ and $V$ are sandwiched between the matrices $\Sx$ and $\Sy$, their sparsity pattern does not get reflected in the  sparsity pattern of  $\Sxy$.  Therefore, if one blindly applies  CT  to $\hSxy$,   they can at best hope to recover the sparsity pattern of  the outer matrices $\Sx$ and $\Sy$. If  the supports of the matrices $U$ and $V$ are of main concern, CT should rather be applied on the matrix $\tSxy=\Sx^{-1}\hSxy\Sy^{-1}$. If $\Sx$ and $\Sy$ are unknown, one needs to efficiently estimate $\tSxy$ before the application of CT. Although under certain structural conditions, it is possible to find rate optimal estimators $\hSx^{-1}$ and $\hSy^{-1}$ of $\Sx^{-1}$ and $\Sy^{-1}$ at least in theory,  the errors 
 $\|(\hSx^{-1}-\Sx^{-1})\hSxy \Sy^{-1}\|_{op}$ and 
  $\|\Sx^{-1}\hSxy (\hSy^{-1}-\Sy^{-1})\|_{op}$ may still blow up due to the presence of the high dimensional matrix $\hSxy$, which can be as big as $O(\sqrt{(p+q)/n})$ in operator norm.
 One may be tempted to replace $\hSxy$ with a sparse estimator of $\Sxy$ to facilitate faster estimation, but that does not work because  we explicitly require the formulation of $\hSxy$ as the sum of Wishart matrices (see equation \ref{representation: Sxy: model 3} in the proof). The latter representation, which is critical for the sharp analysis, may not be preserved by a CLIME \cite{cai2011} or nodewise Lasso estimator \cite{vandegeer2014} of $\Sxy$.

 We remark in passing that it is possible to obtain an estimator $\widehat A$ so that $|\widehat A-\tSxy|_\infty=o_p(1)$. Although the latter does not provide much  control over the operator norm of $\widehat A-\tSxy$, it is sufficient for  partial support recovery, e.g., the recovery of the rows of $U$ or $V$ with strongest signals. (See Appendix B of \cite{gao2013} for example, for some results  in this direction under the easy regime when $r=1$.)

As indicated by the previous paragraph, we apply co-ordinate thresholding to the matrix $\tSxy=\Sx^{-1}\hSxy\Sy^{-1}$, which directly targets the matrix $\txy=U\Lambda V^T$. We call this step the peeling step because it extracts the matrix $\tSxy$ from the sandwiched matrix $\hSxy=\Sx\tSxy\Sy$. 
We then perform the entry-wise co-ordinate thresholding algorithm on the peeled form $\tSxy$ with threshold $\Thres$ so as  to obtain $\eta(\tSxy; \Thres/\sqrt{n})$. We postpone the discussion on $\Thres$ to Section \ref{subsec: analysis of CT}. The thresholded matrix is an  estimator of $\txy$, but we need an estimator of $\Sx^{-1/2}\Sxy\Sy^{-1/2}$. Therefore, we again sandwich $\tSxy$  between $\Sx^{1/2}$ and $\Sy^{1/2}$. The motivation behind this sandwiching is that if $\|\tSxy-\txy\|_{op}=\e_n$, then $\Sx^{1/2}\tSxy\Sy^{1/2}$ is a good estimator of $\Sx^{-1/2}\Sxy\Sy^{-1/2}$ in that
\[\|\Sx^{1/2}\tSxy\Sy^{1/2}-\Sx^{-1/2}\Sxy\Sy^{-1/2}\|_{op}\leq \sqrt{\|\Sx\|_{op}\|\Sy\|_{op}}\e_n\leq\B\e_n.\]
However, $\Sx^{1/2}U\Lambda V^T\Sy^{1/2}$ is an SVD of $\Sx^{-1/2}\Sxy\Sy^{-1/2}$. Using Davis-Kahan sin theta theorem \cite{yu2015}, one can show that the SVD of $\Sx^{1/2}\tSxy\Sy^{1/2}$ produces   estimators $\widehat U'$ and $\widehat V'$ of $\Sx^{1/2}U$ and $\Sy^{1/2}V$,  where the columns of  $\widehat U'$ and $\widehat V'$ are  $\e_n$-consistent in $l_2$ norm for the columns of $\Sx^{1/2}U$ and $\Sy^{1/2}V$, respectively, up to a sign flip (cf. Theorem 2 of \cite{yu2015}). Pre-multiplying the resulting $U'$ by $\Sx^{-1/2}$ yields an estimator $\widehat U$ of $U$ up to a sign flip of the columns. We do not worry about the sign flip because Condition~\ref{Assumption on estimators} allows for the sign flips of the columns.
 Therefore, we feed this $\widehat U$  into \RecoverSupp\ as our final step.
See Algorithm~\ref{algo: CT: CCA} for more details. 

\begin{remark}
\textcolor{black}{In case of electronic health records data, it is possible to obtain large surrogate data on $X$ and $Y$ separately and thus might allow relaxing the known precision matrices assumption above. We do not pursue such semi-supervised setups here. }
\end{remark}

\subsubsection{Analysis of the CT Algorithm}
\label{subsec: analysis of CT}

For the asymptotic analysis of the CT algorithm, we will assume the underlying distribution to be Gaussian, i.e., $\mathbb P\in\mP_G(r,s_x,s_y,\B)$.  This Gaussian assumption will be used to perform a crucial decomposition of sample covariance matrix, which typically holds for Gaussian random vectors. \cite{deshpande2014}, who used similar devices for obtaining the sharp rate results in SPCA,   also required a similar Gaussian assumption. We do not yet know how to extend these results to  sub-Gaussian random vectors.

Let us consider the threshold  $\Thres/\sqrt {n}$, where $\Thres$ is explicitly given in Theorem~\ref{thm: CT matrix operator norm}. Unfortunately, tuning of $\Thres$ requires the knowledge of the underlying sparsity $s_x$ and $s_y$.
 Similar to \cite{deshpande2014}, our  thresholding level  is different than the traditional choice of order $O(\sqrt{\log (p+q)/n})$ in the easy regime analyzed in \cite{bickel2008, cai2012} and \cite{ gao2013}. The latter level is   too large to successfully recover all the nonzero elements in the difficult regime. We  threshold $\tSxy$ at a lower level, which in its turn, complicates the analysis to a greater degree. Our main result in this direction, stated in 
 Theorem~\ref{thm: CT matrix operator norm}, is proved in Appendix~\ref{sec: Proof of CT}.
  
\begin{theorem}\label{thm: CT matrix operator norm}
Suppose $(X_i,Y_i)\sim\mathbb P\in \mP_G(r, s_x,s_y, \B)$. Further suppose $s_x+s_y< \sqn$, $p\vee q=o(\log n)$, and $\log n=o(\sqrt p\vee \sqrt q)$. Let $K$ and $C_1$ be constants so that $K\geq 1288\B^4$ and $C_1\geq C\B^4$, where $C>0$  is an absolute constant. Suppose the threshold level $\Thres$ is defined by
\begin{align*}
\Thres=\begin{cases}
\sqrt{C_1\log(p+q)}  
\\
\text{ if }\quad (s_x+s_y)^2 < 2^{1/4}(p+q)^{3/4}\text{ (case i)}
\\
\slb K\log(\frac{p+q}{(s_x+s_y)^2})\srb^{1/2} 
\\\text{ if }\quad 2^{1/4}(p+q)^{3/4}\leq (s_x+s_y)^2\leq (p+q)/e\\
\text{ (case ii)}\\
0
\quad \text{ o.w. (case iii).}
\end{cases}
\end{align*}
Suppose $c_\B$ is a constant that takes the value $K$,  $C_1$, or one in case (i), (ii), and (iii), respectively. 
Then there exists an absolute constant $C>0$ so that the following holds with probability $1-o(1)$ for $\tSxy=\Sx^{-1}\hSxy\Sy^{-1}:$
\begin{align*}
  \MoveEqLeft \|\eta(\tSxy;\eta)-\txy\|_{op}\leq C\B^2 \frac{(s_x+s_y)}{\sqn}\\
   &\ \times\max\lbs \lb c_\B\log (\frac{p+q}{(s_x+s_y)^2})\rb^{1/2}, 1\rbs.
\end{align*}
\end{theorem}
To disentangle the statement of Theorem \ref{thm: CT matrix operator norm}, let us assume $p+q\asymp n$ for the time being. Then case (ii) in the theorem corresponds to $n^{3/4}\lesssim (s_x+s_y)^2\leq n$.
Thus, CT works in the difficult regime provided $p+q\asymp n$. It should be noted that the threshold for this case is almost of the order $O(1/\sqn)$, which is much smaller than $O(\sqrt{\log(p+q)/n})$, the traditional threshold for the easy regime. Next, observe that case (i) is an easy case because $s_x+s_y$ is much smaller than $\sqn$. Therefore, in this case, the traditional threshold of the easy regime works. Case (iii) includes the hard regime, where polynomial time support recovery is probably impossible.   Because  it is unlikely that CT can  improve over the vanilla estimator $\tSxy$ in this regime, a threshold of zero is set.

\begin{remark}
Theorem \ref{thm: CT matrix operator norm} requires $\log n=o(\sqrt{p}\vee\sqrt{q})$ because one of our concentration inequalities in the  analysis of case (ii)  needs this technical condition (see Lemma~\ref{lemma: sqrt n: M N lemma: rot}). The omitted  regime $\log n>C(\sqrt{p}\vee\sqrt{q})$ is indeed an easier one, where  special methods like CT is not even required. In fact,  it is well known that subgaussian $X$ and $Y$ satisfy (cf. Theorem 4.7.1 of \cite{vershynin2020})
\[\|\hSxy-\Sxy\|_{op}\leq C\lb\lb\frac{p+q}{n}\rb^{1/2}+\frac{p+q}{n}\rb,\]
which is $O(\log n/\sqn)$ in the regime under concern. Including this result in the statement of Theorem \ref{thm: CT matrix operator norm} could  unnecessarily lengthen the exposition. Therefore,  we decided to exclude  this regime from Theorem \ref{thm: CT matrix operator norm} to focus more on the $s_x+s_y\approx \sqrt{p+q}$ regime.
\end{remark}

\begin{remark}
\label{remark: CT constants}
The statement of Theorem~\ref{thm: CT matrix operator norm} is not explicit on the lower bound of the constant $C_1$. However,  our simulation shows that the algorithm works for $C_1\geq 50\B^4$.
 Both  threshold parameters $C_1$ and $K$ in Theorem \ref{thm: CT matrix operator norm} depend on the unknown $\B>0$. The proof actually shows that $\B$ can be replaced by $\max\{ \Lambda_{max}(\Sx),\Lambda_{max}(\Sy), \Lambda_{max}(\Sx^{-1}), \Lambda_{max}(\Sy^{-1})\}$.  
\end{remark}



Finally, Theorem \ref{thm: CT matrix operator norm} leads to the following corollary, which establishes that in the difficult regime, there exist estimators which satisfy Condition~\ref{Assumption on estimators}, and Algorithm~\ref{algo: CT: CCA}   succeeds with probability one  provided $p+q\asymp n$. This answers Question~\ref{Question 3} in the affirmative for Gaussian distributions.
\begin{corollary}\label{cor: CT support recovery}
  Instate the conditions of Theorem~\ref{thm: CT matrix operator norm}. Then there exists $C_\B>0$ so that if 
  \begin{equation}\label{intheorem: statement: CT support recovery}
     n\geq C_\B r(s_x+s_y)^2\max\lbs \log\lb \frac{p+q}{(s_x+s_y)^2}\rb ,1\rbs, 
  \end{equation}
  then the 
  $\widehat U^{(1)}$ defined in  Algorithm~\ref{algo: CT: CCA} satisfies 
  Condition~\ref{Assumption on estimators}, and 
  $\inf_{\PP\in\modG}\PP($Algorithm~\ref{algo: CT: CCA} correctly recovers $D(V)$ $)\to_n 1$.
\end{corollary}

We defer the proof of Corollary~\ref{cor: CT support recovery} to Appendix~\ref{sec: proof of cor CT}.
\textcolor{black}{
We will now present a brief outline of the proof of Theorem \ref{thm: CT matrix operator norm}.
\subsubsection*{Main idea behind the proof  of Theorem \ref{thm: CT matrix operator norm}} The proof hinges on the hidden variable representation of $X$ and $Y$ due to \cite{bach2005}. We discuss this representation in detail in Appendix  \ref{secpf: case B: Thm 1}, which basically says the data matrices $\mX$ and $\mY$  can be represented as
\[
     \mX= \mZ\W_1^T+\ \mZ_1\H_1^T\quad \text{ and }\quad \mY= \mZ\W_2^T+ \mZ_2\H_2^T,
 \]
 where $\mZ\in\RR^{n\times r}$, $\mZ_1\in\RR^{n\times p}$, and $\mZ_2\in\RR^{n\times q}$ are independent  standard Gaussian data matrices, and 
 $\W_1=\Sx U\Lambda^{1/2}$, $\W_2=\Sy V\Lambda^{1/2}$, $\H_1=(\Sx-\W_1\W_1^T)^{1/2}$, 
 and
$\H_2=(\Sy-\W_2\W_2^T)^{1/2}$. We will later show in Section  \ref{secpf: case B: Thm 1} that $\H_1$ and $\H_2$ are well defined positive definite matrices. It follows that $\hSxy=X^TY/n$ has the representation
\begin{align*}
\hSxy=&\ \frac{1}{n}\lbs \W_1 \mZ^T\mZ\W_2^T+\W_1 \mZ^T\mZ_2\H_2+\H_1^T\mZ_1^T\mZ \W_2^T\nn\\
&\  +\H_1^T\mZ_1^T\mZ_2\H_2\rbs.
\end{align*}
Next, we define some sets. Let $E_1=\cup _{i=1}^{r}D(\ai)$, $F_1=[p]\setminus E_1$, $E_2=\cup _{i=1}^{r}D(\bi)$, and $F_2=[q]\setminus E_2$. Therefore $E_1$ and $E_2$ correspond to the supports, where $F_1$ and $F_2$ correspond to their complements. 
  Now we  partition  $[p]\times[q]$ into the following three sets:
  \begin{align}\label{def: E and F}
  E=E_1\times E_2,\quad F= F_1\times F_2,
  \end{align}
and
\begin{equation}\label{def: G}
G=\lb F_1 \times E_2\rb \cup \lb E_1\times F_2\rb.
\end{equation}
Therefore $E$ is  the set that contains the joint support.
We can decompose $\tSxy$ as
\begin{equation}\label{intheorem: CT operator: hsy decomp}
 \tSxy=\underbrace{\mP_E\{\tSxy\}}_{\mS_1}+\underbrace{\mP_F\{\tSxy\}}_{\mS_2}+\underbrace{\mP_G\{\tSxy\}}_{\mS_3}.   \end{equation}
where $\mP$ is the projection operator defined in \eqref{def:projection operator}.
}

\textcolor{black}{
The usefulness of the decomposition in \eqref{intheorem: CT operator: hsy decomp} is that $S_1$, $S_2$, and $S_3$ have different supports, which enables us to write
\[\eta(\tSxy)=\eta(S_1)+\eta(S_2)+\eta(S_3).\]
We can therefore analyze the three terms $\eta(S_1)$, $\eta(S_2)$, and $\eta(S_3)$ separately. 
In general, the thresholding operator $\eta$ is not linear in that for matrices $A$ and $B$,  $\eta(A+B)=\eta(A)+\eta(B)$ generally does  not hold. 
}

\textcolor{black}{
As indicated above, we analyze the operator norms of  $\eta(S_1)$, $\eta(S_2)$, and $\eta(S_3)$ separately.
Among $S_1$, $S_2$, and $S_3$,  $S_1$ is the only matrix that is supported on $E$, the true support. The basic idea of the proof is showing that co-ordinate thresholding preserves the matrix $S_1$, and kills off the other matrices $S_2$ and $S_3$, which contain the noise terms. $S_1$ includes the matrix $U\Lambda^{1/2}\mZ^T\mZ\Lambda^{1/2}U^T$. Because $\mZ^T\mZ$ concentrates around $I_r$ by Bai-Yin law (cf.  Lemma 4.7.1 of \cite{vershynin2020}), $U\Lambda^{1/2}\mZ^T\mZ\Lambda^{1/2}U^T$ concentrates around $U\Lambda V^T$.   Therefore the analysis of $\eta(S_1)$ is relatively straightforward.
}

\textcolor{black}{Most of the proof is devoted towards showing $\|\eta(S_2)\|_{op}$ and $\|\eta(S_3)\|_{op}$ are small, i.e.,  co-ordinate thresholding kills off the noise terms. The difficulty  arises because the threshold was kept smaller than the traditional threshold of order $\sqrt{\log(p+q)/n}$ to 
 adjust for the hard regime. Therefore the approaches of  \cite{bickel2008} or \cite{gao2017} do not work in this regime. The noise matrices $S_2$ and $S_3$ are sum of matrices of the form $M\mZ^T\mZ_1 N$, $M\mZ^T_1\mZ_2 N$, or $M\mZ^T\mZ_2 N$, or their transposes, where for rest of this section, $M$ and $N$ should be understood as  deterministic matrices of appropriate dimension, whose definition can change from line to line.  Analyzing $\|\eta(S_2)\|_{op}$ and $\|\eta(S_3)\|_{op}$  essentially hinges on Lemma \ref{lemma: sqrt n: M N lemma: rot}, which upper bounds the operator norm of matrices of the form $\eta(M\mZ_1^T\mZ_2 N)$.  The proof of Lemma  \ref{lemma: sqrt n: M N lemma: rot} uses, among other tools, a  sharp Gaussian concentration result from \cite{deshpande2014} (see Corollary 10 therein), and a generalized Chernoff's  inequality for dependent Bernoulli random variables   \cite{panconesi1997}. Using  Lemma  \ref{lemma: sqrt n: M N lemma: rot}, we can also upper bound operator norms of matrices of the form $\eta(M_1\mZ_1^T\mZ N_1+M_2\mZ_1^T\mZ_2 N_2)$  because $M_1\mZ_1^T\mZ N_1+M_2\mZ_1^T\mZ_2 N_2$ can be represented as $M_3 [\mZ \ \mZ_1]^T\mZ_2 N_2$ for some matrix $M_3$ of appropriate dimension. Therefore,  to show  $\|\eta(S_2)\|_{op}$ and $\|\eta(S_3)\|_{op}$ are small, Lemma  \ref{lemma: sqrt n: M N lemma: rot} suffices, which also completes the proof.}

\textcolor{black}{
The proof of Theorem \ref{thm: CT matrix operator norm} has similarities with the proof of the analogous result for PCA in \cite{deshpande2014} (see Theorem 1 therein). However, one main difference is that for PCA, the key instrument is the representation of $X$ as the spiked model \cite{johnstone2001}, which yields the representation 
\begin{align}\label{inroadmap: spike model}
    \mX=\mZ M+\sigma \mZ_1,
\end{align}
where $\mZ\in\RR^{n\times r}$ and $\mZ_1\in\RR^{n\times p}$ are standard Gaussian data matrices, and $M\in\RR^{r\times p}$ is a deterministic matrix. The analysis in PCA revolves around the sample covariance matrix $\hSx=X^TX/n$, which, following \eqref{inroadmap: spike model},  writes as
\[\hSx=\frac{1}{n}\lbs M^T\mZ^T\mZ M+\sigma \mZ_1^T\mZ M+\sigma M^T\mZ^T\mZ_1+\sigma^2\mZ_1^T\mZ_1\rbs.\]
From the above representation, it can be shown that the analogues of $S_2$ and $S_3$  in the PCA case are sum of matrices of the form $M_1\mZ_1^T\mZ_2$ or their transposes.  \cite{deshpande2014} uses an upper bound on $\|\eta(\mZ_1^T\mZ_2)\|_{op}$ to bound the PCA analogue of $\|\eta(S_2)\|_{op}$ and $\|\eta(S_3)\|_{op}$ (see Proposition 13 therein). In contrast,  we encounter terms of the form $M_1\mZ_1^T\mZ_2N_1$ since CCA is concerned with $X^TY/n$. To deal with these terms,  we needed the upper bound result  on $\|\eta(M_1\mZ_1^T\mZ_2 N_1)\|_{op}$ instead, which requires a separate elaborate proof. Although the basic idea behind bounding   $\|\eta(M_1\mZ_1^T\mZ_2 N_1)\|_{op}$ and bounding  $\|\eta(\mZ_1^T\mZ_2)\|_{op}$ is similar,   the proof  of bounding $\|\eta(M_1\mZ_1^T\mZ_2 N_1)\|_{op}$ is more involved. For example, some independence structures are destroyed due to the pre and post  multiplication by the matrices $M_1$ and $N_1$, respectively. We required  concentration inequalities on dependent Bernoulli random variables to tackle the latter. 
}

\section{Numerical Experiments}
\label{sec:simulations}
\begin{figure*}[h]
\centering
\subfigure[Type I error for support recovery of $\alpha$]{
         \includegraphics[width=\textwidth]{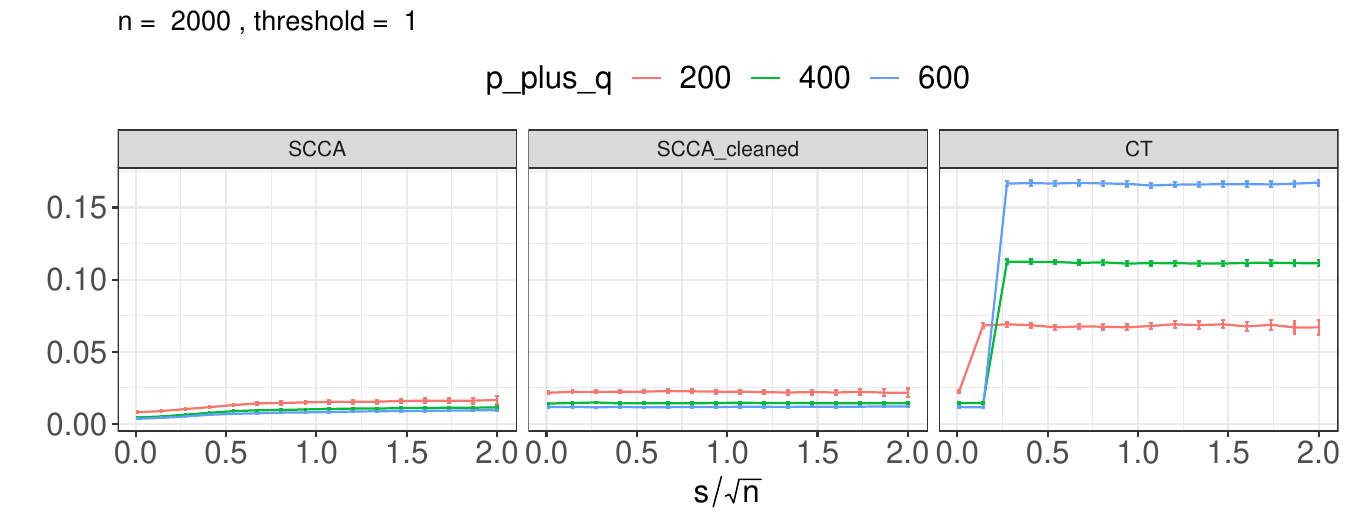}
        }
     \hfil
    \subfigure[Type II error  for support recovery of $\alpha$]{
\includegraphics[width=\textwidth]{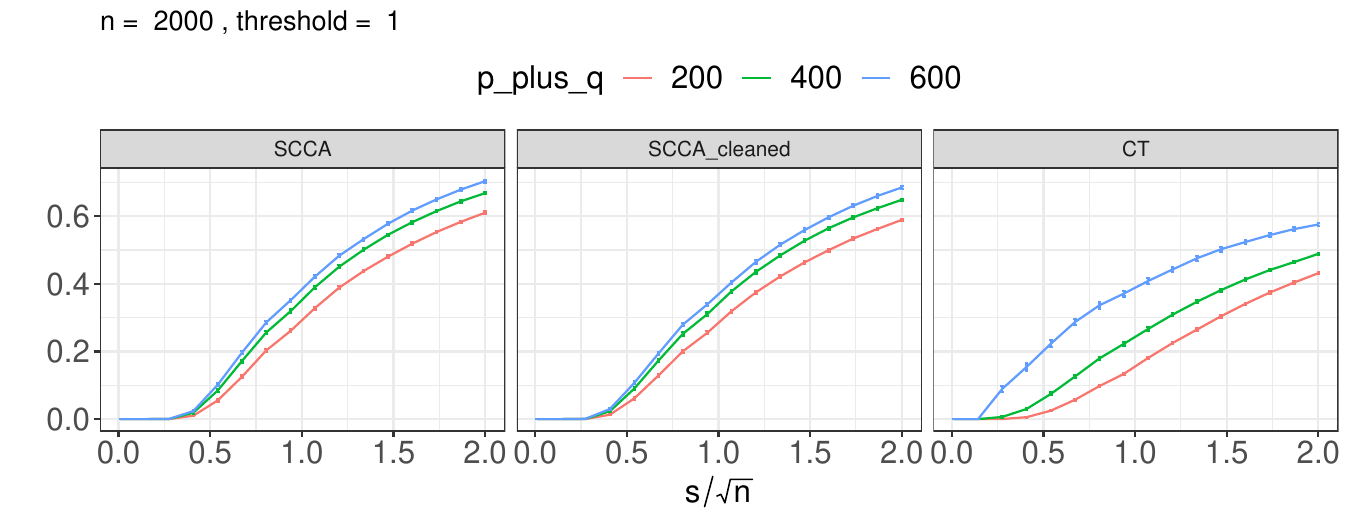}}
       \hfil
      \subfigure[Symmetrized Hamming error for support recovery of $\alpha$]
         {\includegraphics[width=\textwidth]{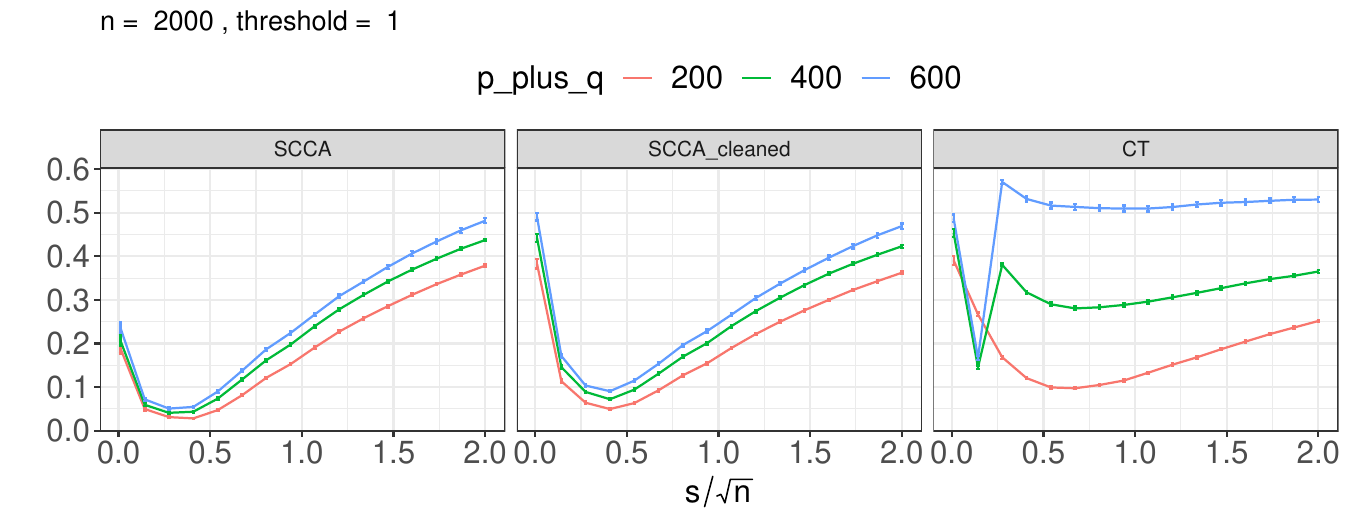}}
         \label{fig:three sin x 2}
     \caption{Support recovery for $\alpha$ when $\Sx=I_p$ and $\Sy=I_q$. Here threshold refers to \texttt{cut} in Theorem \ref{thm: support recovery: sub-gaussians}.
     }
    \label{Fig: id: alpha: large n}
   \end{figure*}
 
 \begin{figure*}[h]
    \centering
\subfigure[Type I error for support recovery of $\beta$]{
         \includegraphics[width=\textwidth]{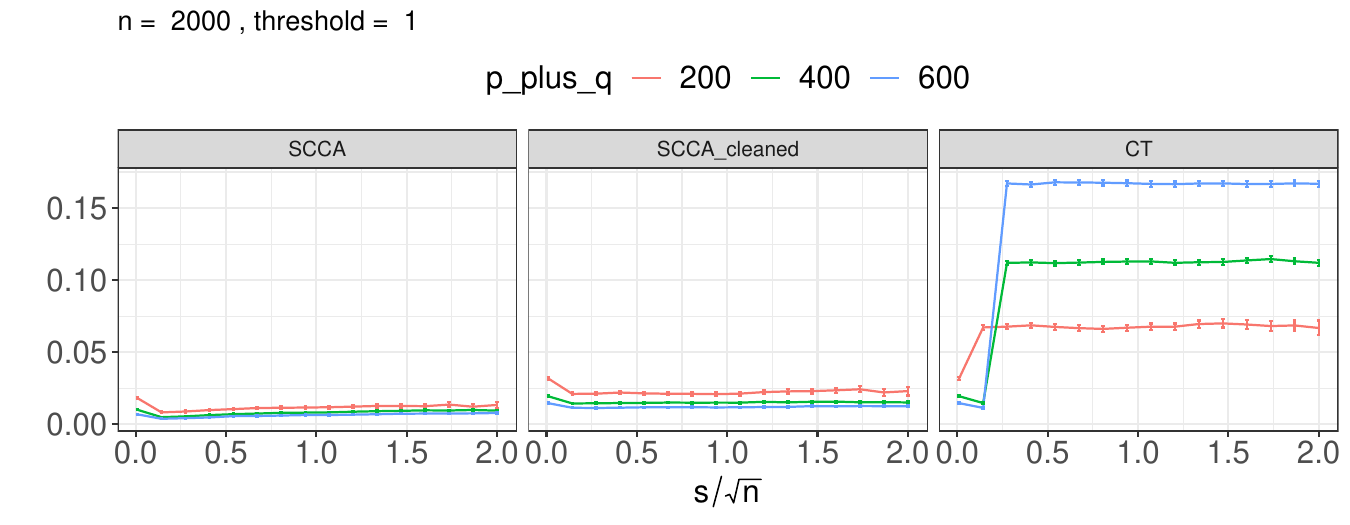}
     }
  \subfigure[Type II error  for support recovery of $\beta$]{
\includegraphics[width=\textwidth]{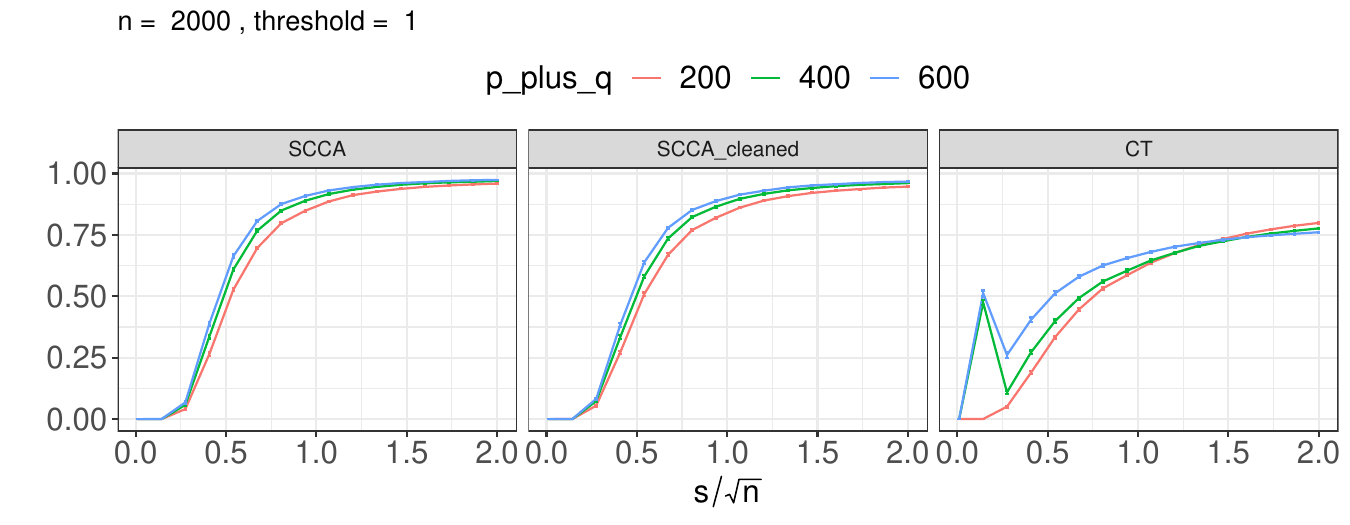}
 }
 \subfigure[Symmetrized Hamming error for support recovery of $\beta$ ]{
         \includegraphics[width=\textwidth]{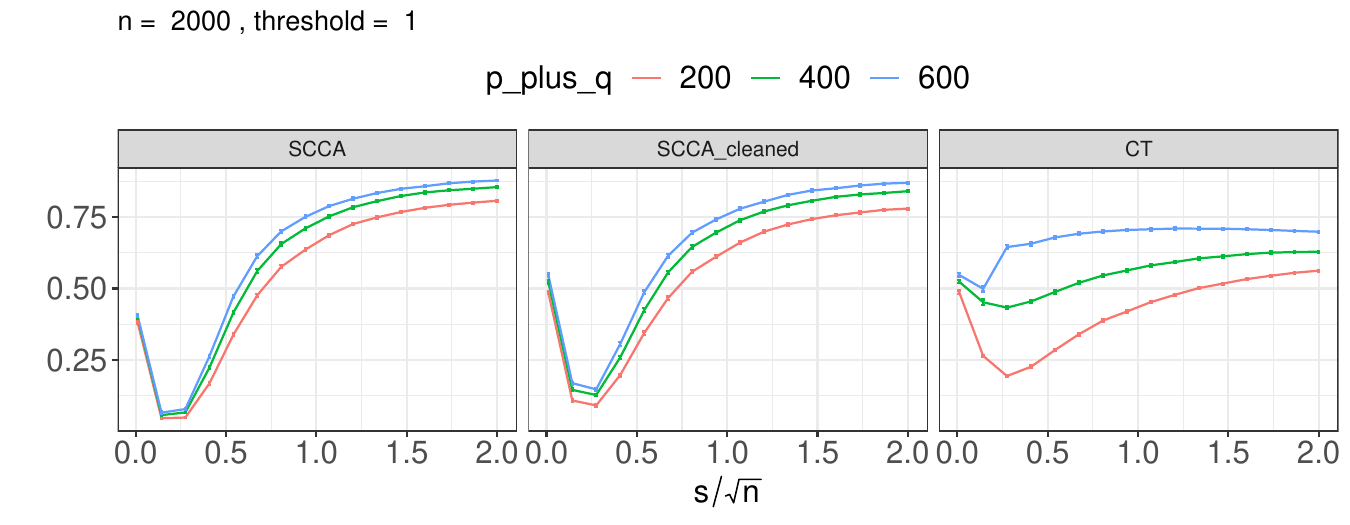}
         \label{fig:three sin x 3}
         }
     \caption{Support recovery for $\beta$ when $\Sx=I_p$ and $\Sy=I_q$. Here threshold refers to \texttt{cut} in Theorem \ref{thm: support recovery: sub-gaussians}.
     }
    \label{Fig: id: beta: large n}
   \end{figure*}

   \begin{figure*}[h]
    \centering
\subfigure[Type I error for support recovery of $\alpha$]{
         \includegraphics[width=\textwidth]{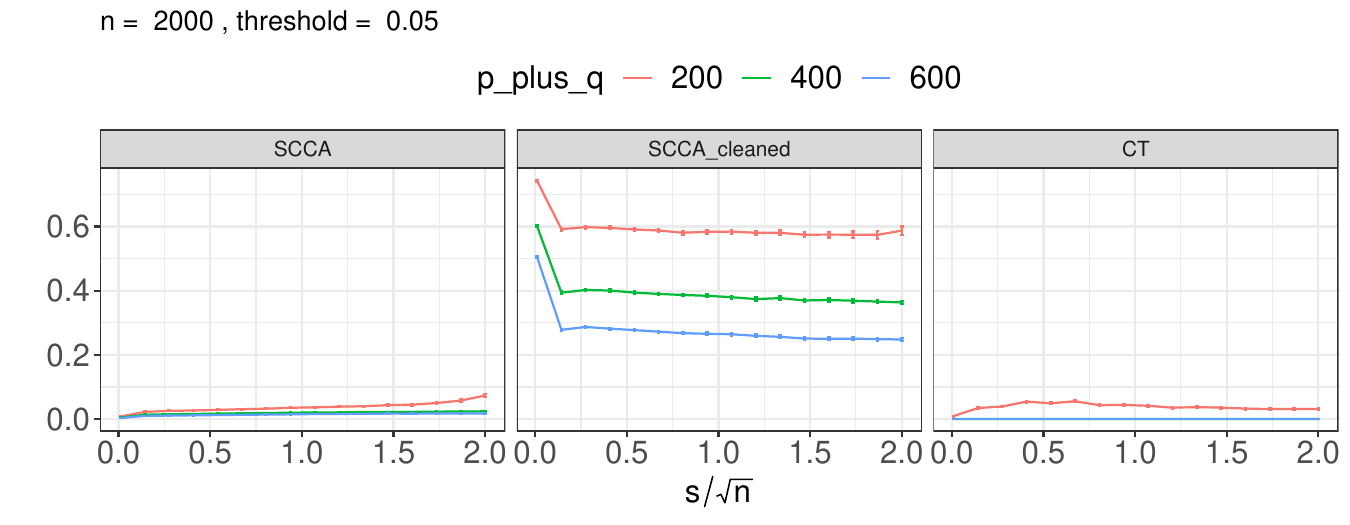}
     }
 \subfigure[Type II error  for support recovery of $\alpha$]{
\includegraphics[width=\textwidth]{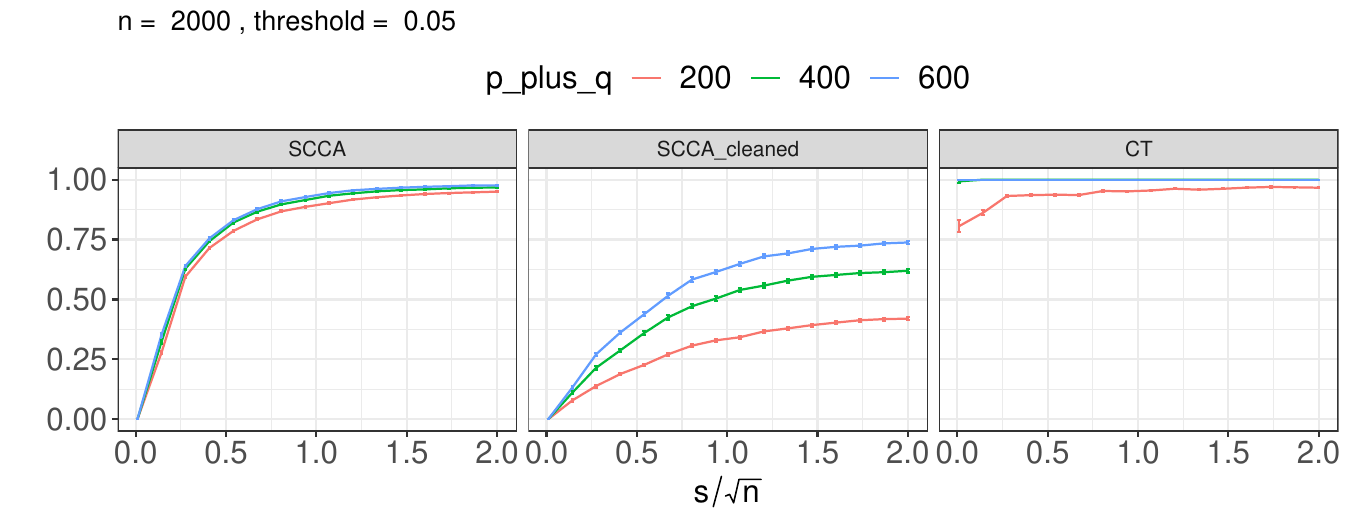}
        }
      \subfigure[Symmetrized Hamming error for support recovery of $\alpha$]{
         \includegraphics[width=\textwidth]{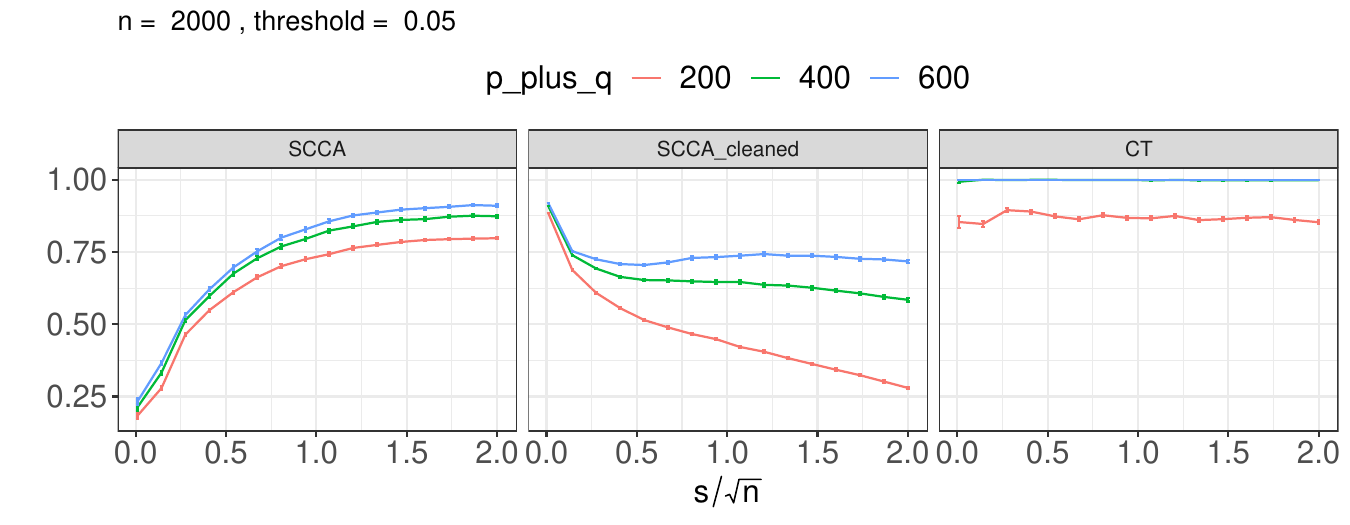}
         \label{fig:three sin x 4}
 }
     \caption{Support recovery for $\alpha$ when $\Sx$ and $\Sy$ are the sparse covariance matrices. Here threshold refers to \texttt{cut} in Theorem \ref{thm: support recovery: sub-gaussians}.}
    \label{Fig: sparseinv: alpha: large n}
   \end{figure*}
   
   \begin{figure*}[h]
    \centering
\subfigure[Type I error for support recovery of $\beta$]{
         \includegraphics[width=\textwidth]{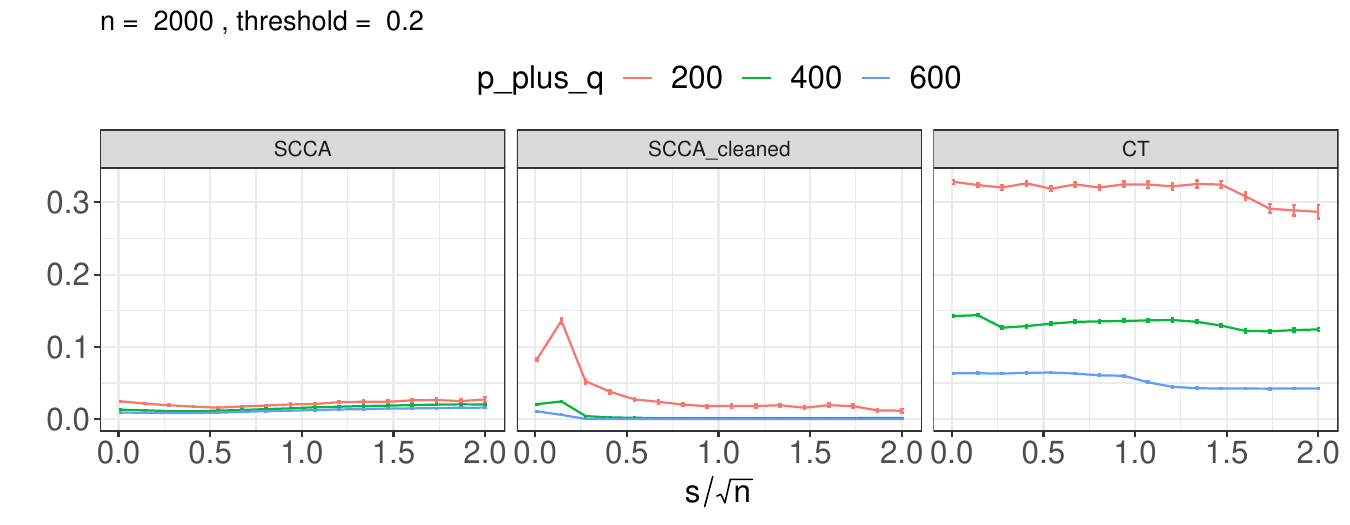}

    }
 \subfigure[Type II error  for the support recovery of $\beta$]{
\includegraphics[width=\textwidth]{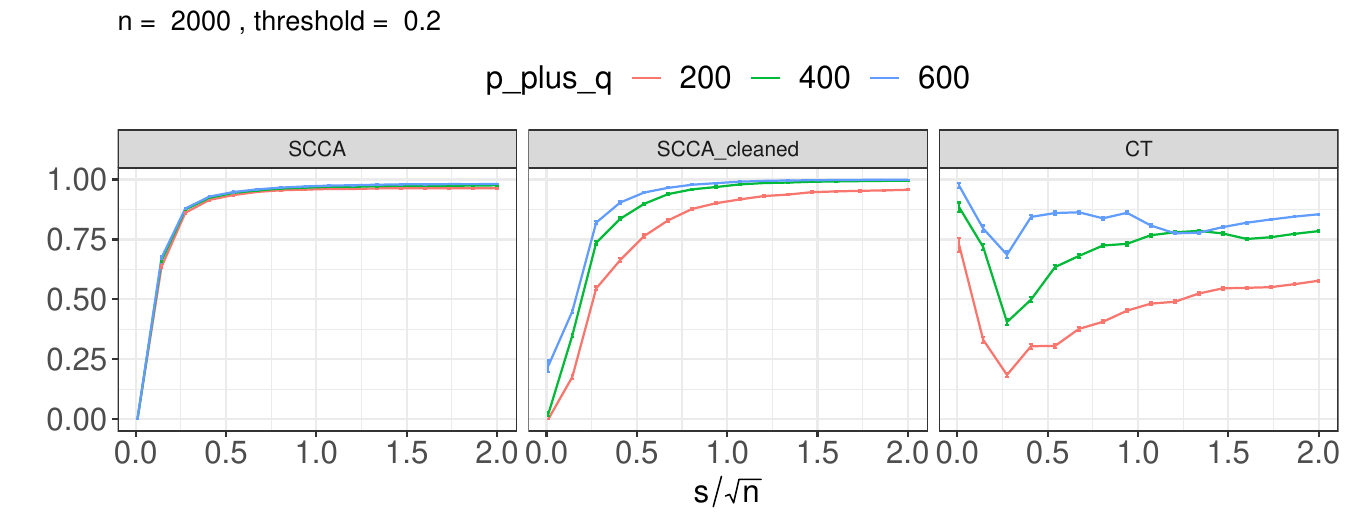}
        }
      \subfigure[Symmetrized Hamming error for support recovery of $\beta$]{
         \includegraphics[width=\textwidth]{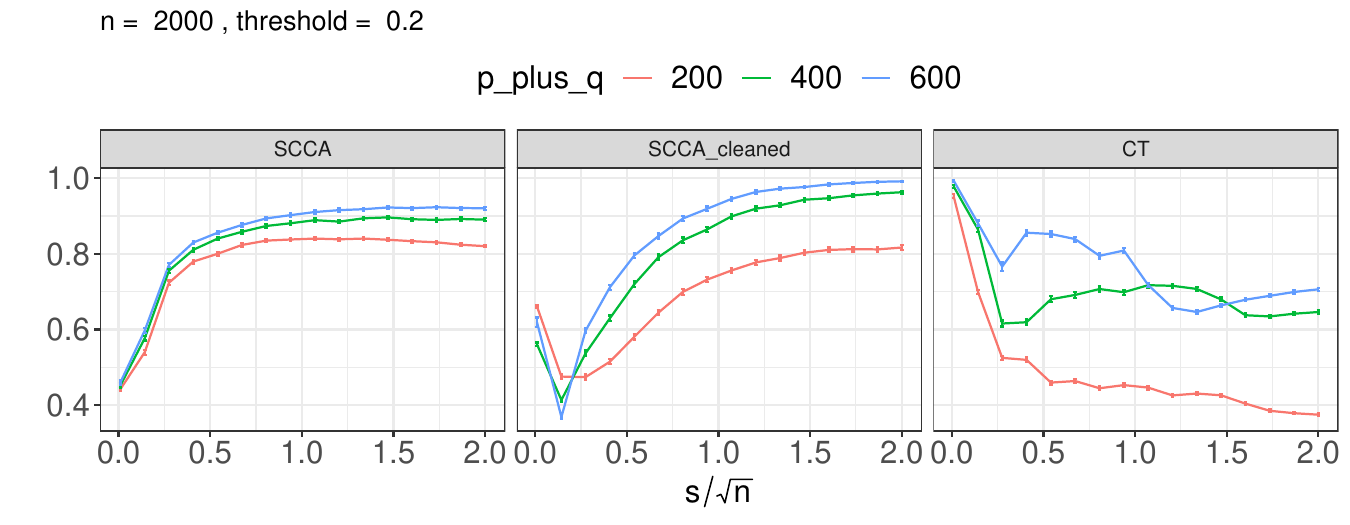}
         \label{fig:three sin x 1}
   }
     \caption{Support recovery for $\beta$ when $\Sx$ and $\Sy$ are the sparse covariance matrices. Here threshold refers to \texttt{cut} in Theorem \ref{thm: support recovery: sub-gaussians}.}
    \label{Fig: sparseinv: beta: large n}
   \end{figure*}
   
   \begin{figure*}[h]
    \centering
\subfigure[Errors for support recovery of $\alpha$]
    {
         \includegraphics[width=\textwidth]{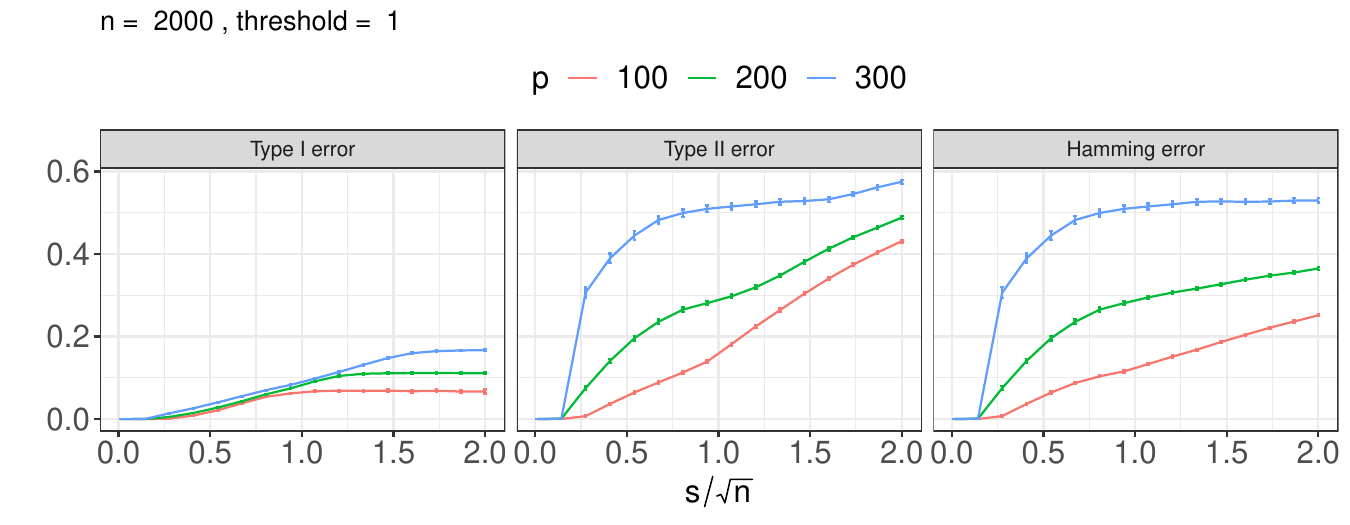}
        }
      \subfigure[Errors for support recovery of $\beta$]
      {
        \includegraphics[width=\textwidth]{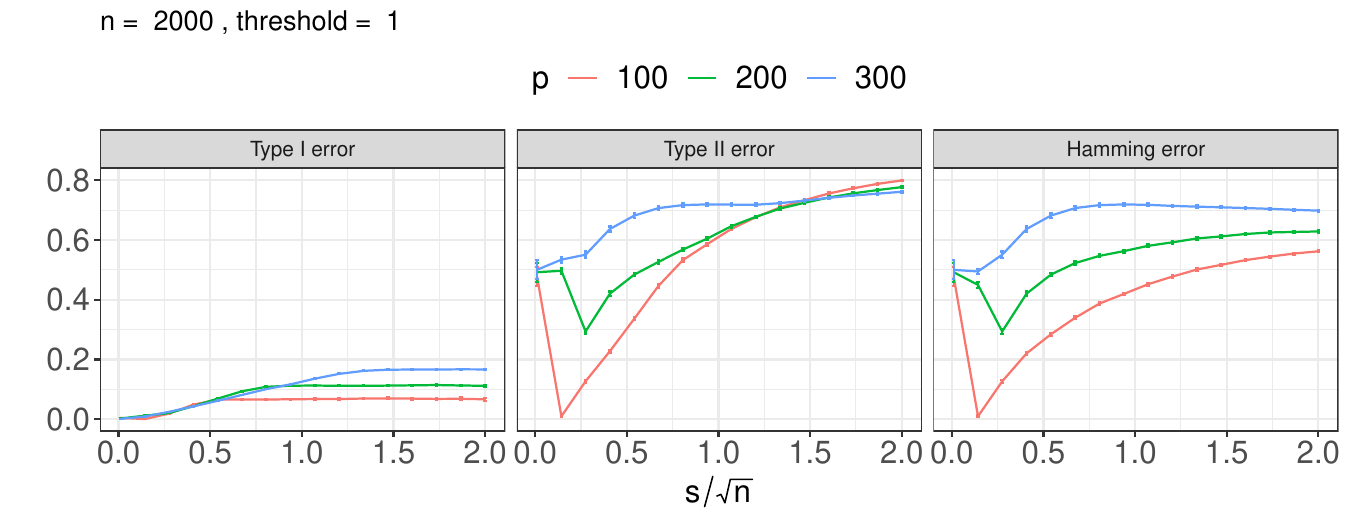}
         \label{fig:three sin x}
      }
     \caption{Support recovery by the CT algorithm when we use the information on sparsity to improve the  type I error.   Here $\Sx$ and $\Sy$ are $I_p$ and $I_q$, respectively, and  threshold refers to \texttt{cut} in Theorem \ref{thm: support recovery: sub-gaussians}.  To see the decrease in type I error, compare the errors with that of Figure~\ref{Fig: id: alpha: large n} and Figure~\ref{Fig: id: beta: large n}.}
    \label{Fig: all error}
   \end{figure*}
This section illustrates the performance of different polynomial time CCA support recovery methods when the sparsity transitions from the easy to difficult regime. We base our demonstration on a  Gaussian rank one model, i.e., $(X,Y)$  are jointly Gaussian with covariance matrix $\Sxy=\rho\Sx\alpha\beta^T\Sy$. For simplicity, we take $p=q$ and $s_x=s_y=s$. In all our simulations, $\rho$ is set to be $0.5$, and $\alpha=\alpha^*/\sqrt{(\alpha^*)^T\Sx\alpha^*}$, $\beta=\beta^*/\sqrt{(\beta^*)^T\Sy\beta^*}$ where 
\begin{align*}
    \alpha^*=  &\ (1/\sqrt{s}, \ldots,1/\sqrt{s},0,\ldots,0),\\
    \beta^*=&\ \slb\sqrt{1- (s-1)s^{-4/3}},s^{-2/3},\dots,s^{-2/3},0,\ldots,0\srb 
\end{align*}
are unit norm vectors.
Note that the order of most elements of $\beta$ is $O(s^{-2/3})$, where a typical element of $\alpha$ is $O(s^{-1/2})$. Therefore,  we will refer to $\alpha$ and $\beta$ as the moderate and the small signal case, respectively. For the population covariance matrices $\Sx$ and $\Sy$ of $X$ and $Y$, we consider the following two scenarios:
\begin{itemize}
\item \textbf{A (Identity):} $\Sx=I_p$ and $\Sy=I_q$. Since $p=q$, they are essentially the same.
\item \textbf{B (Sparse inverse):} This example is taken from \cite{gao2017}. In this case, $\Sx^{-1}=\Sy^{-1}$ are banded matrices, whose entries are given by
\begin{align*}
    (\Sx^{-1})_{i,j}=&\ 
1\{i=j\}+
0.65\times 1\{|i-j|=1\}\\
&\ +
0.4\times 1\{|i-j|=2\}.
\end{align*}
\end{itemize}
  Now we explain our common simulation scheme.
  We take the sample size $n$ to be $1000$, and consider three values for $p$: $100$, $200$, and $300$. The highest value of  $p+q$  is thus $600$, which is smaller than but in  proportion to $n$ regime.  Our simulations indicate that all of the methods considered here requires $n$ to be quite larger than $p+q$ for the asymptotics to kick in at $\rho=0.5$. We will later discuss this point in  detail. We further let  $s/\sqn$  vary in the set $[0.01,2]$. To be more specific, we consider $16$ equidistant points in the set $[0.01,2]$ for the ratio $s/\sqn$. 
  
  Now we discuss the error metric used here to compare the performance of different support recovery methods. Type I  and type II errors are   commonly used tools  to measure the performance of support recovery \cite{deshpande2014}. In case of support recovery of $\alpha$, we define the type I error  to be the proportion of zero elements in $\alpha$ that appear in the estimated support $\widehat D(\alpha)$. Thus, we quantify the type I error of $\alpha$ by $|\widehat D(\alpha)\setminus D(\alpha)|/(p-s)$. On the other hand, the type II error for $\alpha$ is the proportion of elements in $D(\alpha)$ which are absent in $\widehat D(\alpha)$, i.e., the type II error is quantified by $|D(\alpha)\setminus \widehat D(\alpha)|/s $. One can define the type I and type II errors corresponding to $\beta$ similarly.
  Our simulations demonstrate that often the methods with low type I error  exhibit a high type II error, and vice versa. In such situations, comparison between the corresponding methods becomes difficult if one  uses the type I and type II errors separately. Therefore, we consider a scaled 
  Hamming loss type metric, which suitably combines  the type I and type II error. The symmetric Hamming error of  estimating  $D(\alpha)$ by $\widehat D(\alpha)$ is  
  \cite[Section 2.1]{wang2021}
  \[ 1-\frac{|D(\alpha)\cap \widehat D(\alpha)|}{\sqrt{|D(\alpha)|| \widehat D(\alpha)|}}.\]
  Note that the above quantity is always bounded above by one. We can similarly define the symmetric Hamming distance between $D(\beta)$ and  $\widehat D(\beta)$. Finally, the estimates of these three errors (Type I, Type II, and scaled Hamming Loss) are obtained based on $1000$ Monte Carlo replications.
  
   Now we discuss the support recovery methods we compare here. \begin{itemize}
    \item \textbf{Na\"{i}ve  SCCA.} We estimate $\alpha$ and $\beta$ using the SCCA method of \cite{mai}, and set $\widehat D(\alpha)=\{i\in[p]:\widehat\alpha_i\neq 0\}$ and  $\widehat D(\beta)=\{i\in[q]:\widehat\beta_i\neq 0\}$, where $\widehat\alpha$ and $\widehat\beta$ are the corresponding SCCA estimators. To implement the SCCA method of \cite{mai}, we use the  R code referred therein with  default  tuning parameters.
    \item \textbf{Cleaned SCCA.} This method implements \textsc{RecoverSupp} with the above mentioned SCCA estimators of $\alpha$ and $\beta$ as the preliminary estimators.
    \item \textbf{CT.} This is the method outlined in Algorithm~\ref{algo: CT: CCA}, which is \textsc{RecoverSupp} coupled with the CT estimators of $\alpha$ and $\beta$. 
    \end{itemize} 
    
    Our CT method requires the knowledge of the population covariance matrices $\Sx$ and $\Sy$. Therefore, to keep the comparison fair, in case of the cleaned SCCA method as well, we implement \textsc{RecoverSupp} with the popular covariance matrices. Because of their reliance on \RecoverSupp, both cleaned SCCA and CT 
     depend on the threshold $\cut$, tuning which seems to be a non-trivial task. We set $\cut=C\sqrt{\log(p+q)s(\Sx^{-1})/n}$, where $C$ is the thresholding constant.  Our simulations show that a large $C$ results in high  type II error, where insufficient thresholding inflates the type I error. Taking the hamming loss  into account, we observe that  $C\approx 1$ leads to a better performance in case A in an overall sense. On the other hand, case B requires a smaller value of thresholding parameter.
  In particular, we let $C$ to be one in case A, and set $C=0.05$ and $0.2$, respectively, for the support recovery of $\alpha$ and $\beta$ in case B. The CT algorithm requires an extra threshold parameter, namely  the parameter $\Thres$ in Algorithm~\ref{algo: CT: CCA},  which corresponds to the co-ordinate thresholding step. We set $\Thres$ in accordance with Theorem~\ref{thm: CT matrix operator norm} and Remark~\ref{remark: CT constants}, with $K$ being $1288\B^4$ and $C_1$ being $50\B^4$. We set $\B$ as in Remark~\ref{remark: CT constants}, that is
  \[\B=\max\{\Lambda_{max}(\Sx),\Lambda_{max}(\Sy),\Lambda_{max}(\Sx^{-1}), \Lambda_{max}(\Sy^{-1})\}.\] 
  The errors incurred by our methods in case A are displayed in Figure \ref{Fig: id: alpha: large n} (for $\alpha$) and Figure \ref{Fig: id: beta: large n} (for $\beta$). Figures \ref{Fig: sparseinv: alpha: large n} and \ref{Fig: sparseinv: beta: large n}, on the other hand, display the errors in the recovery of $\alpha$ and $\beta$, respectively, in case B.

  Now we discuss the main observations from the above plots.   When the sparsity parameter $s$ is considerably low (less than ten in the current settings), the \naive\ SCCA method is sufficient in the sense that the specialized methods do not perform any better. Moreover, the \naive\ method is the most conservative one among all three methods. As a consequence, the associated type I error is always small, although the type II error of the \naive\ method grows faster than any other method.  The specialized methods are able to improve the type II error at the cost of higher type I error. At a higher sparsity level, the specialized methods can outperform the \naive\ method in terms of the Hamming error, however. This is most evident when the setting is also complex, i.e., the signal is small, or the underlying covariance matrices are not identity. In particular, Figure~\ref{Fig: id: alpha: large n} and \ref{Fig: sparseinv: alpha: large n}  entail  that when the signal strength is moderate and the sparsity is high, the cleaned SCCA has the lowest hamming error. In the small signal case, however, CT exhibits the best hamming error as $s/\sqn$ increases; cf. Figure~\ref{Fig: id: beta: large n} and \ref{Fig: sparseinv: beta: large n}. 
  
   The Type I error of CT can be slightly improved if the sparsity information can be incorporated during the  thresholding step. We simply replace $\cut$ by the maximum of $\cut$ and the $s$-th largest element of $\widehat V^{clean}$, where the latter is as in Algorithm \RecoverSupp. See, e.g.,   Figure \ref{Fig: all error}, which entails that this modification reduces the Hamming error of the CT algorithm in case A. 
  Our empirical analysis hints that the CT algorithm has potential for improvement from the implementation perspective. In particular, it may be desirable to obtain a more efficient procedure for choosing  $\cut$ in a systematic way. However, such a detailed numerical analysis is beyond the scope of the current paper and will require further modifications of the initial methods for estimation of $\alpha,\beta$ both for scalability and finite sample performance reasons. We keep these explorations as important future directions.

      It is  natural to wonder what is the effect of cleaning via \RecoverSupp\ on SCCA.
      As mentioned earlier, during our simulations we observed that a cleaning step  generally improves the type II error of the \naive\ SCCA, but it also increases the type I error. In terms of the combined measure, i.e., the Hamming error, it turns out that  cleaning  does have an edge at higher sparsity levels in case B; cf.  Figure \ref{Fig: sparseinv: alpha: large n} and Figure \ref{Fig: sparseinv: beta: large n}. However, the scenario is different in case A. Although Figures~\ref{Fig: id: alpha: large n} and \ref{Fig: id: beta: large n} indicate that  almost no cleaning occurs at the set threshold level of one, we saw that cleaning happens at lower threshold levels.  However, the latter does not improve the overall Hamming error of \naive\ SCCA.  The consequence of cleaning may be different for other SCCA methods.
        
        To summarize, when the  sparsity is low, support recovery using the  \naive\ SCCA is probably as good as the specialized methods. However, at higher sparsity level, specialized support recovery methods may be preferable. Consequently, the precise analysis of the apparently \naive\ SCCA will indeed be an interesting future direction. 

  \section{Discussion}
In this paper, we have discussed rate optimal behavior of information theoretic and computational limits of the joint support recovery  for the sparse canonical correlation analysis problem. Inspired by recent results in the estimation theory of sparse CCA, a flurry of results in sparse PCA, and related developments based on low-degree polynomial conjecture -- we are able to paint a complete picture of the landscape of support recovery for SCCA. For future directions, it is worth noting that our results are so far not designed to recover $D(v_i)$ for individual $i\in[r]$ separately (and hence the term joint recovery). Although this is also the case for most state of the art in the sparse PCA problem (results often exist only for the combined support \cite{deshpande2014} or the single spike model where $r=1$ \cite{wainwright2009}), we believe that it is an interesting question for deeper explorations in the future. Moreover, moving beyond asymptotically exact recovery of support to more nuanced metrics (e.g., Hamming Loss) will also require new ideas worth studying. Finally, it remains an interesting question to pursue whether polynomial time support recovery is possible in the $\sqrt{n/\log{(p+q)}}\ll s_x,s_y\ll \sqrt{n}$ regime using a CT type idea -- but for unknown yet structured high dimensional nuisance parameters $\Sigma_x,\Sigma_y$.

\FloatBarrier
 \appendices
 \section{Full version of \RecoverSupp}
 \label{sec: full version of recoversupp}
 \begin{algorithm}[H]
\caption{\RecoverSupp : simultaneous support recovery of $U$ and $V$}
\label{algo: recoversupp full}

\begin{algorithmic} 
\REQUIRE \begin{enumerate}
\item Preliminary estimators $\widehat U^{(1)}$ and   $\widehat V^{(1)}$  of $U$ and $V$, and estimators $\hfx$ and $\hf^{(1)}$ of $\Sx^{-1}$ and $\Sy^{-1}$, respectively. All are based on sample $O_1=(x_i,y_i)_{i=1}^{[n/2]}$.
\item Estimator $\hSxy^{(2)}$ of $\Sxy$ based on sample  $O_2=(x_i,y_i)_{i=[n/2]+1}^{n}$.
\item Threshold levels $\cut_x,\cut_y>0$ and rank $r\in\NN$.
\end{enumerate}
\ENSURE $\widehat D(U)$ and $\widehat D(V)$,  estimators of $D(U)$ and ${D(V)}$, respectively.
\STATE 
\begin{enumerate}
\item \textbf{Cleaning}: $\widehat V^{clean}\leftarrow \hf^{(1)}{{\hSyx}}^{(2)} \widehat U^{(1)}$; $\widehat U^{clean}\leftarrow \hfx{{\hSxy}}^{(2)} \widehat V^{(1)}$.
\item \textbf{Threshold:} Compute 
\[ \widehat D(U):= \{i\in[p]: |\widehat U^{clean}_{ik}|>\cut_x\text{ for some }k\in[r]\}\]
and
\[\widehat D(V):= \{i\in[q]: |\widehat V^{clean}_{ik}|>\cut_y\text{ for some }k\in[r]\}.\] 
\end{enumerate}
\RETURN $ \widehat D(U)$ and $ \widehat D(V)$.
 \end{algorithmic}
\end{algorithm}
In Algorithm \ref{algo: recoversupp full}, we used different cut-offs for estimating $\widehat D(U)$ and $\widehat D(V)$, which are $\cut_x$ and $\cut_y$, respectively. In practice, one can choose the same threshold $\cut$ for both of them. 
 \section{Proof preliminaries}
  The Appendix collects the proof of all our theorems and lemmas. This section introduces some new notations and
collects  some facts, which are  used repeatedly in our proofs.
\subsection{New Notations}
\label{sec: new notaton}

Since the columns of $\Sx^{1/2}U$, i.e.,  $[\Sx^{1/2}U_1,\ldots,\Sx^{1/2}U_r]$  are orthogonal, we can extend it to an orthogonal basis of $\RR^p$, which can also be expressed in the form $[\Sx^{1/2}u_1,\ldots,\Sx^{1/2}u_p]$ since $\Sx$ is non-singular. Let us denote the matrix $[u_1,\ldots,u_p]$ by $\tU$, whose first $r$ columns form the matrix $U$. Along the same line, we can define $\tV$, whose first $q$ columns constitute the matrix $V$.

Suppose $A\in\RR^{p\times q}$ is a matrix. Recall the projection operator defined in \eqref{def:projection operator}.
For any $ S\subset [p]$, we let $A_{S*}$ denote the matrix $\mP_{S\times [q]}\{A\}$. Similarly, for $F\subset[q]$, 
we let $A_{F}$ be the matrix $\mP_{[p]\times F}\{A\}$.
For $k\in\NN$, we define the norms $\|A\|_{k,\infty}=\max_{j\in[q]}\|A_j\|_k$ and $\|A\|_{\infty, k}=\max_{i\in[q]}\|A_i\|_k$. We will use the notation $|A|_{\infty}$ to denote the quantity $\sup_{1\in[p],j\in[q]}|A_{i,j}|$.  

The Kullback Leibler  (KL) divergence between two probability distributions $P_1$ and $P_2$ will be denoted by $KL(P_1\mid P_2)$.
For $x\in\RR$, we let $\floor*{x}$ denote greatest integer less than or equal to $x\in\RR$.
\subsection{Facts on $\mod$}
First, note that since $\bi^T\Sy\bi=1$ by \eqref{opt: sparse CCA} for all $i\in[q]$, we have  $\|\bi\|_2\leq \sqrt{\B}$. Similarly, we can also show that $\|\ai\|_2\leq \sqrt{\B}$.
Second, we note that $\|\Sx^{1/2}U\|_{op}=\|\Sy^{1/2}V\|_{op}=1$, and
 \begin{align}\label{ineq: op norm of Syx}
    |\Syx|_{\infty}\leq &\  \|\Syx\|_{op}= \|\Sy V\Lambda U^T\Sx\|_{op}\nn \\
    \leq &\  \|\Sy^{1/2}\|_{op}\|\Sy^{1/2}V\|_{op}\|\Lambda\|_{op}\|\Sx^{1/2} U\|_{op}\|\Sx^{1/2}\|_{op}\nn\\
    \leq &\ \B
 \end{align}
because the largest element of  $\Lambda$ is not larger than one. 
Since $X_i$'s and $Y_i$'s are Subgaussian, for any random vector $v$ independent of $\mX$ and $\mY$, it follows that  \cite[Lemma 7]{jankova2018}
 \begin{equation}\label{eq: Lemma 7 of sara}
 |(\hSyx-\Syx) v|_{\infty}\leq C_\B\|v\|_2 \sqrt{\frac{\log (p+q)}{n}}
 \end{equation}
with $\PP$ probability $1-o(1)$ uniformly over $\PP\in\mod$.
Also, we can show that $\Phi_0=\Sy^{-1}$ satisfies
\begin{align*}
    \|(\Phi_0)_{k}\|_{1,\infty}  \leq &\   \sqrt{s({\Sx})}\|(\Phi_0)_{k}\|_{2,\infty}\leq \sqrt{s({\Sx})}\|\Phi_0\|_{op}\\
    \leq &\ \sqrt{s({\Sx})} \B,
\end{align*}
  where Cauchy-Schwarz inequality was used in the first step.
 \subsection{General Technical Facts}

\begin{fact}\label{Fact: Frobenius norm inequality}
  For two matrices $A\in\RR^{m\times n}$ and $B\in\RR^{n\times q}$, we have
 $$ \|AB\|_F^2\leq \|A\|_{op}^2\|B\|_F^2,\quad \|AB\|_F^2\leq \|A\|_{F}^2\|B\|_{op}^2$$
  \end{fact}
 
  
\begin{fact}[Lemma 11 of \cite{deshpande2014}]
\label{fact: Gaussian matrix concentration inequality}
 Let $\mZ\in\RR^{n\times p}$ be a matrix with i.i.d. standard normal entries, i.e.,  $Z_{i,j}\sim N(0,1)$. Then for every $t>0$,
 \[\mathbb P(\|\mZ\|_{op}\geq \sqrt{p}+\sqrt{n}+t)\leq \exp(-t^2/2).\]
 As a consequence, there exists an absolute constant $C>0$ such that
 \[\mathbb P\slb \|\mZ\|_{op}\geq \sqrt 2(\sqrt{p}+\sqrt{n})\srb\leq \exp(-C(p+n)).\]
\end{fact}
Recall that for $A\in\RR^{p\times q}$, in Appendix~\ref{sec: new notaton}, we defined  $\|A\|_{1,\infty}$ and $\|A\|_{\infty,1}$ to be the  matrix norms $\max_{j\in[q]}\|A_j\|_1$ and  $\max_{i\in[p]}\|A_{i*}\|_1$, respectively.

The following fact is a Corollary to \eqref{eq: Lemma 7 of sara}.
\begin{fact}
\label{fact: Sup norm of subgaussian data}
Suppose $X$ and $Y$ are jointly subgaussian. Then $|\hSxy-\Sxy|_{\infty}=O_p(\sqrt{\log (p+q)/n})$.
\end{fact}

\begin{fact}[Chi-square tail bound]
\label{fact: Chi square tail probability}
Suppose $\mathbb Z_1,\ldots,\mathbb Z_k\iid N(0,1)$. Then for any $y>5$, we have
\[\mathbb P\slb \sum_{l=1}^k\mathbb Z_l^2\geq y k \srb\leq \exp(-yk/5).\]
\end{fact}
\begin{proof}[Proof of Fact~\ref{fact: Chi square tail probability}]
Since $Z_l$'s are independent standard Gaussian random variables, by tail bounds on Chi-squared random variables (The form below is from Lemma 12 of \cite{deshpande2014}), 
\[\mathbb P\slb \sum_{l=1}^k\mathbb Z_l^2\geq k+2\sqrt{kx}+2x \srb\leq \exp(-x).\]
Plugging in $x=yk$, we obtain that
\[\mathbb P\slb \sum_{l=1}^k\mathbb Z_l^2\geq (1+2\sqrt{y}+2y)k \srb\leq \exp(-yk),\]
which implies for $y>1$, 
\[\mathbb P\slb \sum_{l=1}^k\mathbb Z_l^2\geq 5yk \srb\leq \exp(-yk),\]
which can be rewritten as 
\[\mathbb P\slb \sum_{l=1}^k\mathbb Z_l^2\geq y k \srb\leq \exp(-yk/5)\]
 as long as $y>5$. 
\end{proof}

  \section{Proof of Theorem~\ref{thm: support recovery: sub-gaussians}}
 \label{sec: Proof of support rec: subgaussian} 
 
 For the sake of simplicity, we denote $\widehat U^{(1)}$, $\hSxy^{(2)}$, and $\hf^{(1)}$ by $\widehat U$, $\hSxy$, and $\hf$, respectively. The reader should keep in mind that $\widehat U$ and $\hf$ are independent of $\hSxy$ and $\hf$ because they are constructed from a different sample. Next, using Condition~\ref{Assumption on estimators}, we can show that there exists $(w_i,\ldots,w_p)\in\{\pm 1\}^p$ so that
 \begin{align*}
 \inf_{\PP\in\mod}\PP & \slb \max_{i\in[r]}\sbl (w_i\hai-\ai)^T\Sx(w_i\hai-\ai)\sbl\\
 &\ <\errs^2\srb
   \to  1.
 \end{align*}
as $n\to\infty$.
 Without loss of generality, we assume $w_i=1$ for all $i\in[r]$. The proof will be similar for general $w_i$'s.  Thus
 \begin{align}\label{assump: proxy: estimator}
     \inf_{\PP\in\mod}\PP\slb \max_{i\in[r]}\sbl (\hai-\ai)^T\Sx(\hai-\ai)\sbl<\errs^2\srb \to 1
 \end{align}
 Therefore $\|\hai-\ai\|_2\leq \errs\sqrt{\B}$ for all $i\in[r]$ with $\PP$ probability tending to one.
 
 Now we will collect some facts which will be used during the proof. Because $\hai$ and $\hSyx$ are independent, \eqref{eq: Lemma 7 of sara} implies that
\begin{equation*}
 |(\hSyx-\Syx) \hai|_{\infty}\leq C_\B\|\hai\|_2 \sqrt{\frac{\log (p+q)}{n}}.
 \end{equation*} 
Using  \eqref{assump: proxy: estimator}, we obtain that
 $\|\hai\|_2\leq \|\hai-\ai\|_2+\|\ai\|_2\leq \sqrt{B}(\errs+1)$.
Because $\errs<\B^{-1}\leq 1$, we have
 \begin{align}\label{eq: Lemma 7 of sara for our case}
\MoveEqLeft \inf_{\PP\in\mod}\PP\lb \max_{i\in[r]}|(\hSyx-\Syx) \hai|_{\infty}\leq  C_\B\sqrt{\frac{\log (p+q)}{n}}\rb\nn\\
 &\ =1-o(1).
 \end{align}
 Noting \eqref{ineq: op norm of Syx} implies
 $|\Syx\hai|_\infty\leq \|\Syx\|_{op}\|\hai\|_2\leq 2\B^{3/2}$, and that $\log(p+q)=o(n)$, using \eqref{eq: Lemma 7 of sara for our case},
we obtain that
\begin{align}\label{intheorem: bound on sxy times hai}
    \max_{i\in[r]}|\hSyx\hai|_\infty\leq |(\hSyx-\Syx)\hai|_\infty+|\Syx\hai|_\infty\leq 3\B^{3/2}
\end{align}
with $\PP$ probability $1-o(1)$.

Now we are ready to prove Theorem~\ref{thm: support recovery: sub-gaussians}. We will denote the columns of $\widehat V^{clean}_n$ by $\tbi$ for $i\in[r]$.
Because $\Lambda_i(\bi)_k=e_k^T\Sy^{-1}\Syx\ai$, it holds that
\begin{align*}
    ( \tbi)_k-\Lambda_i(\bi)_k = &\ e_k^T(\hf-\Phi_0)\hSyx \hai\\
    &\ + e_k^T\Phi_0(\hSyx-\Syx) \hai\\
    &\ + e_k^T\Phi_0\Syx(\hai-\ai)
\end{align*}
leading to
\begin{align*}
 |( \tbi)_k-\Lambda_i(\bi)_k|\leq &\ \underbrace{  |e_k^T(\hf-\Phi_0)\hSyx \hai|}_{T_1(i,k)}\\
 &\ + \underbrace{  |e_k^T\Phi_0(\hSyx-\Syx) \hai|}_{T_2(i,k)}\\
 &\ + \underbrace{ |e_k^T\Phi_0\Syx(\hai-\ai)|}_{T_3(i,k)}.
\end{align*}
Handling the term $T_2$ is the easiest because 
 \begin{align*}
\MoveEqLeft \max_{i\in[r],k\in[q]}  T_2(i,k)\\
\leq &\  \|\Phi_0\|_{1,\infty}|(\hSyx-\Syx) \hai|_{\infty}\\
 \leq &\ C_\B\sqrt{\frac{\sSy\log(p+q)}{n}}
 \end{align*}
with $\PP$ probability $1-o(1)$ uniformly over $\mod$,
where we used \eqref{eq: Lemma 7 of sara for our case} and the fact that $\|\Phi_0\|_{1,\infty}\leq\sqrt{s({\Sy^{-1}})} \B$. The difference in cases (A), (B), (C) arises only due to different bounds on $T_1(i,k)$ in these cases. We demonstrate the whole proof only for case (A). For the other two cases, we only discuss  the analysis of $T_1(i,k)$ because the rest of the proof remains identical in these cases.

\subsubsection{Case (A)}
Since we have shown in \eqref{intheorem: bound on sxy times hai} that $|\hSyx\hai|_\infty\leq 3\B^{3/2}$, we calculate
\begin{align*}
 \max_{i\in[r],k\in[q]}  T_1(i,k)\leq &\  \|\hf-\Phi_0\|_{1,\infty}\max_{i\in[r]}|\hSyx\hai|_\infty\\
 \leq &\  3\B^{3/2}\Cpr \sSy\sqrt{\frac{\log q}{n}}
\end{align*}
with $\PP$ probability tending to one, uniformly over $\mod$,
where to get the last inequality, we also used the bound on $\|\hf-\Phi_0\|_{\infty,1}$ in case (A).

  Finally, for $T_3$, we notice that
  \begin{align*}
     T_3(i,k)=&\ \big|e_k^T\Phi_0\Syx (\hai-\ai)\big|\\
     = &\ \bl e_k^T\sum_{j=1}^{r}\Lambda_j\bj\aj^T\Sx(\hai-\ai)\bl\\
      \leq &\ \max_{j\in[r]}\big|(\bj)_k\big|\bl\sum_{j=1}^{r}\aj^T\Sx(\hai-\ai)\bl 
  \end{align*}
     since $\Lambda_1\leq 1$. Since $(\bj)_k=V_{kj}$, it is clear that $T_3(i,k)$ is  identically zero if $k\notin D(V)$. Otherwise, 
    Cauchy Schwarz inequality implies, 
    \begin{align*}
      \MoveEqLeft  \bl\sum_{j=1}^{r}\aj^T\Sx(\hai-\ai)\bl\\
        \leq &\ \sqrt{r}\lb\sum_{j=1}^r (\aj^T\Sx(\hai-\ai))^2\rb^{1/2}\\
        \leq &\  \sqrt{r}\|\Sx^{1/2}(\hai-\ai)\|_2
    \end{align*}
because $\Sx^{1/2}\aj$'s are orthogonal. Thus  
     \begin{align*}
    \max_{i\in[r],k\in D(V)}|T_3(i,k)|\leq \sqrt{r}\max_{j\in[r]}\big|(\bj)_k\big|\errs. \end{align*}

 Now we will combine the above pieces together. Note that
 \begin{align}\label{intheorem: Thm 1: e-n}
    \MoveEqLeft \max_{i\in[q]}\max_{k\in[r]}(|T_1(i,k)|+|T_2(i,k)|\nn\\
     \leq&\ C_\B\underbrace{\Cpr \sSy\sqrt{\frac{\log(p+q)}{n}}}_{\e_n}.
 \end{align}
 For $k\notin D(V)$, denoting the $i$-th column of $\widehat V^{clean}$ by $\tbi$ we observe that,
 \begin{align}\label{eq: inlemma: general: the support recovery}
\max_{k\notin D(V)} \max_{i\in[r]} |\widehat V_{ki}^{clean}|=&\ \max_{k\notin D(V)} \max_{i\in[r]} |( \tbi)_k|\nn\\
\leq &\ \max_{i\in[q]}\max_{k\in[r]}(|T_1(i,k)|+|T_2(i,k)|)\nn\\
\leq &\  C_\B{\e_n}
 \end{align}
 with $\PP$ probability $1-o(1)$ uniformly over $\PP\in\mod$.
 On the other hand,
if  $k\in D(\bi)$, then we have for all $i\in[r]$,
\begin{align*}
    |(\tbi)_k|> &\ \Lambda_i|(\bi)_k|-\sqrt{r}\max_{j\in[r]}\big|(\bj)_k\big|\errs\\
    &\ -\max_{i\in[q]}\max_{k\in[r]}(|T_1(i,k)|+|T_2(i,k)|),
\end{align*}
which implies
\[\max_{i\in[r]}|\widehat V_{ki}^{clean}|>\max_{i\in[r]}\Lambda_i|(\bi)_k|-\sqrt{r}\max_{i\in[r]}\big|(\bi)_k\big|\errs-C_\B\e_n.\]
Since $\errs<\B^{-1}/(2\sqrt r)$ and $\B^{-1}<\min_{i\in[r]}\Lambda_i$, we have
\begin{align*}
  \MoveEqLeft  \max_{i\in[r]}\Lambda_i|(\bi)_k|-\sqrt{r}\max_{i\in[r]}\big|(\bi)_k\big|\errs\\
    &\ > (\B^{-1}-\sqrt r\errs)\max_{i\in[r]}\big|(\bi)_k\big|\\
    &\ >\B^{-1}\max_{i\in[r]}\big|(\bi)_k\big|/2.
\end{align*}
Thus, noting $V_{ki}=(\bi)_{k}$, we obtain that
\begin{align*}
   \min_{k\in D(V)}\max_{i\in[r]}|(\tbi)_k|=&\ \min_{k\in D(V)}\max_{i\in[r]}|\widehat V_{ki}^{clean}|\\
   >  &\ \min_{k\in D(V)}\max_{i\in[r]}\big|V_{ki}\big|/(2\B)-C_\B\e_n 
\end{align*}
with $\PP$ probability $1-o(1)$ uniformly over $\PP\in\mod$.
Suppose $C_\B'=2\B C_\B$. Note that
\[\min_{k\in[p]}\max_{i\in[r]}\big|(\bi)_k\big|=\ratio C_\B' \e_n\]
where $\ratio>2$. Then with $\PP$ probability $1-o(1)$ uniformly over $\PP\in\mod$,
\[\min_{k\in D(V)}\max_{i\in[r]}\widehat V_{ki}^{clean}>(\ratio-1)C_\B' \e_n/(2\B).\] 
 This, combined with \eqref{eq: inlemma: general: the support recovery}
implies setting $\cut\in[C_\B'\e_n/(2\B), (\ratio-1)C_\B' \e_n/(2\B)]$ leads to full support recovery with $\PP$ probability $1-o(1)$. The proof of the first part follows. 


\subsubsection{Case (B)}
\label{secpf: case B: Thm 1}
  In the Gaussian case, we resort to the hidden variable representation of $X$ and $Y$ due to \cite{bach2005}, which enables sharper bound on the term $T_1(i,k)$.
 Suppose $\mZ\sim N_r(0,I_r)$  where $r$ is the rank of $\Sxy$. Consider $Z_1\sim N_p(0, I_p)$ and $Z_2\sim N_q(0,I_q)$ independent of $\mZ$. Then  
 $X$ and $Y$ can be represented as
 \begin{equation}\label{model: 3}
     X=\W_1 Z+\ \H_1 Z_1\quad \text{ and }\quad Y=\W_2 Z+\H_2 Z_2,
 \end{equation}
 where
 \[\W_1=\Sx U\Lambda^{1/2},\ \W_2=\Sy V\Lambda^{1/2},\ \H_1=(\Sx-\W_1\W_1^T)^{1/2},\] 
 and
 \[  \H_2=(\Sy-\W_2\W_2^T)^{1/2}.\]
 Here $(\Sx-\W_1\W_1^T)^{1/2}$ is well defined because
 \[\Sx-\W_1\W_1^T=\Sx\tU(I_p-\xLambda)\tU^T\Sx,\]
 where $\xLambda$ is a $p\times p$ diagonal matrix whose first $p$ elements are $\Lambda_1,\ldots,\Lambda_r$, and they rest are zero. Because $\Lambda_1\leq 1$, we have 
 \[(\Sx-\W_1\W_1^T)^{1/2}=\Sx\tU(I_p-\xLambda)^{1/2}\tU^T\Sx.\]
 Similarly, we can show that \[(\Sy-\W_2\W_2^T)^{1/2}=\Sy\tV(I_q-\yLambda)^{1/2}\tV^T\Sy,\]
 where $\yLambda$ is the diagonal matrix whose first $r$ elements are $\Lambda_1,\ldots,\Lambda_r$, and the rest are zero.
 It can be easily verified that
 \[Var(X)=\W_1\W_1^T+\H_1=\Sx,\  Var(Y)=\W_2\W_2^T+\H_2=\Sy,\]
 and
 \[\Sxy=\W_1\W_2^T=\Sx U\Lambda V^T\Sy,\]
 which ensures that the joint variance  of $(X,Y)$ is still $\Sigma$. Also, some linear algebra leads to
  \begin{align}\label{intheorem: 1: ineq: ub on H1 H2 W1 and}
     \max\lbs \|\H_1\|^2_{op},\|\H_2\|^2_{op},\|\W_1\|_{op},\|\W_2\|_{op}\rbs<\B.
  \end{align}

 Suppose we have $n$ independent  realizations of the pseudo-observations $Z_1$, $Z_2$, and $Z$.
 Denote by $\mZ_1$, $\mZ_2$, and $\mZ$, the stacked data matrices with the i-th row as  $(Z_1)_i$, $(Z_2)_i$, and $Z_i$, respectively, where $i\in[n]$. Here we used the term data-matrix although we do not observe $\mZ
$, $\mZ_1$ and $\mZ_2$ directly. Due to the representation  in \eqref{model: 3},
the data matrices $\mX$ and $\mY$ have the form
\[\mX= \mZ\W_1^T+\mZ_1\H_1,\quad \mY=\mZ\W_2^T+\mZ_2\H_2.\]
We can write
 the covariance matrix $\hSxy=\mX^T \mY/n$ as
\begin{align}\label{representation: Sxy: model 3}
    \hSxy=&\ \frac{1}{n}\lbs \W_1 \mZ^T\mZ\W_2^T+\W_1 \mZ^T\mZ_2\H_2+\H_1^T\mZ_1^T\mZ \W_2^T\nn\\
    &\ +\H_1^T\mZ_1^T\mZ_2\H_2\rbs.
\end{align}
Therefore, for any vector $\th_1\in\RR^p$ and $\th_2\in\RR^q$, we have
\begin{align}\label{intheorem: thm 1; expansion of quad}
  \MoveEqLeft \th_1^T(\hSxy-\Sxy) \th_2= \th_1^T\W_1^T\slb \frac{\mZ^T\mZ}{n}-I_r\srb \W_2 \th_2\nn\\
   & +\frac{1}{n}\th_1^T\slb \W_1 \mZ^T\mZ_2\H_2+\H_1^T\mZ_1^T\mZ \W_2^T+\H_1^T\mZ_1^T\mZ_2\H_2\srb \th_2. \end{align}
By Bai-Yin law on eigenvalues of Wishart matrices \cite{Bai-Yin}, there exists abolute constant $C>0$ so that for any $t>1$,
\[P\lb\norm{\frac{\mZ^T\mZ}{n}-I_r}_{op} <t\sqrt{r/n}\rb\geq 1-2\exp(-Ct^2r),\]
which, combined with \eqref{intheorem: 1: ineq: ub on H1 H2 W1 and}, implies 
\begin{align*}
    \inf_{\PP\in\modG}\PP\slb &  \sbl\th_1^T\W_1^T({\mZ^T\mZ}/{n}-I_r) \W_2 \th_2\sbl\\
    \leq &\  t\B^2 \|\th_1\|_2\|\th_2\|_2\sqrt{r/n}\srb\\
   \quad  \geq &\  1-2\exp(-Ct^2r).
\end{align*}
Now we will state a lemma which will be required to control the other terms on the right hand side of \eqref{intheorem: thm 1; expansion of quad}.
\begin{lemma}
\label{lemma: thm 1: op of covariance mat}
Suppose $\mZ_1\in\RR^{n\times p}$ and $\mZ_2\in\RR^{n\times q}$ are independent Gaussian data matrices. Further suppose $x\in\RR^p$ and $y\in\RR^q$ are either deterministic or  independent of both $\mZ_1$ and $\mZ_2$. Then there exists a constant $C>0$ so that for any $t>1$,
\[P\slb \abs{x^T\mZ_1^T\mZ_2y}>t\|x\|_2\|y\|_2\sqrt{n}\srb\leq \exp(-Cn)-\exp(t^2/2).\]
\end{lemma}
The proof of Lemma \ref{lemma: thm 1: op of covariance mat} follows directly  setting $b=1$ in the following Lemma, which is proved in Appendix~\ref{sec: add lemma: operator norm}.
\begin{lemma}
\label{lemma: normal data matrix}
Suppose $\mZ_1\in\RR^{n\times p}$ and $\mZ_2\in\RR^{n\times q}$ are independent standard Gaussian data matrices, and $D\in\RR^{n\times k_1}$ and $B\in\RR^{n\times k_2}$ are deterministic matrices with rank $a$ and $b$, respectively. Let $a\leq b\leq n$. 
Then there exists an absolute constant $C>0$ so that for any $t\geq 0$, the following holds with probability at least $1-\exp(-Cn)-\exp(- t^2/2)$:
\[\|D^T\mZ_1^T\mZ_2B\|_{op}\leq C\|D\|_{op}\|B\|_{op}\sqrt{n}\max\{\sqrt{b},t\}.\]
\end{lemma}
Lemma \ref{lemma: thm 1: op of covariance mat}, in conjunction with \eqref{intheorem: 1: ineq: ub on H1 H2 W1 and}, implies that there exists  an absolute constant $C>0$ so that
\begin{align*}
 \MoveEqLeft \frac{1}{n}\sbl\th_1^T\slb \W_1 \mZ^T\mZ_2\H_2+\H_1^T\mZ_1^T\mZ \W_2^T+\H_1^T\mZ_1^T\mZ_2\H_2\srb \th_2\sbl\\
   \leq &\ t\B^2\|\th_1\|_2\|\th_2\|_2n^{-1/2} 
\end{align*}
with $\PP$ probability at least  $1-\exp(-Cn)-\exp(t^2/2)$ for all  $\PP\in\modG$. Therefore, there exists $C>0$ so that
\begin{align}\label{intheorem 1: prob of quad term}
\MoveEqLeft\PP\slb|\th_1^T(\hSxy-\Sxy) \th_2|
\leq  t\sqrt{r}\B^2\|\th_1\|_2\|\th_2\|_2n^{-1/2}\srb\nn\\
    \geq &\ 1-\exp(-Cn)-\exp(-Ct^2).
\end{align}
for all $\PP\in\modG$. Note that
 \begin{align*}
     T_{1}(i,k)\leq &\  \underbrace{\sbl\slb(\hf)_{k*}-(\Sy^{-1})_{k*}\srb^T(\hSyx-\Syx)\hai\sbl}_{T_{11}(i,k)}\\
     &\ +\underbrace{\sbl\slb(\hf)_{k*}-(\Sy^{-1})_{k*}\srb^T\Syx\hai\sbl}_{T_{12}(i,k)}.
 \end{align*}
Now suppose $\th_1=(\hf)_{k*}-(\Sy^{-1})_{k*}$ and $\th_2=\hai$.  By our assumption,
$\|\th_1\|_2\leq \Cpr \sqrt{\sSy(\log q)/n}$ with $\PP$ probability $1-o(1)$ uniformly across $\PP\in\modG$. We also showed that $\|\hai\|_2\leq 2\sqrt{\B}$.
It is not had to see that
\begin{align}\label{intheorem: B: T12}
    \sup_{i\in[q],k\in[r]}T_{12}(i,k)\leq 2\B^{3/2}\Cpr \sqrt{\sSy(\log q)/n}
\end{align}
with $\PP$ probability $1-o(1)$ uniformly across $\PP\in\modG$.
For $T_{11}$, observe that \eqref{intheorem 1: prob of quad term}  applies because $\th_i=(\hf)_{k*}-(\Sy^{-1})_{k*}$ and $\th_2=\hai$ are independent of $\hSxy$. 
   Thus we can write that for any $t>1$, there exists $C_\B>1$ such that
   \begin{align*}
    \MoveEqLeft   \sup_{\PP\in\modG}\PP\slb |T_{11}(i,k)|
       >  tC_\B\Cpr\sqrt{r\sSy\log q}/n\srb\\
       \leq &\ \exp(-Cn)+\exp(-Ct^2).
   \end{align*}
Applying union bound, we obtain that for any $\PP\in\modG$,
\begin{align*}
    \ \MoveEqLeft\PP\left ( \max_{i\in[q]}\max_{k\in[r]}|T_{11}(i,k)|>\frac{tC_\B\Cpr\sqrt{r\sSy\log q}}{n}\right )\\ &\leq \exp(-Cn+\log (qr))+\exp(-Ct^2+\log(qr)).
\end{align*}
Since $r<q$ and $\log q=o(n)$, setting $t=2\sqrt{\log q}/C$,  we obtain that
\begin{align*}
    \sup_{\substack{\PP\in\\\modG}}\PP\left( \max_{i\in[q]}\max_{k\in[r]}|T_{11}(i,k)|>C_\B\Cpr\sqrt{r\sSy}\frac{\log q}{n}\right )
\end{align*}
is $o(1)$.
Using \eqref{intheorem: Thm 1: e-n} and \eqref{intheorem: B: T12}, one can show that 
\[\e_n=\Cpr \sqrt{\sSy(\log (p+q))/n}\max\{\sqrt{r(\log q)/n}, 1\}\]
in this case.

\subsubsection{Case (C)}
 Note that when $\hf=\Sy^{-1}$, $T_1(i,k)=0$. Therefore, \eqref{intheorem: Thm 1: e-n} implies $\e_n=\sqrt{{\sSy\log(p+q)}/{n}}$  in this case.

\section{Proof of Theorem~\ref{thm: lower bound}}
\label{sec: Proof of IT limit}
Since the proof for $U$ and $V$ follows in a similar way, we will only consider the support recovery of $U$. The proof for both cases follows a common structure. Therefore, we will elaborate the common structure first. 
Since the model $\mod$ is fairly large, we will work with a smaller submodel. Specifically, we will  consider a subclass of  the single spike models, i.e.,  $r=1$. Because we are concerned with only  the  support recovery of the left singular vectors, we fix  $\beta_0$ in $\RR^q$ so that $\|\beta_0\|_2=1$. We also fix $\rho\in(0,1)$ and consider the subset 
$\mE\subset\{\alpha\in\RR^p:\|\alpha\|_2=1\}$. Both $\rho$ and $\mE$  will be chosen later. We  restrict our attention to the submodel $\mM(s_x,s_y,\rho,\mE)$ given by
\begin{align*}
    \lbs &\ \PP\in \mP(1,s_x,s_y,\B)\ :\ \PP\equiv N_{p+q}(0,\Sigma)  \text{ where }\Sigma\\
   &\ \text{ is of the form }\eqref{def: sigma in Theorem 3} \text{ with }\alpha\in\mathcal E,\beta=\bk\rbs,
\end{align*}
where \eqref{def: sigma in Theorem 3} is as follows:
\begin{equation}\label{def: sigma in Theorem 3}
    \Sigma=\begin{bmatrix}
    I_p & \rho \alpha\beta^T\\
    \rho\beta\alpha^T & I_q
    \end{bmatrix}.
\end{equation}
That $\Sigma$ is positive definite for $\rho\in(0,1)$ can be shown either using elementary linear algebra or the  the hidden variable representation \eqref{model: 3}.
 During the proof of part (B), we will choose $\mE$ so that $\Sig_x^2\leq (B^2-1)(\log(p-s_x))/8n$, which will ensure that $\mM(s_x,s_y,\rho,\mE)\subset\modsig$ as well. 
 
 Note that for $\PP\in\mM(s_x,s_y,\rho,\mE)$, $U$ corresponds to $\alpha$, and hence $D(U)=D(\alpha)$.
 Therefore for the proof of both parts, it suffices to show that for any decoder $\widehat D_{\alpha}$ of $D(\alpha)$,
\begin{equation}\label{intheorem: thm 3: end goal}
    \inf_{\widehat D_{\alpha}}\sup_{\PP\in \mM(s_x,s_y,\mE)}\PP\slb \widehat D_{\alpha}\neq D(\alpha)\srb>1/2.
\end{equation}
In both of the proofs, our $\mE$ will be a finite set.  Our goal is to choose $\mE$  so that $\mM(s_x,s_y,\rho,\mE)$ is structurally rich enough to guarantee \eqref{intheorem: thm 3: end goal}, yet lends itself to easy computations.
The guidance for choosing $\mE$ comes from our main technical tool for this proof, which is Fano's inequality.  We use  the  verson of Fano's inequality  in \cite{yatracos1988} (Fano's Lemma). Applied to our problem, this inequality yields
   \begin{align}\label{ineq: fano: pre}
  \MoveEqLeft \inf_{\widehat{D}_\alpha}  \sup_{\PP\in  \mM(s_x,s_y,\rho,\mE)}\PP\slb\widehat D_{\alpha}\neq D(\alpha)\srb\nn\\
   \geq &\  1-\dfrac{\frac{\sum_{\PP_1,\PP_2\in\mM(s_x,s_y,\rho,\mE)}KL(\PP_1^n|\PP_2^n)}{|\mM(s_x,s_y,\rho,\mE)|^2}+\log 2}{\log(|\mM(s_x,s_y,\rho,\mE)|-1)},  
   \end{align}
   where $\PP_n$ denotes the product measure corresponding to $n$ i.i.d. observations from $\PP$. We also have the following result for product measures,
   $KL(\PP_1^n|\PP_2^n)=nKL(\PP_1|\PP_2)$. Moreover, when
 $\PP_1,\PP_2\in\mM(s_x,s_y,\rho,\mE)$ with left singular vectors $\alpha_1$ and $\alpha_2$, respectively, 
 \[KL(\PP_1|\PP_2)=\log\frac{\text{det}(\Sigma_2)}{\text{det}(\Sigma_1)}-(p+q)+Tr(\Sigma_2^{-1}\Sigma_1),\]
 where $\text{det}(\Sigma_1)=\text{det}(\Sigma_2)=1-\rho^2$ by  Lemma~\ref{Lemma: determinant of Sigma}, and
 \[-(p+q)+Tr(\Sigma_2^{-1}\Sigma_1)=\frac{2\rho^2}{1-\rho^2}\lb1-(\alpha_1^T\alpha_2)\|\bk\|_2^2\rb\]
 by Lemma~\ref{Lemma: trace }.
 Noting $\alpha_1$, $\alpha_2$, and $\bk$ are unit vectors, we derive
 $
  KL(\PP_1|\PP_2)=\rho^2(\|\alpha_1-\alpha_2\|^2_2)/(1-\rho^2)$.
 Therefore, in our case, \eqref{ineq: fano: pre} reduces to
 \begin{align}\label{ineq: fano}
  \MoveEqLeft \inf_{\widehat{D}_\alpha}  \sup_{\PP\in  \mM(s_x,s_y,\rho,\mE)}\PP\slb\widehat D_{\alpha}\neq D(\alpha)\srb\nn\\
   \geq &\  1-\dfrac{n\rho^2\sup_{\alpha_1,\alpha_2\in\mE}\|\alpha_1-\alpha_2\|^2/(1-\rho^2)+\log 2}{\log(|\mE|-1)}. 
   \end{align}
 Thus, to ensure the right hand side of \eqref{ineq: fano} is non-negligible, the key is to choose $\mE$ so that the  $\alpha$'s in $\mE$ are  close in $l_2$ norm, but $|\mE|$ is sufficiently large. Note that the above ensures that distinguishing the $\alpha$'s in $\mE$ is difficult.
 
 \subsection{Proof of part (A)}
 Note that our main job is to choose $\mE$ and $\rho$ suitably.
 Let us denote 
   \[\alk=(\underbrace{1/\sqrt s_x,\ldots,1/\sqrt s_x}_{s_x\text{ many }},\underbrace{0,\ldots,0}_{p-s_x\text{ many }}).\]
   We generate a class of $\alpha$'s by replacing one of the $1/\sqrt s_x$'s in $\alpha_0$ by $0$, and  one of the zero's in $\alk$ by $1/\sqrt s_x$. A typical  $\alpha$ obtained this way looks like 
   \[\alpha=\slb \underbrace{1/\sqrt s_x,\ldots,\textcolor{red}{\mathbf 0},\ldots 1/\sqrt s_x}_{s_x\text{ many }},\underbrace{0,\ldots,\textcolor{red}{\mathbf{1/\sqrt{s_x}}},\ldots,0}_{p-s_x\text{ many }}\srb.\]
 Let $\mE$ be the class, which consists of $\alk$, and all such resulting $\alpha$'s. Note that  $|\mE|=s_x(p-s_x)$, and  $\alpha_1,\alpha_2\in\mE$ satisfy
    \[\|\alpha_1-\alpha_2\|^2_2\leq \|\alpha_1-\alpha_0\|^2_2+\|\alpha_2-\alpha_0\|^2_2\leq 4 s_x^{-1}.\]
 Because $p>s_x>1$, we have \[\log(s_x(p-s_x)-1)\geq \log(p-s_x).\] Therefore,  \eqref{ineq: fano} leads to
 \begin{align*}
    \MoveEqLeft  \inf_{\widehat{D}_\alpha}  \sup_{\PP\in  \mM(s_x,s_y,\rho,\mE)}\PP\slb\widehat D_{\alpha}\neq D(\alpha)\srb\\
      \geq&\    1- \frac{4\rho^2 n s_x^{-1}/(1-\rho^2)+\log 2}{\log (p-s_x)},
 \end{align*}
which is bounded below by $1/2$  whenever
   \[s_x>\frac{8\rho^2 n}{(1-\rho^2)\{\log(p-s_x)-\log 4\}},\]
   which follows if 
   \[s_x>\frac{16\rho^2 n}{(1-\rho^2)\log(p-s_x)}\]
 because $4=\sqrt{16}<\sqrt{p-s_x}$. To get the best bound on $s_x$, we choose the value of $\rho$ which minimizes $\rho^2/(1-\rho^2)$ for $\PP\in\mod$, that is $\rho=1/\B$.
Plugging in $\rho=1/\B$, the proof follows.
\subsection{Proof of part (B)}
Suppose each $\alpha\in\mE$  is of the following form 
\[\alpha=\slb \underbrace{ b,\ldots, b}_{s_x-1\text{ many }},\underbrace{0,\ldots,0,\textcolor{red}{z},0,\ldots,0}_{p-s_x+1\text{ many }}\srb.\]
We fix $z\in(0,1)$, and hence $b=\sqrt{(1-z^2)/(s_x-1)}$ is also fixed.  We will choose the value of $\rho$ and $z$ later so that $\modsig\supset \mM(s_x,s_y,\rho,\mE)$.
Since $z$ is fixed, such an $\alpha$ can be chosen in $p-s_x+1$ ways. Therefore $|\mE|=p-s_x+1$. Also note that for $\alpha,\alpha'\in\mE$, 
$\|\alpha-\alpha'\|_2^2\leq 2z^2$.
Therefore \eqref{ineq: fano} implies
\begin{align}\label{ineq: fano: 2}
  \MoveEqLeft \inf_{\widehat{D}_\alpha}  \sup_{P\in  \mM(s_x,s_y,\rho,\mE)}P\slb\widehat D_{\alpha}
   \neq  D(\alpha)\srb\\
   \geq &\ 1-\dfrac{2n\rho^2z^2/(1-\rho^2)+\log 2}{\log(p-s_x)},  
  \end{align}
 which is greater than $1/2$  whenever
 \[z^2<\frac{1-\rho^2}{4n\rho^2}\log(\frac{p-s_x}{4}),\]
which holds if  \[z^2=\frac{1-\rho^2}{8n\rho^2}\log(p-s_x)\]
 because $16<p-s_x$.
 To get the best bound on $z$, we choose the value of $\rho$ for $\PP\in\mod$ which maximizes $(1-\rho^2)/\rho^2$, that is $\rho=1/\B$. Thus  \eqref{intheorem: thm 3: end goal} is satisfied when $\rho=1/\B$, and $\mE$ corresponds to \[z^2=(\B^2-1)\log(p-s_x)/(8n).\] Since the minimal signal strength $\Sig_x$ for any $\PP\in\mM(s_x,s_y,\B^{-1},\mE)$ equals $\min(z,b)\leq z$, we have $\modsig\supset \mM(s_x,s_y,\B^{-1},\mE)$, which
 completes the proof.

\section{Proof of Theorem~\ref{thm: low deg}}
\label{sec; proof of low deg}
We first introduce some notations and terminologies that are required for the proof.
For $\myv\in\ZZ^m$, and $x\in \RR^m$, we denote $\myv!=\prod_{i=1}^{m}\myv_i!$ and  $x^\myv=\prod_{i=1}^{m}x_i^{\myv_i}$. In low-degree polynomial literature, when $\myv\in\ZZ^m$,  the notation $|\myv|$ is commonly used to denote the sum $\sum_{i=1}^m \myv_i$ for sake of simplicity. We also follow the above convention. 
Here the notation $|\cdot|$ should not be confused with the  absolute value of real numbers.
   Also, for any function $f:\RR^m\mapsto\RR$, $\myv\in\ZZ^m$, and $t=(t_1,\ldots,t_m)$, we denote 
   \[\partial_t^{\myv} f(t)=\frac{\partial^{|\myv|}}{\partial t_1^{\myv_1}\cdots\partial t_r^{\myv_r}}f(t).\]
  We will also use the shorthand notation $\E_{\pi}$ to denote $\E_{\alpha\sim\pi_x,\beta\sim\pi_y}$ sometimes.

   Our analysis relies on the Hermite polynomial, which we will discuss here very briefly. For a  detailed  account on the Hermite polynomials, see Chapter V of  \cite{szego1939}. The univariate Hermite polynomials of degree $k$ will be denoted by $h_k$. For $k\geq 0$, the univariate Hermite polynomials $h_k:\RR\mapsto\RR$ are defined recursively as follows:
   \begin{align*}
      \MoveEqLeft  h_0(x)=1,\quad h_1(x)=xh_0(x),\quad\ldots,\\
      &\ h_{k+1}(x)=xh_k(x)-h_k'(x).
   \end{align*}
   The normalized univariate Hermite polynomials are given by $\widehat h_k(x)=h_k(x)/\sqrt{k!}$. The univariate Hermite polynomials  form an orthogonal basis of $L_2(N(0,1))$. 
 For $\myv\in\ZZ^m$, the $m$-variate Hermite polynomials are given by $H_{\myv}(y)=\prod_{i=1}^{m}h_{\myv_i}(y_i)$, where $y\in\RR^m$. The normalized version $\widehat H_{\myv}$ of  $H_{\myv}$ equals 
   $H_{\myv}/\sqrt{\myv!}$. The polynomials $\widehat H_{\myv}$'s form an orthogonal basis of $L_2(N_m(0,I_m))$. We denote by $\Pi_n^{\leq D_n}$ the linear span of all $n(p+q)$-variate Hermite polynomials of degree at most $D_n$.
     Since $\LL_n^{\leq D_n}$ is the projection of $\LL_n$ on $\Pi^{\leq D_n}$, it then follows that
   \begin{equation}\label{eq: low deg: projection}
    \|\LL_n^{\leq D_n}\|^2_{L_2(\QQ_n)}=\sum_{\substack{\myv\in \ZZ^{n(p+q)}\\ |\myv|\leq D_n}}\langle \LL_n, \widehat H_\myv\rangle_{L^2(\QQ_n)}^2.   
   \end{equation}
From now on, the degree-index vector $w$ of $\widehat H_w$ or $H_w$ will be assumed to lie in $\ZZ^{n(p+q)}$. We will partition $\myv$ into $n$ components, which gives $\myv=(\myv_1,\ldots,\myv_n)$, where $\myv_i\in\ZZ^{p+q}$ for each $i\in[n]$. Clearly,  $i$ here  corresponds to the $i$-th observation. We also separate each $w_i$ into two parts $\myv_i^x\in \ZZ^p$ and $\myv_i^y\in\ZZ^q$ so that $\myv_i=(\myv_i^x,\myv_i^y)$. We will also denote $w^x=(w^x_1,\ldots,w^x_n)$, and $w_y=(w^y_1,\ldots,w^y_n)$. Note that $w^x\in\ZZ^{np}$ and $w^y\in\ZZ^{nq}$, but $w\neq (w^x,w^y)$ in general, although $|w|=|w^x|+|w^y|$.

Now we state the main lemmas which yields the value of $\|\LL_n^{\leq D_n}\|^2_{L_2(\QQ_n)}$. The first lemma,  proved in Appendix~\ref{sec: add lemmas: low-deg}, gives the form of the inner products $ \langle \LL_n, \widehat H_\myv\rangle_{L^2(\QQ_n)}$.
   
   \begin{lemma}\label{lemma: low deg: dot product}
   Suppose $\myv$ is as defined above and $\LL_n$ is as in \eqref{def: def of LL n}. Then it holds that
   \begin{align*}
    \MoveEqLeft  \langle \LL_n, \widehat H_\myv\rangle_{L^2(\QQ_n)}^2\\
    =&\ \begin{cases}\frac{\B^{-|\myv|}}{\myv!}\lbs\E_{\pi}\lbt 1\{\|\alpha\|_2\|\beta\|_2<\B\}\alpha^{\sum_{i=1}^n\myv^{x}_i}\beta^{\sum_{i=1}^{n}\myv^{y}_i}\rbt\rbs^2  \\
    \times\lb\prod_{i=1}^n{|\myv^{x}_i|!}\rb^2 \quad \text{ if }\quad |\myv^{x}_i|=|\myv^{y}_i|\text{ for all } i\in[n],\\
  0 \quad \text{ o.w.}\end{cases} 
   \end{align*}
  Here the priors $\pi_x$ and $\pi_y$ are the Rademacher priors defined in \eqref{prior: radamander}.
   \end{lemma}
  Our next lemma uses Lemma~\ref{lemma: low deg: dot product} to give the form of $\|\LL_n^{\leq D_n}\|^2_{L_2(\QQ_n)}$. This lemma uses replicas of $\alpha$ and $\beta$. Suppose $\alpha_1,\alpha_2\sim\pi_x$ and $\beta_1,\beta_2\sim\pi_y$ are all independent Rademacher priors, where $\pi_x$ and $\pi_y$ are defined as in \eqref{prior: radamander}. We overload notation, and use $\E_\pi$ to denote the expectation under $\alpha_1$, $\alpha_2$, $\beta_1$, and $\beta_2$.
     \begin{lemma}\label{lemma: low deg: projection}
     Suppose $\myW$ is the indicator function of  the event
 $\{\|\alpha_1\|_2\|\beta_1\|_2<\B,\ \|\alpha_2\|_2\|\beta_2\|_2<\B\}$. Then 
    For any $D_n\in\NN$, $\|\LL_n^{\leq D_n}\|^2_{L_2(\QQ_n)}$ equals
   \[ \E_{\pi}\lbt W\sum_{d=0}^{\floor*{D_n/2}}{d+n-1\choose d}\lbs \B^{-2}(\alpha_1^T\alpha_2)(\beta_1^T\beta_2)\rbs^d\rbt.\]
   \end{lemma}
 The proof of Lemma~\ref{lemma: low deg: projection} is also deferred to 
   Appendix~\ref{sec: add lemmas: low-deg}.
 We remark in passing that the  negative binomial series expansion yields 
 \begin{equation}\label{expansion: negative binomial}
     (1-x)^{-n}=\sum_{d=0}^{\infty}{n+d-1\choose d}x^d,\quad\text{for } |x|<1,
 \end{equation}
   whose  $D_n$-th order truncation equals
   \[\lb(1-x)^{-n}\rb^{\leq D_n}=\sum_{d=0}^{D_n}{n+d-1\choose d}x^d.\]
  Note that  $W$ is nonzero if and only if $\|\alpha_1\|_2\|\beta_1\|_2<\B$ and $\|\alpha_2\|_2\|\beta_2\|_2<\B$, which, by Cauchy Schwarz inequality, implies 
  \[|(\alpha_1^T\alpha_2)(\beta_1^T\beta_2)|<\B^2.\]
  Thus $|\B^{-2}(\alpha_1^T\alpha_2)(\beta_1^T\beta_2)|<1$ when $W=1$. Hence Lemma~\ref{lemma: low deg: projection} can also be written as 
  \begin{align*}
      \MoveEqLeft \|\LL_n^{\leq D_n}\|^2_{L_2(\QQ_n)}\\
      =&\ \E_{\pi}\left [W\lbs\lb  1-\B^{-2}(\alpha_1^T\alpha_2)(\beta_1^T\beta_2)\rb^{-n}\rbs^{\leq \floor*{D_n/2}}\right ].
  \end{align*}
 Now we are ready to prove Theorem~\ref{thm: low deg}.

\begin{proof}[Proof of Theorem~\ref{thm: low deg}]
Our first task is to get rid of $W$ from the expression of $\|\LL_n^{\leq D_n}\|_{L_2(\QQ_n)}$ in Lemma~\ref{lemma: low deg: projection}.
However, we can not directly bound $W$ by one since the term $(\alpha_1^T\alpha_2)^{d}(\beta_1^T\beta_2)^{d}W$ may be negative for odd  $d\in\NN$.
We claim that 
$\E[(\alpha_1^T\alpha_2)^{d}(\beta_1^T\beta_2)^{d}W]=0$ if $d\in\NN$ is odd.
To see this, first we write
\begin{align}\label{intheorem: low deg: claim for odd}
\E\slbt (\alpha_1^T\alpha_2)^{d}(\beta_1^T\beta_2)^{d}W\srbt= \E\slbt\E\slbt (\alpha_1^T\alpha_2)^{d}W\sbl \beta_1,\beta_2\srbt(\beta_1^T\beta_2)^{d}\srbt.   
\end{align}
Note that 
$(\alpha_1^T\alpha_2)^{d}W\mid \beta_1,\beta_2$
has the same distribution as 
\begin{align*}
 1\{\|\alpha_1\|_2<\B\|\beta_1\|_2^{-1}\}1\{\|\alpha_2\|_2<\B\|\beta_2\|_2^{-1}\}(\alpha_1^T\alpha_2).   
\end{align*}
Notice from \eqref{prior: radamander} that marginally,  $\alpha_1\stackrel{d}{=}-\alpha_1$, and $\alpha_1$ is independent of $\alpha_2$,  $\beta_1$ and $\beta_2$. Therefore,
\[(\alpha_1^T\alpha_2)W\mid \beta_1,\beta_2\stackrel{d}{=}-(\alpha_1^T\alpha_2)W\mid \beta_1,\beta_2.\]
Hence, conditional on $\beta_1$ and $\beta_2$, $(\alpha_1^T\alpha_2)W$ is a symmetric random variable,  and $\E[(\alpha_1^T\alpha_2)^{d}W^d\mid \beta_1,\beta_2]=0$ for any odd positive integer $d$. Since $W$ is binary random variable, $W^d=W$. Thus, $\E[(\alpha_1^T\alpha_2)^{d}W\mid \beta_1,\beta_2]=0$ as well for an odd number $d\in\NN$. Thus the claim follows from \eqref{intheorem: low deg: claim for odd}.
Therefore, from Lemma~\ref{lemma: low deg: projection}, it follows that 
\begin{align*}
 \MoveEqLeft   \|\LL_n^{\leq D_n}\|^2_{L_2(\QQ_n)}\\
    =&\ \E_{\pi}\lbt W\sum_{d=0}^{\floor*{\frac{\floor*{\frac{D_n}{2}}}{2}}}{2d+n-1\choose 2d}\lbs \B^{-2}(\alpha_1^T\alpha_2)(\beta_1^T\beta_2)\rbs^{2d}\rbt. 
\end{align*}
Observe that $\floor*{\floor*{D_n/2}/2}\leq D_n/4$. Hence, $\floor*{\floor*{D_n/2}/2}\leq \floor*{D_n/4}$. Also the summands in the last expression are non-negative. Therefore, using the fact that $W\leq 1$, we obtain
\begin{align}\label{intheorem: low deg: ub}
 \MoveEqLeft  \|\LL_n^{\leq D_n}\|^2_{L_2(\QQ_n)}\nn\\
   \leq&\ \E_{\pi}\lbt \sum_{d=0}^{\floor*{D_n/4}}{2d+n-1\choose 2d}\lbs \B^{-2}(\alpha_1^T\alpha_2)(\beta_1^T\beta_2)\rbs^{2d}\rbt. 
\end{align}
Our next step is to simplify the above bound on $\|\LL_n^{\leq D_n}\|^2_{L_2(\QQ_n)}$.
To that end, define the random variables
 $\xi_{i}=\alpha_{1i}\alpha_{2i}$ for $i\in[p]$, and $\xi'_j=\beta_{1j}\beta_{2j}$ for $j\in[q]$.
Denoting
  \[\nu=({s_x}/{p})^2,\quad\text{ and }\quad\omega=(s_y/q)^2,\]
we note that
 \begin{align*}
     \xi_{i}=\begin{cases}
     \frac{+1}{s_x} & w.p.\ \nu/2\\
       \frac{-1}{s_x} & w.p.\ \nu/2\\
      0 & w.p.\ 1-\nu,
     \end{cases}\quad\text{and}\quad 
      \xi'_{j}=\begin{cases}
     \frac{+1}{s_y} & w.p.\ \omega/2\\
     \frac{-1}{s_y} & w.p.\ \omega/2\\
      0 & w.p.\ 1-\omega.
     \end{cases}
 \end{align*}
  Also, since $\xi_{i}$ and $\xi_j$'s are symmetric, $\E\xi_{i}^{2k+1}$ and $\E\xi_j^{2k+1}$ vanishes for any $k\in\ZZ$.
Then  for any $d\in\ZZ$,
\begin{align*}
  \MoveEqLeft \E_{\pi}\lbt (\alpha_1^T\alpha_2)^{2d}(\beta_1^T\beta_2)^{2d}\rbt\\
    = &\ E_{\pi_x}\lbt\lb\sum_{i=1}^{p}\xi_i\rb ^{2d}\rbt E_{\pi_y}\lbt\lb\sum_{j=1}^q   \xi'_{j}\rb^{2d}\rbt\\
    =&\ \lb\sum_{\substack{z\in\ZZ^{p},\\
    |z|=2d}}\frac{(2d)!}{z!}\prod_{i=1}^p
    \E\slbt\xi_{i}^{z_i}\srbt\rb\lb
    \sum_{\substack{l\in\ZZ^{q},\\
    |l|=2d}}\frac{(2d)!}{l!}\prod_{j=1}^q
    \E\slbt(\xi'_{j})^{l_j}\srbt\rb
\end{align*}
by Fact~\ref{fact: multinomial theorem}. Since the odd moments of $\xi$ and $\xi'$ vanish,
the above equals
\begin{align*}
 \MoveEqLeft   \lb\sum_{\substack{z\in\ZZ^{p},\\
    |z|=d}}\frac{(2d)!}{(2z)!}\prod_{i=1}^p
    \E\slbt\xi_{i}^{2z_i}\srbt\rb\lb
    \sum_{\substack{l\in\ZZ^{q},\\
    |l|=d}}\frac{(2d)!}{(2l)!}\prod_{j=1}^q
    \E\slbt(\xi'_{j})^{2l_j}\srbt\rb\\
    =&\ \lb\sum_{\substack{z\in\ZZ^{p},\\
    |z|=d}}\frac{\nu^{|D(z)|}(2d)!}{(2z)!}\prod_{i=1}^p
    s_x^{-2z_i}\rb\\
    &\ \times\lb
    \sum_{\substack{l\in\ZZ^{q},\\
    |l|=d}}\frac{\omega^{|D(z)|}(2d)!}{(2l)!}\prod_{j=1}^q
    s_y^{-2l_j}\rb,
\end{align*}
where we remind the readers that $|D(z)|$ denotes the cardinality of the support of $z$ for any vector $z$. The above implies
\begin{align*}
    \MoveEqLeft \E_{\pi}\lbt (\alpha_1^T\alpha_2)^{2d}(\beta_1^T\beta_2)^{2d}\rbt\\
     =&\ (s_xs_y)^{-2d}\underbrace{\sum_{\substack{z\in\ZZ^{p},\\
    |z|=d}}\frac{(2d)!}{(2z)!}\nu^{|D(z)|}}_{\mJ(d;p)}\underbrace{\sum_{\substack{l\in\ZZ^{q},\\
    |l|=d}}\frac{(2d)!}{(2l)!}\nu^{|D(l)|}}_{\mJ(d;q)}.
\end{align*}

Plugging the above into \eqref{intheorem: low deg: ub} yields
 \begin{align*}
   \MoveEqLeft  \|\LL_n^{\leq D_n}\|^2_{L_2(\QQ_n)}\\
     \leq&\ \sum_{d=0}^{\floor*{D_n/4}}{2d+n-1\choose 2d}(s_xs_y)^{-2d}\B^{-4d} \mJ(d;p)\mJ_d(d;q)\\
     \stackrel{(a)}{\leq} &\ \sum_{d=0}^{\floor*{D_n/4}}\lb \frac{(2d+n-1)e}{2d}\rb^{2d}(s_xs_y)^{-2d}\B^{-4d} \mJ(d;p)\mJ(d;q),
 \end{align*}
 where (a) follows since ${a \choose b}\leq (ae/b)^b$ for $a,b\in\NN$. 
 Let us denote $\mu_x=\sqrt{ne}/(\sqrt{p}\B)$ and  $\mu_y=\sqrt{ne}/(\sqrt{q}\B)$.
 By \eqref{condition: low deg: spasity},  $\mu_x,\mu_y<1/\sqrt 3$
and
\[D_n\leq \frac{\min\{s_y^2,s_x^2\}\B^2}{ne}.\]
Therefore we have
\[ \mu_x^2D_n< \frac{s_x^2}{p}\quad\text{ and }\quad \mu_y^2 D_n<\frac{s_y^2}{q}.
\]
   Hence  Lemma 4.5 of \cite{ding2019} implies that for any $11\leq d\leq D_n$,
 \[\mJ(d;p)\lesssim (2d)!{p\choose d}\sqrt d e^{d^2/p+d/2}2^{-3d/2}\mu_x^{-2d}\nu^d,\]
 \[\mJ(d;q)\lesssim (2d)!{q\choose d}\sqrt d e^{d^2/q+d/2}2^{-3d/2}\mu_y^{-2d}\omega^d.\]
 For $d\geq 1$, Theorem 5 of \cite{debnath2000} gives
 \[(2d)!\leq \frac{(2d)^{2d+1}e^{-2d}\sqrt{2\pi}}{\sqrt{2d-1}}.\]
Also since ${p\choose d}\leq (pe/d)^d$, we have
  \[\mJ(d;p)\lesssim (2d)^{2d+1/2}\lb\frac{pe}{d}\rb^d\sqrt d e^{d^2/p-d}2^{-3d/2}\mu_x^{-2d}\nu^d,\]
  \[\mJ(d;q)\lesssim (2d)^{2d+1/2}\lb\frac{qe}{d}\rb^d\sqrt d e^{d^2/q-d}2^{-3d/2}\mu_y^{-2d}\omega^d,\]
  leading to
  \[\mJ(d;p)\mJ(d;q)\lesssim d^{2d+2}e^{d^2/p+d^2/q}2^{d+1}(\mu_x\mu_y)^{-2d}(\nu p)^d(\omega q)^d.\]
 Therefore $\|\LL_n^{\leq D}\|^2_{L_2(\QQ_n)}$ is bounded by a constant multiple of 
 \begin{align*}
\MoveEqLeft \sum_{d=11}^{\floor*{D_n/4}}\lb \frac{(2d+n-1)e}{2d}\rb^{2d} (s_xs_y)^{-2d}d^{2d+2}e^{d^2/p+d^2/q}\\
&\ \times 2^{d+1}(\mu_x\mu_y)^{-2d}(\nu p)^d(\omega q)^d\B^{-4d}\\
\lesssim &\ \sum_{d=11}^{\floor*{D_n/4}} d \lbs 
\frac{\B^{-4} (2d+n-1)^2e^2 }{2\mu_x^2\mu_y^2 pq}\rbs^de^{d^2/p+d^2/q}.
 \end{align*}
Since $D_n^2\leq \min\{p,q\}$, it follows that  $e^{d^2/p+d^2/q}\leq e^2$. Note that the above sum converges if
\begin{equation*}
 (D_n/2+n-1)^2e^2<2\B^4\mu_x^2\mu_y^2 pq=2n^2e^2,  
\end{equation*}
 or equivalently $(D_n/2+n-1)^2<2n^2$, which is satisfied for all  $n\in\NN$ since $D_n<n$. Thus the proof follows.
\end{proof}
 
\section{Proof of Theorem \ref{thm: CT matrix operator norm}}
\label{sec: Proof of CT}

We invoke the decomposition of $\hSxy$ in \eqref{representation: Sxy: model 3}. But first, we will derive a simplified form for the matrices $\H_1$ and $\H_2$ in  \eqref{representation: Sxy: model 3}. Note that we can write  $\H_1$  as 
\[\H_1=\Sx^{1/2}(I_p-\Sx^{1/2}U\Lambda U^T\Sx^{1/2})\Sx^{1/2}.\]
Let us  denote
\begin{equation}\label{def:Bx}
  B_x=\text{diag}(\underbrace{1-\Lambda_1, \ldots, 1-\Lambda_r}_{r\ \text{times}},\underbrace{1,\ldots,1}_{p-r\ \text{times}}).  
\end{equation}
Because $\Sx^{1/2}\tU$ is an orthogonal matrix, $\Sx^{1/2}\tU B_x \tU^T\Sx^{1/2}$ is a spectral decomposition, which leads to
\[\H_1=\Sx^{1/2}\slb \Sx^{1/2}\tU B_x \tU^T\Sx^{1/2}\srb^{1/2}\Sx^{1/2}=\Sx\tU B_x^{1/2} \tU^T\Sx.\]
Similarly, we can show that the matrix $\H_2$ in \eqref{representation: Sxy: model 3} equals $\Sy\tV B_y^{1/2} \tV^T\Sy$, where
\[B_y=\text{diag}(\underbrace{1-\Lambda_1, \ldots, 1-\Lambda_r}_{r\ \text{times}},\underbrace{1,\ldots,1}_{q-r\ \text{times}}).\]
Finally the fact that $\H_1= \Sx U\Lambda^{1/2}$ and $\W_2= \Sy V\Lambda^{1/2}$  in conjuction with \eqref{representation: Sxy: model 3} produces the following representation for $\tSxy=\Sx^{-1}\hSxy\Sy^{-1}$:
\begin{align*}
    \tSxy =&\ \frac{1}{n}\lbs U\Lambda^{1/2}\mZ^T\mZ\Lambda^{1/2}V^T+U\Lambda^{1/2}\mZ^T\mZ_2\Sy(\tV B_y\tV^T)\\
    &\ +(\tU B_y\tU^T)\Sx \mZ_1^T\mZ\Lambda^{1/2}V^T\\
    &\ +(\tU B_x\tU^T)\Sx \mZ_1^T\mZ_2\Sy(\tV B_y\tV^T)\rbs.
\end{align*}


Now recall the sets $E$, $F$, and $G$ defined in \eqref{def: E and F} and \eqref{def: G} in Section \ref{sec: CT}, and the decomposition of $\tSxy$ in \eqref{intheorem: CT operator: hsy decomp}.
From \eqref{intheorem: CT operator: hsy decomp} it follows that
\[\eta(\tSxy)=\eta(\mP_E\{\tSxy\})+\eta(\mP_F\{\tSxy\})+\eta(\mP_G\{\tSxy\}).\]

Recall that for any matrix $A\in\RR^{p\times q}$, and $S\subset [p]$, we denote by $A_{S^*}$ the matrix $\mP_{S\times[q]}\{A\}$. Then it is not hard to see that
$U_{E_1^*}=U$ and $V_{E_2^*}=V$, which leads to
\begin{align}\label{def: Operator: S1}
    \mS_1=&\ \frac{1}{n}\lbs  U\Lambda^{1/2}\mZ^T\mZ\Lambda^{1/2}V^T\nn\\
    &\ +U\Lambda^{1/2}\mZ^T\mZ_2\Sy\slb\tV B_y(\tV_{E_2*})^T\srb\nn\\
    &\ +(\tU_{E_1*} B_y\tU^T)\Sx \mZ_1^T\mZ\Lambda^{1/2}V^T\nn\\
    &\ +(\tU_{E_1*} B_x\tU^T)\Sx \mZ_1^T\mZ_2\Sy\slb\tV B_y(\tV_{E_2*})^T\srb\rbs.
\end{align}
Next, note that  $U_{F_1*}=0$ and  $V_{F_2*}=0$. Therefore, 
\begin{align}\label{def: Operator: S2}
    \mS_2=\frac{1}{n}\lbs(\tU_{F_1*} B_x\tU^T)\Sx \mZ_1^T\mZ_2\Sy\slb\tV B_y(\tV_{F_2*})^T\srb\rbs.
\end{align}
Finally, we note that $\mS_3=(\mathbf H_1+\mathbf H_2)$, where
\begin{align}\label{def: Operator: S3}
  \mathbf H_1= &\ \frac{1}{n}\lbs U\Lambda^{1/2}\mZ^T\mZ_2\Sy\slb\tV B_y(\tV_{F_2*})^T\srb\nn\\
  &\ +(\tU_{E_1*} B_x\tU^T)\Sx \mZ_1^T\mZ_2\Sy\slb\tV B_y(\tV_{F_2*})^T\srb\rbs
\end{align}
and
\begin{align*}
    \mathbf H_2= &\ \frac{1}{n}\lbs (\tU_{F_1*} B_y\tU^T)\Sx \mZ_1^T\mZ\Lambda^{1/2}V^T\\
    &\ +(\tU_{F_1*} B_x\tU^T)\Sx \mZ_1^T\mZ_2\Sy\slb\tV B_y(\tV_{E_2*})^T\srb\rbs.
\end{align*}


Here the term $\mS_1$ holds the information about  $\txy=U\Lambda V^T$.  Its elements are  not killed off by co-ordinate thresholding because it contains the  Wishart matrix $\mZ^T\mZ$ which concentrates around $I_r$ by  Bai-Yin law (cf. Theorem 4.7.1 of  of \cite{vershynin2020}). The only term that contributes to $\hSxy$ is $\mS_1$. 
Lemma~\ref{lemma: supp: s1} entails that $\eta(\mS_1)$ concentrates around $\txy$ in operator norm. The proof of Lemma~\ref{lemma: supp: s1} is  deferred to Appendix~\ref{Subsection: key lemma for CT}.
 \begin{lemma}\label{lemma: supp: s1}
  Suppose $s_x,s_y<n$.  Then 
  with probability $1-o(1)$,
  \begin{align*}
      \|\eta(\mS_1)-\txy\|_{op}\leq &\ \frac{\Thres \min\{s_x,s_y\}}{\sqn}\\
      &\ +C\B^2 \frac{\max\{\sqrt{s_x},\sqrt{s_y}\}}{\sqrt{n}}.
  \end{align*}
  \end{lemma}

The entries of $\mS_2$ and $\mS_3$ are linear combinations of  the entries of $\mZ_1^T\mZ_2$, $\mZ^T\mZ_1$, and $\mZ^T\mZ_2$. Since $\mZ$, $\mZ_1$, $\mZ_2$ are independent, the entries from the latter matrices are of order $O_p(n^{-1/2})$, and as we will see, they  are killed off by the thresholding operator $\eta$. Our main work boils down to showing that thresholding kills off most terms of the noise matrices $\mS_2$ and $\mS_3$, making $\|\eta(\mS_2)\|_{op}$ and $\|\eta(\mS_3)\|_{op}$ small. To that end, we state some general  lemmas,  which are proved in  Appendix~\ref{Subsection: key lemma for CT}.
That $\|\eta(\mS_2)\|_{op}$ and $\|\eta(\mS_3)\|_{op}$ are small follows as corollaries to this lemmas. Our next lemma provides a sharp concentration bound which is our main tool in analyzing the difficult regime, i.e.,  $s_x+s_y\approx \sqrt{p+q}$ case.



 



\begin{lemma}\label{lemma: sqrt n: M N lemma: rot}
Suppose $\mZ_1\in\RR^{n\times p}$ and $\mZ_2\in\RR^{n\times q}$ are independent standard Gaussian data matrices.  Let us also denote $\mQ=M\mZ_1^T\mZ_2 N$ where $M\in\RR^{p'\times p}$ and $N\in\RR^{q\times q'}$ are fixed matrices so that $p'\leq p$ and $q'\leq q$. Further suppose   $\log (p\vee q)=o(n)$ and $\log n=o(\sqrt p\vee \sqrt q)$.  Let $K_0={161}\|M\|^2_{op}\|N\|^2_{op}$.  Suppose $K\geq K_0$ is such that threshold level $\tau$ satisfies $\tau\in[\sqrt{K_0},  \sqrt{K\log (\max\{p,q\})}/2]$. 
Then there exists a constant $C>0$ so that with probability $1-o(1)$,
\begin{align*}
 \MoveEqLeft   \norm{\eta ( \mQ;\tau/\sqn)}_{op}\leq C\|M\|_{op}\|N\|_{op}\lb \sqrt{\frac{p+q}{n}}\vee\frac{p+q}{n}\rb\\
    &\ \times e^{-\tau^2/K}.
\end{align*}
\end{lemma}

Our next lemma, which also is proved in Appendix~\ref{Subsection: key lemma for CT}, handles the easier case when the threshold is exactly of the  order $\sqrt{\log (p+q)}$. This thresholding, as we will see, is required 
in the easier sparsity regime, i.e.,  $s_x+s_y\ll \sqrt{p+q}$. Although Lemma~\ref{lemma: CT: Bickel levina case} follows as a corollary to Lemma A.3 of  \cite{bickel2008}, we include it here for the sake of completeness.

\begin{lemma}\label{lemma: CT: Bickel levina case}
  Suppose $\mZ_1$, $\mZ_2$, $M$, $N$, and $\mQ$ are as in Lemma~\ref{lemma: sqrt n: M N lemma: rot}, and $\log(p+q)=o(n)$. Further suppose $\|M\|_{op},\|N\|_{op}\leq C\B$ where $C>0$ is an absolute constant. Let  $\tau=\sqrt{C_1\log(p+q)}$. Here the tuning parameter $C_1>C\B^4$ where  $C>0$ is a sufficiently large constant.  Then $\eta(\mQ;\tau/\sqn)=0$ with probability tending to one.
  \end{lemma}
  
We will need another technical lemma for handling the terms $\mS_2$ and $\mS_3$.
\begin{lemma}\label{lemma: technical}
Suppose $A\in\RR^{m\times p}$ and  $D=D_1\times D_2 \subset[m]\times[p]$. Then the followings hold:
\begin{itemize}
\item[(a)]  $\mP_D(\eta(A))=\eta(\mP_D(A))$.
\item[(b)] $\|\mP_D(A)\|_{op}\leq \|A\|_{op}$
\end{itemize}
\end{lemma}

Note that $M=\tU_{F_1*}B_x\tU^T\Sx$ satisfies
\[\|M\|_{op}\leq \|\Sx^{-1/2}\|_{op}\|\Sx^{1/2}\tU_{F^*}\|_{op}\|B_x\|_{op}\|\Sx^{1/2}\tU\|_{op}\|\Sx^{1/2}\|_{op}.\]
However, $\|B_x\|_{op}\leq 1$. Also because $\mathbb P\in\mP(r,s_x,s_y,\B)$, it follows that $\|\Sx\|_{op},\|\Sx^{-1}\|_{op}\leq \B$. Moreover, since $\Sx^{1/2}\tU$ is orthogonal, $\|\Sx^{1/2}\tU\|_{op}=1$.
 Part B of Lemma~\ref{lemma: technical} then yields $\|\Sx^{1/2}\tU_{F^*}\|_{op}\leq \|\Sx^{1/2}\tU\|_{op}=1$.
Therefore 
\begin{align}\label{eq: intheorem: CT: M}
\|M\|_{op}= \|\tU_{F_1*}B_x\tU^T\Sx\|_{op}\leq \B.   
\end{align}
 Similarly we can show that the matrix $N=\tV B_y(\tV_{F_2*})^T\Sy$ satisfies $\|N\|_{op}\leq\B$. Because $\mS_2=\eta(MZ_1^TZ_2N)$ by \eqref{def: Operator: S2},
that   $\eta(\mS_2)$ is small follows immediately from  Lemma~\ref{lemma: sqrt n: M N lemma: rot}. Under the conditions of Lemma~\ref{lemma: sqrt n: M N lemma: rot},
we have
\begin{align}\label{inlemma: supp: S2}
    \|\eta(\mS_2)\|_{op}\leq &\  C\B^2 \lb \sqrt{\dfrac{p+q}{n}}\vee \dfrac{p+q}{n}\rb\nn\\
    &\ \times e^{-\Thres^2/K}
\end{align}
  with high probability provided $K\geq 161\B^4$ and  $\Thres\in[13\B^2,\sqrt{K\log(\max\{p,q\})}/2]$. On the other hand, under the setup of Lemma~\ref{lemma: CT: Bickel levina case}, $P(\|\eta(S_2)\|_{op}=0)\to_n 1$.
   Lemma~\ref{lemma: supp: S3}, which we prove in Appendix~\ref{Subsection: key lemma for CT}, entails that the same holds for $\mS_3$.
   
  \begin{lemma}\label{lemma: supp: S3}
  Consider the setup of Lemma~\ref{lemma: sqrt n: M N lemma: rot}. Suppose $K\geq 1288\B^4$ is such that $\Thres\in[36\B^2,\sqrt{K\log(2\max\{p,q\})}/2]$. Then there exists a constant $C>0$  so that with probability tending to one,
  \[\|\eta(\mS_3)\|_{op}\leq C\B^2 \lb \sqrt{\dfrac{p+q}{n}}\vee \dfrac{p+q}{n}\rb e^{-\Thres^2/K}.\]
 Under the setup of Lemma~\ref{lemma: CT: Bickel levina case}, on the other hand, $\eta(\|\eta(S_3)\|_{op})=0$ with probability tending to one.
  \end{lemma}

We will now combine all the above lemmas and finish the proof. First, we consider the regime when   $(s_x+s_y)^2\leq (p+q)e$, so that there is thresholding, i.e.,  $\Thres>0$. We split this regime into two subregimes: $2^{1/4}(p+q)^{3/4}\leq (s_x+s_y)^2\leq (p+q)e$ and $(s_x+s_y)^2\leq 2^{1/4}(p+q)^{3/4}$.

\subsubsection{Regime $2^{1/4}(p+q)^{3/4}\leq (s_x+s_y)^2\leq (p+q)/e$:}
First, we explain why we needed to split the $e(s_x+s_y)^2\leq p+q$ regime into two parts.
Since $s_x,s_y<\sqrt{n}$,  Lemma~\ref{lemma: supp: s1} applies.  Note that  if $\Thres\in[36\B^2,\sqrt{K\log(\max\{p,q\})}/2]$ with $K\geq 1288\B^4$, then
 Lemma~\ref{lemma: supp: S3} and \eqref{inlemma: supp: S2} also apply.
Therefore it follows that in this case
\begin{align}
    \label{intheorem: CT: stitch: 1}
   \MoveEqLeft \|\eta(\tSxy)-\txy\|_{op}\leq C\B^2\lb \frac{ (s_x+s_y)\Thres}{\sqn}\nn\\
   &\ +\frac{\sqrt{s_x+s_y}}{n}+  \sqrt{\frac{p+q}{n}}\vee \frac{p+q}{n}e^{-\Thres^2/K}\rb .
\end{align}
  We will shortly show that under $(s_x+s_y)^2\leq (p+q)/e$, setting $\Thres^2 =K\log((p+q)/(s_x+s_y)^2)$ ensures that the bound in \eqref{intheorem: CT: stitch: 1} is small. However, for  \eqref{intheorem: CT: stitch: 1} to hold,
  $\Thres^2$ needs to satisfy
  \[\Thres^2/K\leq \frac{\log(\max\{p,q\}))}{4},\]
 which holds with $\Thres^2 =K\log((p+q)/(s_x+s_y)^2)$ if and only if
  \[\log((p+q)/(s_x+s_y)^2)\leq \log(\max\{p,q\}^{1/4}).\]
 Since $\max\{p,q\}\geq (p+q)/2$ the above holds when
  \[(p+q)^{3/4}\leq 2^{-1/4}(s_x+s_y)^2.\]
   Therefore, setting $\Thres^2 =K\log((p+q)/(s_x+s_y)^2)$ is useful when we are in the regime $2^{1/4}(p+q)^{3/4}\leq (s_x+s_y)^2\leq (p+q)/e$. We will analyze the regime $(s_x+s_y)^2\leq ^{1/4}(p+q)^{3/4}$ using separate procedure.
   
   In the $2^{1/4}(p+q)^{3/4}\leq (s_x+s_y)^2\leq (p+q)/e$ case,
   \begin{align*}
     \MoveEqLeft  \sqrt{\frac{p+q}{n}}e^{-\Thres^2/K} =\sqrt{\frac{p+q}{n}}\frac{(s_x+s_y)^2}{(p+q)}\\
       =&\ \frac{(s_x+s_y)}{\sqn}\frac{s_x+s_y}{\sqrt{p+q}}<\frac{(s_x+s_y)}{\sqrt{en}}
   \end{align*}
   because $(s_x+s_y)^2\leq (p+q)/e$, and similarly,
\[\frac{p+q}{n}e^{-\Thres^2/K} =\frac{p+q}{n}\frac{(s_x+s_y)^2}{(p+q)}=\frac{(s_x+s_y)}{2\sqn}\]   
  since  we also assume $s_x+s_y\leq 2\sqn$.
The above bounds entail that, in this regime, the first term on the bound in \eqref{intheorem: CT: stitch: 1} is the leading term provided $\Thres>1$, i.e., 
\begin{align*}
  \MoveEqLeft  \|\eta(\tSxy)-\txy\|_{op}\\
  \leq &\ C\B^2\lb \frac{ (s_x+s_y)\Thres}{\sqn}+  \frac{(s_x+s_y)}{\sqn}\rb\\
     \leq &\ C\B^2\frac{ (s_x+s_y)\max(\Thres,1)}{\sqn}
\end{align*}
   with probability $1-o(1)$.
Plugging in the value of $\Thres$ leads to
\begin{align*}
 \MoveEqLeft    \|\eta(\tSxy)-\txy\|_{op}\\
     \leq &\ C\B^2\frac{s_x+s_y}{\sqn}\lb \max\lbs K\log(\frac{(s_x+s_y)^2}{p+q}),1\rbs \rb^{1/2}
\end{align*}
in the regime $2^{1/4}(p+q)^{3/4}\leq (s_x+s_y)^2\leq (p+q)/e$. In our case, $(s_x+s_y)^2/(p+q)\geq e$. Also since $\B>1$ by definition of $\mP(r,s,p,q,\B)$, we also have $K\geq 1$, indicating
\begin{align*}
   \MoveEqLeft  \|\eta(\tSxy)-\txy\|_{op}\\
     \leq&\ C\B^2\frac{s_x+s_y}{\sqn}\lb  K\log(\frac{(s_x+s_y)^2}{p+q}) \rb^{1/2}.
\end{align*}


 

\subsubsection{Regime $(s_x+s_y)^2< 2^{1/4}(p+q)^{3/4}$}
When $(s_x+s_y)^2< 2^{1/4}(p+q)^{3/4}$, of course, the above line of arguments may not work although this indeed is an easier regime because $s_x+s_y$ is less than $\sqrt{(p+q)/\log(p+q)}$. In this regime, we  set $\Thres=\sqrt{C_1\log(p+q)}$ where $C_1$ is a constant depending on $\B$ as in Lemma~\ref{lemma: CT: Bickel levina case}. 
For this $\tau$, we have showed that  $\|\eta(S_2)\|_{op}= 0$ with probability tending to one. Lemma~\ref{lemma: supp: S3} implies the same holds for $\|\eta(S_3)\|_{op}$ as well. Thus from the decomposition of $\tSxy$ in \eqref{intheorem: CT operator: hsy decomp}, it follows that the asymptotic error  occurs  only due to the estimation of $\txy$ by  $\eta(S_1)$. Using Lemma~\ref{lemma: supp: s1}, we thus obtain
\[\|\tSxy-\txy\|_{op}\leq C\B^2 \frac{ (s_x+s_y)\max\{\Thres,1\}}{\sqn}.\]
On the other hand, since  $(p+q)^{3/4}>2^{-1/4}(s_x+s_y)^2$, rearranging terms, we have 
\begin{align*}
  \log((p+q)/(s_x+s_y)^2)> &\ \slb\log(p+q)-\log 2\srb /4\\
    > &\ C\log(p+q).
\end{align*}
Thus, in the regime $(s_x+s_y)^2< 2^{1/4}(p+q)^{3/4}$, we have
\begin{align}
  \MoveEqLeft  \|\tSxy-\txy\|_{op}\nn\\
    \leq &\  C\B^2\frac{s_x+s_y}{\sqrt{n}}\max\lbs\sqrt{C_1\log(\frac{p+q}{(s_x+s_y)^2})},1\rbs.
\end{align}

\subsubsection{Regime $(s_x+s_y)^2>(p+q)/e$}
It remains to analyze the case when either $(s_x+s_y)^2>(p+q)/e$.  In that case, there is no thresholding, i.e.,  $\Thres=0$.   We will show that the assertions of Theorem~\ref{thm: CT matrix operator norm} holds in this case as well. To that end, note that \eqref{intheorem: CT operator: hsy decomp} implies
\[\|\tSxy-\txy\|_{op}\leq \|S_1-\txy\|_{op}+\|S_2+S_3\|_{op}.\]
From the proof of Lemma~\ref{lemma: supp: s1}
it follows that $\|S_1-\txy\|_{op}\leq C\B^2\max\{\sqrt{s_x},\sqrt{s_y}\}/\sqn$. For $S_2$, we have shown that it is of the form $M\mZ_1^T\mZ_2N/n$ where $\|M\|_{op},\|N\|_{op}\leq\B$. On the other hand, we showed that $S_3=\mathbf H_1+\mathbf H_2$, where  the proof of Lemma~\ref{lemma: supp: S3} shows  $\mathbf H_1$ and $\mathbf H_2$ are of the  form $MAN$ where $\|M\|_{op}\|N\|_{op}\leq 2\B^2$ and $A$ is either $[\mZ\ \mZ_1]^T\mZ_2/n$ (for $\mathbf H_1)$ or $\mZ_1^T[\mZ\ \mZ_2] /n$ (for $\mathbf H_2$). Therefore, it is not hard to see that $\|S_2+S_3\|_{op}$ is bounded by
\[ C\B^2\slb \|\mZ_1^T\mZ_2\|_{op}+\|\mZ_1^T[\mZ\ \mZ_2]\|_{op}+\|[\mZ\ \mZ_1]^T\mZ_2\|_{op}\srb.\]
For standard Gaussian matrices $\mZ_1\in\RR^{n\times p}$ and $\mZ_2\in\RR^{n\times q}$ it holds that $\|Z_1^TZ_2/n\|_{op}\leq C(\sqrt{(p+q)/n}+(p+q)/n)$ with probability $1-o(1)$ (cf. Theorem 4.7.1 of  of \cite{vershynin2020}). Since $r\leq \min\{p,q\}$, 
it follows that  $\|S_2+S_3\|_{op}\leq C\B^2(\sqrt{{p+q}/{n}}+(p+q)/n)$ with probability $1-o(1)$. The above discussion leads to
\begin{align*}
    \|\tSxy-\txy\|_{op}\leq &\  C\B^2\lb \slb\frac{s_x+s_y}{n}\srb^{1/2}+\slb \frac{p+q}{n}\srb^{1/2}\\
    &\ + \frac{p+q}{n}\rb\\
    \leq &\ 2C\B^2\lb\lb \frac{p+q}{n}\rb^{1/2}+\frac{p+q}{n}\rb
\end{align*}
because $s_x+s_y<p+q$.
If $(p+q)\leq e(s_x+s_y)^2$, the above bound is of the order $(s_x+s_y)/\sqn$.  Thus  Theorem \ref{thm: CT matrix operator norm} follows.

\section{Proof of Corollary~\ref{cor: CT support recovery}}
\label{sec: proof of cor CT}

\begin{proof}[Proof of Corollary~\ref{cor: CT support recovery}]
We will first show that there exist $C'_\B, c_\B>0$ so that
\begin{align}\label{intheorem: statement 1: CT support recovery}
     \MoveEqLeft \max_{i\in[r]}\|\hai-\myv\ai\|_2\nn\\
     \leq &\ C'_\B \frac{(s_x+s_y)}{\sqn}\max\lbs \lb c_\B\log (\frac{p+q}{(s_x+s_y)^2})\rb^{1/2}, 1\rbs. 
  \end{align}
For the sake of simplicity, we denote the matrix $\widehat U^{(1)}$ in Algorithm~\ref{algo: CT: CCA} by $\widehat U$. 
Denoting $\e_n'=\|\eta(\tSxy)-\txy\|_{op}$, we note that 
\[\|\Sx^{1/2}\eta(\tSxy)\Sy^{1/2}-\Sx^{1/2}U\Lambda V^T\Sy^{1/2}\|_{op}\leq \B\e_n'.\]
Also the matrix $\widehat{U}_{\text{pre}}$ defined in Algorithm~\ref{algo: CT: CCA} and $\Sx^{1/2}U$ are the matrices corresponding to the
leading $r$ singular vectors of $\Sx^{1/2}\eta(\tSxy)\Sy^{1/2}$ and $\Sx^{1/2}U\Lambda V^T\Sy^{1/2}$, respectively.
By Wedin's sin-theta theorem (we use Theorem 4 of \cite{yu2015}), for any $1\leq i< r$,
\[\min_{\myv\in\{\pm 1\}}\|\widehat U_i^{\text{pre}}-\myv\Sx^{1/2}\ai\|_2\leq \frac{2^{3/2}(2\Lambda_1+\e'_n)\e_n'}{\min\{\Lambda_{i-1}^2-\Lambda_i^2,\Lambda_i^2-\Lambda_{i+1}^2\}},\]
where $\Lambda_0$ is taken to be $\infty$, and
\[\min_{\myv\in\{\pm 1\}}\|\widehat U_r^{\text{pre}}-\myv\Sx^{1/2}u_r\|_2\leq \frac{2^{3/2}(2\Lambda_1+\e_n')\e_n'}{\Lambda_{r-1}^2-\Lambda_r^2}.\]
Since $\PP\in\modG$, $\min_{i\in[r]}(\Lambda_{i-1}-\Lambda_{i})>\B^{-1}$ and $\min_{i\in[r]}\Lambda_i>\B^{-1}$. Therefore, for $\e_n'<1$, we have
\[\max_{i\in[r]}\min_{\myv\in\{\pm 1\}}\|\widehat U_i^{\text{pre}}-\myv\Sx^{1/2}\ai\|_2\leq C_\B\e_n'.\]
We have to show $\e_n'<1$.
Theorem~\ref{thm: CT matrix operator norm} gives a bound on $\e_n'$, which can be made smaller than one if the $C_\B$ in \eqref{intheorem: statement: CT support recovery} is chosen to be sufficiently large. Hence, the above inequality holds.
Because $\widehat U_i^{\text{pre}}=\Sx^{1/2}\hai$, using the fact $\|\Sx\|_{op}\leq \B$, the last display implies
\[\max_{i\in[r]}\min_{\myv\in\{\pm 1\}}\|\hai-\myv\ai\|_2\leq \B^{1/2}C_\B\e_n',\]
which, combined with Theorem~\ref{thm: CT matrix operator norm}, proves \eqref{intheorem: statement 1: CT support recovery}. Now note that the constant $C_\B$ in \eqref{intheorem: statement: CT support recovery} can be chosen so large such that the right hand side of \eqref{intheorem: statement 1: CT support recovery} is smaller than  $1/(2\B^2\sqrt{r})$. Since $\|\Sx\|_{op}<\B$, it follows that Condition~\ref{Assumption on estimators} is satisfied, and the rest  of the proof then follows from Theorem~\ref{thm: support recovery: sub-gaussians}.
\end{proof}

\section{Proof of Auxilliary Lemmas}
\subsection{Proof of Technical Lemmas for Theorem~\ref{thm: lower bound}}
\label{sec: add lemma: lower bound }

The  following lemma can be verified using elementary linear algebra, and hence its proof is  omitted.
  \begin{lemma}\label{Lemma: spectrum of sigma}
Suppose $\Sigma$ is of the form \eqref{def: sigma in Theorem 3}. Then the spectral decomposition of $\Sigma$ is as follows:
\[\Sigma=\sum_{i=1}^{p-1}x_1^{(i)}(x_1^{(i)})^T+\sum_{i=1}^{q-1}x_2^{(i)}(x_2^{(i)})^T+(1+\rho)x_3x_3^T+(1-\rho)x_4x_4^T,\]
where the eigenvectors are of the following form:
\begin{enumerate}
    \item For $i\in[p-1]$, $x_1^{(i)}=(y_i,0_{q})$, where $\{y_i\}_{i=1}^{p-1}\subset \RR^p$ forms an orthonormal basis system  of the orthogonal space of $\alpha$.
    \item For $i\in[q-1]$, $x_2^{(i)}=(0_{p},z_i)$, where $\{z_i\}_{i=1}^{q-1}\subset \RR^p$ forms an orthonormal basis system  of the orthogonal space of $\beta$.
    \item $x_3=(\alpha/\sqrt{2},\beta/\sqrt{2})$ and $x_4=(\alpha/\sqrt{2},-\beta/\sqrt{2})$.
\end{enumerate}
Here for $k\in\NN$, $0_k$ denotes the $k$-dimensional vector whose all entries are zero.
 \end{lemma}

\begin{lemma}\label{Lemma: determinant of Sigma}
 Suppose $\Sigma$ is as in \eqref{def: sigma in Theorem 3}. Then $\text{det}(\Sigma)=1-\rho^2$ and 
\begin{align*}
\Sigma^{-1}=&\ \begin{bmatrix}I-\alpha\alpha^T& 0\\
0 & I-\beta\beta^T\\
\end{bmatrix}
+\dfrac{1}{2(1+\rho)}\begin{bmatrix}
\alpha\alpha^T & \alpha\beta^T\\
\beta\alpha^T & \beta\beta^T
\end{bmatrix} \\
+&\ \dfrac{1}{2(1-\rho)}\begin{bmatrix}
\alpha\alpha^T & -\alpha\beta^T\\
-\beta\alpha^T & \beta\beta^T
\end{bmatrix}.
\end{align*}
\end{lemma}

\begin{proof}[Proof of Lemma \ref{Lemma: determinant of Sigma}]
Follows directly from Lemma~\ref{Lemma: spectrum of sigma}.
\end{proof}


\begin{lemma}\label{Lemma: trace }
Suppose $\Sigma_1$ and $\Sigma_2$ are of the form \eqref{def: sigma in Theorem 3} with singular vectors $\alpha_1$, $\beta_1$, $\alpha_2$, and $\beta_2$, respectively. 
Then
\[ Tr(\Sigma_1 \Sigma_2^{-1})=p+q +\dfrac{2\rho^2}{1-\rho^2} \lb 1-(\beta_1^T\beta_2)(\alpha_1^T\alpha_2)\rb. \]
\end{lemma}

\begin{proof}
 Lemma~\ref{Lemma: determinant of Sigma} can be used to obtain the form of $\Sigma_2^{-1}$, which implies $ \Sigma_1 \Sigma_2^{-1}$ equals
 \begin{align*}
\begin{bmatrix}
 I_p & \rho\alpha_1\beta_1^T\\
 \rho\beta_1\alpha_1^T & I_q
 \end{bmatrix}
 \begin{bmatrix}
 I_p+\rho C_{\rho}\alpha_2\alpha_2^T & -C_{\rho}\alpha_2\beta_2^T\\
 -C_{\rho}\beta_2\alpha_2^T & I_q+\rho C_{\rho}\beta_2\beta_2^T,
 \end{bmatrix}
 \end{align*}
 where 
 $C_{\rho}={\rho}/{1-\rho^2}$.
 Since $Tr(\Sigma_1\Sigma^{-1}_2)$ equals the sum of the two $p\times p$ and $q\times q$ diagonal submatrices, we obtain that
 \begin{align*}
  Tr(\Sigma_1\Sigma^{-1}_2)
 =&\ Tr\lb I_p+\rho C_{\rho}\alpha_2\alpha_2^T-\rho C_{\rho}(\beta_1^T\beta_2)\alpha_1\alpha_2^T\rb\\ &\ +Tr\lb I_q+\rho C_{\rho}\beta_2\beta_2^T-\rho C_{\rho}(\alpha_1^T \alpha_2)\beta_1\beta_2^T\rb\\
 =&\ p+q + \dfrac{\rho^2}{1-\rho^2} \lb Tr(\alpha_2^T\alpha_2)+Tr(\beta_2^T\beta_2)\\
 &\ -(\beta_1^T\beta_2)Tr(\alpha_1\alpha_2^T)-(\alpha_1^T\alpha_2)Tr(\beta_1\beta_2^T)\rb,
 \end{align*}
 where we used the linearity of Trace operator, as well as the fact that $Tr(AB)=Tr(BA)$.
 Noticing $\|\alpha_2\|_2=\|\beta_2\|_2=1$, the result follows.
\end{proof}
\subsection{Proof of Key Lemmas for Theorem~\ref{thm: CT matrix operator norm}}
 \label{Subsection: key lemma for CT}
\subsubsection{Proof of Lemma~\ref{lemma: supp: s1}}
\begin{proof}[Proof of Lemma~\ref{lemma: supp: s1}]
 Note that
\begin{align*}
\|\eta(\mS_1)- \txy\|_{op}\leq &\  \underbrace{\| \eta(\mS_1) -\mS_1 \|_{op}}_{T_1}\\
&\ +\underbrace{\|\mS_1-\txy\|_{op}}_{T_2}.
\end{align*}


We deal with the term $T_1$ first.
 Recall from \eqref{intheorem: CT operator: hsy decomp} that $\mS_1=\mP_{E_1\times E_2}(\tSxy)$  is a sparse matrix. In particular, each row and column  of $\mS_1$ can have at most $s_y$ and $s_x$ many nonzero elements, respectively. 
Now we make use of two elementary facts.
First, for $x\neq 0$, 
$|\eta(x)-x|\leq {\Thres}/{\sqn}$, and second,  for any matrix  $A\in\RR^{p\times q}$,
 $$\|A\|_{op}\leq \max_{1\leq i\leq p}\sum_{j=1}^{p}|A_{ij}|\wedge \max_{1\leq j\leq q}\sum_{i=1}^{n}|A_{ij}|. $$
 The above results, combined with the row and column sparsity of $\mS_1$, lead to
 \begin{align*}
  T_1=&\ \| \eta(\mS_1) -\mS_1\|_{op}\\
  \leq &\ \slb\max_{1\leq i\leq p}\|(\mS_1)_{i*}\|_{0}\srb\wedge \slb \max_{1\leq j\leq q}\|(\mS_1)_{j}\|_{0}\srb \frac{\Thres}{\sqn}\\
  \leq &\ \min\{s_x,s_y\}\dfrac{\Thres}{\sqn},   
 \end{align*}
which is the first term in the bound of $\|\eta(\mS_1)-\Sxy\|_{op}$.

Now for $T_2$, noting $\txy=U\Lambda V^T$, observe that
\begin{align*}
  \MoveEqLeft  \mS_1-\txy\\
  =&\ \underbrace{ U\Lambda^{1/2}\left(\frac{\mZ^T\mZ}{n}-I_r\right)\Lambda^{1/2}V^T}_{S_{11}}\\
    &\ +\underbrace{ U\Lambda^{1/2}\frac{\mZ^T\mZ_2}{n}\Sy\slb\tV B_y(\tV_{E_2*})^T\srb}_{S_{12}}\\
    &\ +\underbrace{(\tU_{E_1*} B_y\tU^T)\Sx \frac{\mZ_1^T\mZ}{n}\Lambda^{1/2}V^T}_{S_{13}}\\
    &\ +\underbrace{(\tU_{E_1*} B_x\tU^T)\Sx \frac{\mZ_1^T\mZ_2}{n}\Sy\slb\tV B_y(\tV_{E_2*})^T\srb }_{S_{14}}.
\end{align*}
It is easy to see that
\begin{align*}
    \|S_{11}\|_{op}\leq &\  \|\Sx^{-1/2}\|_{op}\|\Sx^{1/2}U\|_{op}\|\Sy^{-1/2}\|_{op}\|\Sy^{1/2}V\|_{op}\|\Lambda\|_{op}\\
    &\ \times \norm{\frac{\mZ^T\mZ}{n}-I_r}_{op}.
\end{align*}
Since $(\mX,\mY)\sim\mathbb P\in\mP(r,s_x,s_y,\B)$,  $\Sx^{-1}$ and $\Sy^{-1}$ are bounded it operator norm by $\B$. Also, $\Sx^{1/2}\tU$ and $\Sy^{1/2}\tV$ are orthonormal matrices.
Therefore the operator norms of the matrices $\Sx^{1/2}U$, $\Sy^{1/2}V$, and $\Lambda$ are bounded by one.
On the other hand, by Bai-Yin's law on eigenvalues of Wishart matrices (cf.  Theorem 4.7.1 of \cite{vershynin2020}), $\|{\mZ^T\mZ}/{n}-I_r\|_{op}\leq C(\sqrt{r/n}+r/n)$ with high probability. Since $r<s_x<\sqn$, clearly $r/n<1$. Thus  $\|S_{11}\|_{op}\leq \B C\sqrt{r/n}$ with high probability. Hence it suffices to show that the terms $S_{12}$, $S_{13}$, and $S_{14}$ are small in operator norm, for which, we will make use of Lemma~\ref{lemma: normal data matrix}. First let us consider the case of $S_{12}$. Clearly,
\begin{align*}
    \|S_{12}\|_{op}\leq &\  \|\Sx^{-1/2}\|_{op}\|\Sx^{1/2}U\|_{op}\|\Lambda^{1/2}\|_{op}\\
    &\ \times\norm{\frac{\mZ^T\mZ_2}{n}\Sy\slb\tV B_y(\tV_{E_2*})^T\srb}_{op}.
\end{align*}
We already mentioned that $\|\Sx^{-1}\|_{op}\leq\B$, and  $\|\Sx^{1/2}U\|_{op}$ and $\|\Lambda\|_{p}$ are bounded by one.
 Therefore, it follows that
\[\|S_{12}\|_{op}\leq\B^{1/2}\norm{\frac{\mZ^T\mZ_2}{n}\Sy\slb\tV B_y(\tV_{E_2*})^T\srb}_{op}.\]
Now we apply Lemma~\ref{lemma: normal data matrix} on the term $\mZ^T\mZ_2\Sy\slb\tV B_y(\tV_{E_2*})^T\srb$ with $A=I_r$, and $B=\Sy\tV B_y(\tV_{E_2*})^T$. Note that $\Sy$, $\tV$, and $B_y$ are full rank matrices, i.e., they have rank $q$. Therefore, the rank of $B$ equals rank of 
$\tV_{E_2*}$. Note that the rows of the matrix $\tV$ are linearly independent because the square matrix $\tV$ has full rank. Therefore, the rank of $\tV_{E_2*}$ is $|E_2|$, which is $s_y$. Hence, the rank of $B$ is also $s_y$. Also note that $\text{rank}(A)=r\leq s_y\leq n$. Therefore Lemma~\ref{lemma: normal data matrix} can be applied with  $a=r$ and $b=s_y$. Also, $\|A\|_{op}=1$ trivially follows. Using the same arguments which led to \eqref{eq: intheorem: CT: M}, on the other hand, we can show that $\|B\|_{op}\leq\B$ by \eqref{intheorem: CT operator: hsy decomp}.
Therefore Lemma~\ref{lemma: normal data matrix} implies that for any $t>0$, the following holds with probability at least $1-\exp(-Cn)-\exp(-t^2/2)$:
\[\norm{\frac{\mZ^T\mZ_2}{n}\Sy\slb\tV B_y(\tV_{E_2*})^T\srb}_{op}\leq C\B\frac{\max\{\sqrt{s_y},t\}}{\sqn},\]
which implies $|S_{12}|\leq C\B^{3/2}\max\{\sqrt{ s_y},t\}/\sqrt{n}$ with high probability. 
Exchanging the role of $X$ and $Y$ in the above arguments, we can show that $|S_{13}|\leq C\B^{3/2}\max\{\sqrt{ s_x},t\}/\sqrt{n}$ with high probability. For $S_{14}$, we note that
\[\|S_{14}\|_{op}\leq \norm{(\tU_{E_1*} B_x\tU^T)\Sx \frac{\mZ_1^T\mZ_2}{n}\Sy\slb\tV B_y(\tV_{E_2*})^T\srb}_{op}.\] 
We intend to apply Lemma~\ref{lemma: normal data matrix} with $A=\Sx\tU B_x(\tU_{E_1*})^T$ and $B=\Sy\tV B_y(\tV_{E_2*})^T$. Arguing in the lines of the proof for the term $S_{12}$, we can show that $A$ and $B$ have rank $a=s_x$ and $b=s_y$, respectively. Without loss of generality we assume $s_y\geq s_x$, which yields $b\geq a$, as requred by Lemma~\ref{lemma: normal data matrix}.
Otherwise, we can just take the transpose of $S_{14}$, which leads to $a=s_y$ and $b=s_x$, implying $b\geq a$. Using \eqref{eq: intheorem: CT: M}, as before, we can show that the operator norms of $A$ and $B$ are bounded by $\B$. Therefore,  Lemma~\ref{lemma: normal data matrix} implies that  for all $t\geq 0$,
\[\|S_{14}\|_{op}\leq C\B^2\frac{\max\{\sqrt s_x,\sqrt s_y, t\}}{\sqn}\]
with probability at least $1-\exp(-Cn)-\exp(-t^2/2)$. Hence, it follows that with probability $1-o(1)$,
\[\|\mS_1-\txy\|_{op}\leq  C\B^2 \frac{\max\{\sqrt{s_x},\sqrt{s_y}\}}{\sqrt{n}}.\]

\end{proof}

\subsubsection{Proof of Lemma~\ref{lemma: sqrt n: M N lemma: rot}}
\label{subsec: proof of sqrt n lemma}
 Without loss of generality, we will assume that $p>q$. We will also assume, without loss of generality, that $p'=p$ and $q'=q$. If that is not the case, we can add some zero rows to  $M$ and zero columns to $N$, respectively, which does not change their operator norm, but ensures $p'=p$ and $q'=q$.
 For any $p\in\NN$, let $\S^{p-1}$ denote the unit sphere in $\RR^p$. We denote an $\e$-net (with respect to Eucledian norm) on any set $\mathcal{X}\subset \RR^p$ by $T^{\e}(\mathcal{X})$. When $\mathcal{X}=\S^{p-1}$, there exists an $\e$-net of $\S^{p-1}$ so that 
  \[|T^{\e}(\S^{p-1})|\leq \slb 1+2/\e\srb ^p.\]
  By  $T^{\e}_{p}$, we denote  such an $\e$-net. Although $T^{\e}_{p}$ may not be unique, that is not necessary for our purpose.  For a subset $S\subset [p]$,  $T_{p}^{\e}(S)$ will denote an $\e$-net of the set
  $\{x\in \S^{p-1}\ :\ x_i=0\text{ if }i\neq 0\}$. Note that each element of the latter set has at most $|S|-1$ many degrees of freedom, from which, one can show that
  $|T^{\e}_k(S)|\leq ( 1+2/\e) ^{|S|}$.
  The following Fact on $\e$-nets will be very useful for us. The proof is standard and can be found, for example, in \cite{vershynin2020}.
  \begin{fact}\label{fact: e-net}
   Let $A\in\RR^{p\times q}$ for $p,q\in\mathbb{N}$. Then there exist $x\in T^{\e}_{p}$ and $y\in T^{\e}_q$ such that
  $|\langle x,Ay\rangle |\geq (1-2\e)\|A\|_{op}$.
  \end{fact}

  Letting $\mA=\eta(\M \mZ_1^T\mZ_2\N )$, 
and  using Fact \ref{fact: e-net}, we obtain that 
\[P\left( \|\mA\|_{op}>\delta\right )\leq P\left( \max_{x\in T^{\e}_p,y\in T^{\e}_q}|\langle x, \mA y\rangle |\geq (1-2\e)\delta\right)\]
for any $\delta>0$.
Proceeding like Proposition 15 of \cite{deshpande2014}, we fix $1<\da\leq \min\{p,q\}$, and introduce
the sets
\begin{align}\label{inlemma: def: sxsy}
S_x &\ =\{i \in[p]:|x_i|\geq \sqrt{\da/p}\},\nn\\
 S_y&\ =\{i\in[q]: |y_i|\geq \sqrt{\da/q}\},
\end{align}
and their complements  $S_x^c=[p]\setminus S_x$ and $S_y^c=[q]\setminus S_y$. The precise value of $\da$ will be chosen later.
For any subset $A\subset [k]$, $k\in\NN$, and vector $x\in\RR^k$, we denote by  $x_A$ the projection of $x$ onto $A$, which means $x_A\in\RR^p$ and $(x_A)_i=x_i$ if $i\in A$, and zero otherwise.
Let us denote the projections of $x$ and $y$ on $S_x$, $S_x^c$, $S_y$, and $S_y^c$,  by $x_{S_x}$, $x_{S_x^c}$, $y_{S_y}$, and $y_{S_y^c}$, respectively.
Note that this implies
\[x=x_{S_x}+x_{S_x^c},\quad y=y_{S_y}+y_{S_y^c},\]
as well as  
\[x_{S_x}, x_{S_x^c}\in \RR^p,\quad y_{S_y},y_{S_y^c}\in\RR^q.\]
There are fewer elements the sets $S_x$ and $S_y$ compared to their complements. Therefore, we will treat these sets separately.
To that end, we consider the splitting
\begin{align}\label{inlemma: CT: big lemma}
\MoveEqLeft P\lp \max_{x\in T^{\e}_p,y\in T^{\e}_q}|\langle x, \mA y\rangle |\geq 4\delta(1-2\e)\rp\nn\\
\leq &\ \underbrace{P\lp \max_{x\in T^{\e}_p,y\in T^{\e}_q}|\langle x_{S_x}, \mA y_{S_y}\rangle |\geq \delta(1-2\e)\rp}_{T_1}\nn \\
&\ + \underbrace{P\lp \max_{x\in T^{\e}_p,y\in T^{\e}_q}|\langle x_{S_x}, \mA y_{S_y^c}\rangle |\geq \delta(1-2\e)\rp}_{T_2}\nn\\
&\ +\underbrace{ P\lp \max_{x\in T^{\e}_p,y\in T^{\e}_q}|\langle x_{S_x^c}, \mA y\rangle |\geq \delta(1-2\e)\rp }_{T_3}
\end{align}
The  term $T_1$ can be bounded by Lemma \ref{Adlemma: lowrank: premultiply: 1}.

\begin{lemma}\label{Adlemma: lowrank: premultiply: 1}
Suppose $\M$ and $\N$ are as in
Lemma~\ref{lemma: sqrt n: M N lemma: rot} and
$\mA=\eta(\mQ)$ where
$\mQ=\M \mZ_1^T\mZ_2 \N/n$.
Then for any $\Delta>0$, there exist absolute constants $C,c>0$ such that
\begin{align*}
\MoveEqLeft P\lbs \max_{x\in T^{\e}_p,y\in T^{\e}_q}|\langle x_{S_x}, \mA y_{S_y}\rangle |\geq \Delta\rbs\\
\leq &\ C\exp \lb (p+q)\dfrac{\log (C\da)}{\da}\\
&\ -\dfrac{n^2\Delta^2}{4C\|\M\|_{op}^2\|\N\|_{op}^2(2n+p+q)}\rb\\
&\ + \dfrac{C}{\Delta^2}\|\M\|_{op}^2\|N\|_{op}^2(n(p+q))^C\lbs e^{ -c(n+q)}+ e^{ -c(n+p)}\rbs 
\end{align*}
\end{lemma}

We state another lemma which helps in controlling  the terms  $T_2$ and $T_3$.

\begin{lemma}\label{Adlemma: lowrank: premultiply: 2}
Suppose $\M$, $\N$, $\mZ_1$, $\mZ_2$, and $\mA$ are as in Lemma~\ref{lemma: sqrt n: M N lemma: rot}.
Let $K_0={161}\|M\|^2_{op}\|N\|^2_{op}$.  Suppose $K>0$ is such that $K\geq  K_0$ and moreover,  $\tau\in[\sqrt{K_0},  \sqrt{K\log p}/2]$.
  Let $\mathcal T_2$ be either the set $T_q^\e$ or the set $\Tvq=\{y_{S_y}:y\in T_q^{\e}\}$.
Then there exist absolute constants $C, c>0$  such that the following holds for any $\Delta>0$:
\begin{align*}
& P\lbs \max_{x\in T^{\e}_p,y\in \mathcal T_2 }|\langle x_{S^c_x}, \mA y\rangle |\geq \Delta\rbs\\
\leq &\ C\exp\lb C(p+q) -\dfrac{\Delta^2n^2 e^{\tau^2/K}}{C\|M\|_{op}^2\|N\|_{op}^2 \da(2n+p+q)}\rb\\
&\ +\dfrac{C\|M\|_{op}^2\|N\|_{op}^2}{\Delta^2}(n(p+q))^C\exp\lb-c\min(n,\sqrt{p})\rb.
\end{align*}
\end{lemma}

Note that when $\mathcal  D=T_q^{\e}$, Lemma~\ref{Adlemma: lowrank: premultiply: 2} yields a bound on $T_3$. On the other hand, the case $\mathcal T_2=\Tvq$ yields a bound on the term
\begin{equation}\label{inlemma: prob: second term pseudo}
 T_2'=  P\lp \max_{x\in T^{\e}_p,y\in T^{\e}_q}|\langle x_{S_x^c}, \mA y_{S_y}\rangle |\geq \delta(1-2\e)\rp.
 \end{equation}
While $T_2'$ is not exactly equal to $T_2$,  interchanging the role of $x$ and $y$ in $T_2'$ gives $T_2$. Since the upper bound on $T_2'$ given by Lemma~\ref{Adlemma: lowrank: premultiply: 2} is symmetric in $p$ and $q$, it is not hard to see that the same bound works for $T_2$.

If we let $\e=1/4$, then $\Delta=\delta/2$. 
Combining the  bounds on $T_1$, $T_2$, and $T_3$,  we conclude that the right hand side of \eqref{inlemma: CT: big lemma} is $o(1)$ if $\Delta^2$ is larger than some constant multiple of \begin{align*}
   \max\lbs & \frac{(n+p+q)(p+q)}{n^2}\slb \frac{\log \da}{\da}+J_{p,q}e^{-\tau^2/K_0}\srb,\\
    &\ \frac{(n(p+q))^C}{\exp(c\min\{n,\sqrt p\})} \rbs \|M\|^2_{op}\|N\|_{op}^2
\end{align*}
where $K_0={320}\|M\|^2_{op}\|N\|^2_{op}$. We will show that the first term dominates the second term. 
By our assumption on $\tau$,  $\tau^2<80\log p\|M\|_{op}^2\|N\|_{op}^2$, which implies $\tau^2/K_0<\log(p\wedge q)/2$, which combined with the fact $\da>1$, yields $\da\exp(-\tau^2/K_0)>\da/\sqrt{p\wedge q}$. 
On the other hand, under $p>q$, our assumption on $n$ implies $\log n=o(\sqrt p)$. Also because $p+q=o(\log n)$, it  follows  that ${(n(p+q))^C}{\exp(-c\min\{n,\sqrt p\})}$ is small, in particular 
\begin{align*}
 \MoveEqLeft \frac{(n+p+q)(p+q)}{n^2}\slb \frac{\log \da}{\da}+J_{p,q}e^{-\tau^2/K_0}\srb\\
  \geq &\ \frac{(n+p+q)(p+q)}{n^2\sqrt{p\wedge q}}\\
  \gg &\ (n(p+q))^C\exp(-c\min\{n,\sqrt p\}).  
\end{align*}
Therefore,  for $P(\|\mA\|_{\delta}>\delta)$ to be small,
\begin{align*}
\delta^2 > &\ C \min_{1<\da<p\wedge q}\|M\|_{op}^2\|N\|_{op}^2\frac{(n+p+q)(p+q)}{n^2}\\
&\ \times \slb \frac{\log \da}{\da}+\da e^{-\tau^2/K_0}\srb    
\end{align*}
suffices.  In particular, we choose $\da=\exp(\tau^2/(2K_0))$. Note that because $\tau^2\leq K_0\log(p\wedge q)/2$, this choice of $\da$ ensures that $\da\ll \min\{p,q\}$, as required. The proof follows noting this choice of $\da$ also implies
\begin{align*}
  \MoveEqLeft  \frac{\log \da}{\da}+\da e^{-\tau^2/K_0}\leq e^{-\tau^2/(2.5 K_0)}\\
    =&\ \lbs\exp(\frac{-\tau^2}{40^2\|M\|_{op}^2\|N\|_{op}^2})\rbs^2.
\end{align*}
$\hfill\Box$


\subsubsection{Proof of Lemma~\ref{lemma: CT: Bickel levina case}}
\begin{proof}[Proof of \ref{lemma: CT: Bickel levina case}]
For any $i\in[p]$ and $j\in[q]$, \[\mZ_1M_{i*}/\|M_{i*}\|_2\sim N(0,I_n)\]
and $\mZ_2 N_{j}/\|N_j\|_2\sim N(0,I_n)$ are independent.  In this case,  there exist absolute constants $\delta$, $c$ and $C>0$,  so that (cf. Lemma A.3 of  \cite{bickel2008})
\[P\slb \frac{|M_{i*}^T\mZ_1^T\mZ_2 N_j|}{\|M_{i*}\|_2\|N_j\|_2}\geq nt\srb\leq C\exp(-c nt^2)\]
for all $t\leq \delta$. Since $(M\mZ_1^T\mZ_2N)_{ij}=M_{i*}^T\mZ_1^T\mZ_2 N_j$, and $\|M_{i*}\|_2,\|N_i\|_{2}\leq\B$, using union bound we obtain
\[P\slb {|M\mZ_1^T\mZ_2 N|_\infty}\geq n t\srb\leq C\exp(\log(p'q')-c  nt^2/\B^4).\]
Letting $\tau=\B^2\sqrt{C'\log(p+q)}$ and   $t=\tau/\sqn$, we observe that for our choice of $\tau$, $t<\delta$ for all sufficiently large $n$ since $\log(p+q)=o(n)$. Therefore, the above inequality leads to
\begin{align*}
   P(\eta(\mQ)\neq 0) &\ =P\slb |\mQ|_\infty \geq \tau/\sqn\srb\\
   \leq &\ C\exp(2\log(p'+q')-c C'\log(p+q)). 
\end{align*}
Because $p'\leq p$ and $q'\leq q$ by our assumption on $M$ and $N$,  $C'>2/c$ suffices. Hence the proof follows. 
\end{proof}

\subsubsection{Proof of Lemma~\ref{lemma: supp: S3}}
\begin{proof}[Proof of Lemma~\ref{lemma: supp: S3}]

From the definition of $\mS_3$ in \eqref{intheorem: CT operator: hsy decomp},  and \eqref{def: Operator: S3}, it is not hard to see that $\eta(\mS_3)=\eta(\mathbf H_1)+\eta(\mathbf H_2)$.
We will show that $\mathbf H_1$ is of the form $M[\mZ \ \mZ_1]^T \mZ_2 N$ where $\|M\|_{op}\leq 2\B$ and $\|N\|_{op}\leq \B$. Then the first part would follow from Lemma~\ref{lemma: sqrt n: M N lemma: rot}, which, when applied to this case, would imply
\[\|\eta(\mathbf H_1)\|_{op}\leq C\B^2\lb \sqrt{\dfrac{p+q}{n}}\vee \dfrac{p+q}{n}\rb e^{-\Thres^2/K}\]
provided $\Thres\in[36\B^2,\sqrt{K\log(\max{p+r,q})}/2]$ and $K\geq 1288\B^4$. Since $r<\min\{p,q\}$, the upper bound of $\Thres$ becomes $\sqrt{K\log(2\max\{p,q\})}/2$.
The proof for $\|\eta(\mathbf H_2)\|_{op}$ will follow in a similar way, and hence skipped. 

 Letting
 \begin{align*}
     A_1=&\ \Lambda^{1/2}U^T,\quad A_2=\Sx\tU B_x(\tU_{E_1*})^T,\\
     A_3=&\ \Sy\tV B_y(\tV_{F_2*})^T,
 \end{align*}
we note that \eqref{def: Operator: S3} implies $\mathbf H_1=A_1^T\mZ^TZ_2A_3+A_2^T\mZ_1^TZ_2A_3$, which can be written as
\begin{align*}
 \mathbf H_1=A_4^T\slb\begin{bmatrix}
   \mZ & \mZ_1
   \end{bmatrix}\srb^TZ_2A_3,\quad \text{where} \quad A_4=\begin{bmatrix}
   A_1\\ A_2
   \end{bmatrix}.
\end{align*}
We will now invoke Lemma~\ref{lemma: sqrt n: M N lemma: rot} because
 $Z_3=[\mZ\ \mZ_2]$ is a Gaussian data matrix with $n$ rows and $p+r\leq 2p$ columns, and the matrices $A_4$ and $A_3$ are also bounded in operator norm. To see the latter, first, noting  $\|A_4\|_{op}=\sqrt{\|A_4^TA_4\|_{op}}$, we observe that
\[\|A_4^TA_4\|_{op}= \|A_1^TA_1+A_2^TA_2\|_{op}\leq \|A_1\|_{op}^2+\|A_2\|_{op}^2.\]
Therefore it suffices to 
 bound the operator norms of $A_1$, $A_2$, and $A_3$ only. Using \eqref{eq: intheorem: CT: M}, we can show that the operator norm of the matrices of the form $A_2$ or $A_3$ is bounded by $\B$ for $(X,Y)\sim\mathbb P\in\mP(r,s_x,s_y,\B)$. Since $\Sx^{1/2} U$ has orthogonal columns, it can be easily seen that $\|A_1\|_{op}\leq 1$. Therefore
\[\|A_4\|_{op}\leq \|A_1\|_{op}+\|A_2\|_{op}\leq 1+\B\leq 2\B\]
because $\B>1$ as per the definition of $\mP(r,s_x,s_y,\B)$.
The proof of the first part  now follows by Lemma~\ref{lemma: sqrt n: M N lemma: rot}. Because
$\|A_4\|_{op}\leq 2\B$ and $\|A_3\|_{op}\leq\B$, the proof of the second part follows directly from Lemma~\ref{lemma: CT: Bickel levina case}, and hence skipped.
\end{proof}

 \subsection{Proof of Additional Lemmas for Section~\ref{sec: low deg} and Theorem~\ref{thm: low deg}}
\label{sec: add lemmas: low-deg}
\begin{proof}[Proof of Lemma~\ref{lemma: low deg: bayesian}]
To prove the current lemma, we will require a result on the concentration of $\alpha$ and $\beta$ under $\pi_x$ and $\pi_y$. To that end,  for $s,m\in\NN$ satisfying $s\leq m$,  let us define the set
 \[\mathcal W(s,m)=\lbs x\in\RR^m\ :\ \|x\|_0\in[s/2,2s],\|x\|_2\in[0.9,1.1]\rbs.\]
 Suppose $\pi_x$ and $\pi_y$ are the Rademacher priors on $\alpha$ and $\beta$ as defined in Section \ref{sec: low deg}. The following lemma then says that $\alpha$ and $\beta$ concentrates on $\mathcal W(s_x,p)$ and $\mathcal W(s_y,q)$ with probability tending to one.
 \begin{lemma}\label{lemma: can use radamander}
Suppose $s_x,s_y\to\infty$. Then
  \begin{equation}\label{limit: pi x}
     \lim_n\pi_x(\alpha\in\mathcal W(s_x,p))=1;\quad \lim_n\pi_y(\beta\in\mathcal W(s_y,q))=1.
 \end{equation}
 \end{lemma}
 Here the probability $\pi_x(\alpha\in\mathcal W(s_x,p))$ depends on $n$ through $s_x$ and $p$. Similarly $\pi_y(\beta\in\mathcal W(s_y,q))$ depends on $n$ through $s_y$ and $q$.
 
Recall the definition of   $\PP_{\alpha,\beta}$ from  \eqref{def: low deg: PP alpha beta}.
  Let us consider the class
  \[\mP_{sub}(\B)=\lbs \PP_{\alpha,\beta}\ :\  \alpha\in \mathcal W(s_x,p),\ \beta\in\mathcal W(s_y,q)\rbs.\]
   If $\alpha\in\mathcal W(s_x,p)$ and $\beta\in \mathcal W(s_x,p)$, than $\|\alpha\|_2\|\beta\|_2\leq (1.1)^2<\B$ because $\B>2$. 
Therefore \eqref{def: low deg: PP alpha beta} implies that $(X,Y)\sim\PP\in\mP_{sub}(\B)$ has canonical correlation $\B^{-1}$. Thus $\mP_{sub}(\B)\subset\mP_G(r,2s_x,2s_y,\B)$, implying
\begin{align*}
  \MoveEqLeft  \liminf_n \sup_{\PP_n\in \mP_G(r,2s_x,2s_y,\B)^n}\PP_n\slb\Phi_n(\X,\Y)=1)\srb\\
  \geq &\ \liminf_n \sup_{\PP_n\in \mP_{sub}(\B)}\PP_n\slb\Phi_n(\X,\Y)=1)\srb.
\end{align*}
  Suppose $\mathcal F_x$ and $\mathcal F_y$ are the Borel $\sigma$-field associated with $\mathcal W(s_x,p)$ and $\mathcal W(s_y,q)$, respectively. 
Define the probability measures $\tilde \pi_x$ and $\tilde \pi_y$ on  $( \mathcal W(s_x,p), \mathcal F_x)$ and  $(\mathcal W(s_y,q), \mathcal F_y)$, respectively, by \[\tilde\pi_x(A)=\frac{\pi_x(A)}{\pi_x(\mathcal W(s_x,p))}\quad \text{for all }A\in \mathcal F_x, \]
and
\[\tilde\pi_y(B)=\frac{\pi_y(B)}{\pi_y(\mathcal W(s_y,q))}\quad\text{for all }  B\in \mathcal F_y.\]
Note also that if $\alpha\in \mathcal W(s_x,p)$ and $\beta\in\mathcal W(s_y,q)$, then $\PP_{\alpha,\beta}\in \mP_{sub}(\B)$.
Therefore 
 \begin{align*}
 &  \liminf_n \sup_{\PP_n\in \mP_{sub}(\B)}\PP_n(\Phi_n(\X,\Y)=1))\\
 \geq &\ \liminf_n \dint_{\substack{\W(s_x,p)\\\times\W(s_y,q)}} \PP_{n,\alpha,\beta}\slb \Phi_n(\X,\Y)=1\srb d\tilde\pi_x(\alpha)d\tilde\pi_y(\beta)\\
    =&\ \frac{\liminf_n \dint_{\substack{\W(s_x,p)\\\times\W(s_y,q)}} \PP_{n,\alpha,\beta}\slb \Phi_n(\X,\Y)=1\srb d\pi_x(\alpha)d\pi_y(\beta)}{\limsup_n \slb \pi_x(\mathcal W(s_x,p))\pi_y(\mathcal W(s_y,q)) \srb},
    \end{align*}
   whose denominator is one by Lemma~\ref{lemma: can use radamander}.
Denoting $\W(s_y,q)^c=\RR^p\setminus \W(s_y,q)$, we note that
\begin{align*}
 \MoveEqLeft  \dint_{\RR^p\times\W(s_y,q)^c} \PP_{n,\alpha,\beta}\slb \Phi_n(\X,\Y)=1\srb d\pi_x(\alpha)d\pi_y(\beta)\\
   \leq &\ 1-\pi_y(\W(s_y,q))\to_n 0  
\end{align*}
   by Lemma~\ref{lemma: can use radamander}. Similarly, denoting $\W(s_x,p)^c=\RR^p\setminus \W(s_x,p)$, we can show that
   \[ \dint_{\W(s_x,p)^c\times\RR^y} \PP_{n,\alpha,\beta}\slb \Phi_n(\X,\Y)=1\srb d\pi_x(\alpha)d\pi_y(\beta)\to_n 0.\]
   Therefore, it holds that
   \begin{align*}
       \liminf_n \dint\limits_{\W(s_x,p)\times\W(s_y,q)} \PP_{n,\alpha,\beta}\slb \Phi_n(\X,\Y)=1\srb d\pi_x(\alpha)d\pi_y(\beta)\\
       =\& \liminf_n \E_{\pi}\slbt\PP_{n,\alpha,\beta}\slb \Phi_n(\X,\Y)=0\srb\srbt.
   \end{align*}
Thus the proof follows. 
 
 \end{proof}
\subsubsection*{Proof of Lemma~\ref{lemma: can use radamander}}

  \begin{proof}[Proof of Lemma~\ref{lemma: can use radamander}.]
    We are  going to show \eqref{limit: pi x} only for $\pi_x$ because the proof for $\pi_y$ follows in the identical manner. Throughout we will denote by $\E_{\pi_x}$ and $\text{var}_{\pi_x}$ the expectation and variance under $\pi_x$.
Note that when $\alpha\sim \pi_x$,
$\|\alpha\|_0=\sum_{i=1}^{p}I[\alpha_i\neq 0]$, 
where $I[\alpha_i\neq 0]$'s are i.i.d. Bernoulli random variables with success probability $s_x/p$. Therefore, Chebyshev's inequality yields that for any $\e>0$,
\[\pi_x\slb \sbl\la\alpha\ra_0-s_x\sbl>s_x\e\srb\leq p\frac{\text{var}_{\pi_x}(I[\alpha_i\neq 0])}{s_x^2\e^2}=\frac{1-s_x/p}{s_x\e^2},\]
which goes to zero if $s_x\to\infty$.
Therefore, for $\epsilon=1/2$, we have
\[\pi_x\slb \|\alpha\|_0\in [s_x/2,2 s_x]\srb \leq \pi_x\slb \sbl\la\alpha\ra_0-s_x\sbl>s_x\e\srb\to 0.\]
Also, since  $\E_{\pi_x}[\sum_{i=1}^p\alpha_i^2]=1$, Chebyshev's inequality implies that
\begin{align*}
  \pi_x\slb \sum_{i=1}^p\alpha_i^2-1\geq \e\srb
   \leq &\ \frac{\text{var}_{\pi_x} \slb\sum_{i=1}^p\alpha_i^2\srb}{\e^2}\\
    &\ \stackrel{(a)}{=}\frac{p.\text{var}_{\pi_x}(\alpha_i^2)}{\e^2}  \leq \frac{p\E_{\pi_x}[\alpha_i^4]}{\e^2}=\frac{1}{s_x\e^2},
\end{align*}
which goes to zero if $s_x\to\infty$ for any fixed $\e>0$. Here (a) uses the fact that $\alpha_i$'s are i.i.d. The proof now follows setting $\e=0.1$.


\end{proof}

 \subsubsection{Proof of Lemma~\ref{lemma: low deg: dot product}}
 Proof of Lemma depends on two auxiliary lemmas. We state and prove these lemmas first.
 \begin{lemma}\label{lemma: low deg: hermite poly}
  Suppose $\myv\in \ZZ^m$, and $A\in \RR^{m\times m}$ is a matrix. Let $\PP$ be the measure induced by the $m$-dimensional standard Gaussian random vector and denote by $\E_{\PP}$ the corresponding expectation. Then for any $x\in\RR^m$ we have
   \[\sum_{j\in \ZZ^m} \frac{t^j}{j!}E_{ \PP}[H_j(AZ)]=e^{t^T(A^2-I)t/2}.\]
  \end{lemma}
  
  \begin{proof}[Proof of Lemma~\ref{lemma: low deg: hermite poly}]
  The generating function of $H_{\myv}$ has the convergent expansion \cite[Proposition 6]{rahman2017}
  \[\sum_{j\in \ZZ^m} \frac{t^j}{j!}H_j(x)=\exp\lbs t^Tx-t^Tt/2\rbs\]
 for any $x\in\RR^m$.  Therefore,
  \[\sum_{j\in \ZZ^m} \frac{t^j}{j!}H_j(Ax)=\exp\lbs t^TAx-t^Tt/2\rbs.\]
  Multiplying both side by the density $d\PP$ of $\PP$ and then integrating over $\RR^m$ gives us
  \[\sum_{j\in \ZZ^m} \frac{t^j}{j!}E_{ \PP}[H_j(AZ)]=E_{ \PP}\left[e^{ t^TAZ}\right]e^{-t^Tt/2}=e^{t^T(A^2-I)t/2}.\]
  \end{proof}
  
  \begin{lemma}\label{lemma: low deg: mgf lemma}
  Let $\Sigma(\alpha,\beta,1/\B)$ be as defined in \eqref{def: sigma in low deg}.
  Suppose $z=(z_x,z_y)$ where $z_x\in\ZZ^{p}$ and $z_y\in\ZZ^q$. Then for any $t\in\RR^{p+q}$, we have
  \begin{align*}
    \MoveEqLeft  \partial^{z}_t \exp\left\{\frac{1}{2}t^T\slb \Sigma (\alpha,\beta,1/\B) -I_{p+q}\srb t\right\}\bl_{t=(0,\ldots,0)}\\
      =&\ \begin{cases}
  \B^{-|z_x|}|z_x|!\alpha^{z_x}\beta^{z_y} & \text{if }|z_x|=|z_y|,\\
  0 & o.w.
   \end{cases}
  \end{align*}
  \end{lemma}
  \begin{proof}[Proof of Lemma~\ref{lemma: low deg: mgf lemma}]
  Let us partition $t$ as $(t_x,t_y)$ where $t_x=(t_x(1),\ldots,t_x(p))\in\RR^p$ and $t_y=(t_y(1),\ldots,t_y(q))\in\RR^q$. We then calculate
   \[
  \frac{ t^T\slb \Sigma(\alpha,\beta,1/\B)-I_{p+q}\srb t}{2}=\B^{-1} t_x^T\alpha\beta^T t_y,
 \]
   which implies
   \begin{align*}
 \MoveEqLeft \exp\left\{\frac{1}{2}t^T\slb \Sigma (\alpha,\beta,1/\B) -I_{p+q}\srb t\right\}\\
  =&\ \exp \slbs\B^{-1} \sum_{i=1}^{p}\sum_{j=1}^{q} \alpha_i\beta_j t_x(i)t_y(j)\srbs  \\
  =&\ \sum_{k=0}^\infty \B^{-k} \frac{\slb \sum_{i=1}^{p} \alpha_i t_x(i)\srb^k \slb \sum_{j=1}^{q} \beta_j t_y(j)\srb^k}{k!}
  \end{align*}
  which equals
  \begin{align*}
  \MoveEqLeft \sum_{k=0}^{\infty}\frac{\B^{-k}}{k!}\sum_{\substack{z_x\in\ZZ^p,\\|z_x|=k}}\sum_{\substack{z_y\in\ZZ^q,\\|z_y|=k}}\frac{k!}{z_x!}\frac{k!}{z_y!}\alpha^{z_x}\beta^{z_y}t_x^{z_x}t_y^{z_y}\\
  =&\ \sum_{k=0}^{\infty}\sum_{\substack{z_x\in\ZZ^p,\\|z_x|=k}}\sum_{\substack{z_y\in\ZZ^q,\\|z_y|=k}}{\B^{-k}}\frac{k!}{z!}\alpha^{z_x}\beta^{z_y}t_x^{z_x}t_y^{z_y}\\
  \stackrel{(a)}{=}&\ \sum_{\substack{z\in\ZZ^{p+q}\\ |z_x|=|z_y|}}\B^{-|z_x|}\frac{|z_x|!}{z!}\alpha^{z_x}\beta^{z_y}t^z.
   \end{align*}
   In step (a), we stacked the variables $z_x$ and $z_y$ to form $z=(z_x,z_y)^T$. Note that following the terminologies set in the beginning of Appendix~\ref{sec; proof of low deg}, $z!=z_x!z_y!$ and $t^z=t_x^{z_x}t_y^{z_y}$. Note that if $|z_x|\neq|z_y|$, then the term $t^z$ has zero coefficient in the above expansion.
  Thus the lemma follows.
  \end{proof}
 \begin{proof}[Proof of Lemma~\ref{lemma: low deg: dot product}]
   \begin{align*}
     \langle \LL_n, H_\myv\rangle_{L^2(\QQ_n)}=& \ \E_{(\X,\Y)\sim \QQ_n}\lbt \E_{\pi}\lbt H_\myv(\X,\Y) \frac{d \PP_{n,\alpha,\beta}}{d \QQ_n}\rbt \rbt\\
     =&\ \E_{\pi}\slbt \E_{(\X,\Y)\sim \PP_{n,\alpha,\beta}}\slbt H_\myv (\X,\Y)\srbt\srbt\\
     \stackrel{(a)}{=}&\ \E_{\pi} \lbt\E_{\substack{(X_i,Y_i)\sim \PP_{\alpha,\beta},\\ i\in[n]}}\lbt\prod_{i\in[n]}H_{\myv_i}(X_i,Y_i)\rbt\rbt\\
     =&\ \E_{\pi}\lbt \prod_{i\in[n]}\E_{(X_i,Y_i)\sim \PP_{\alpha,\beta}}\slbt H_{\myv_i} (X_i,Y_i)\srbt\rbt
     \end{align*}
     where (a) follows because $(X_i,Y_i)$'s are independent observations.
     Now note that  if $\|\alpha\|\|\beta\|_2\geq \B$, then \eqref{def: low deg: PP alpha beta} implies
      \[\E_{(X_i,Y_i)\sim \PP_{\alpha,\beta}}\slbt H_{\myv_i} (X_i,Y_i)\srbt=\E_{(X_i,Y_i)\sim \QQ}\slbt H_{\myv_i} (X_i,Y_i)\srbt=0,\]
       where the last step follows because $\E_{Z\sim\QQ}[H_{\myv_i}(Z)]=0$ for any $i\in[n]$.
     If $\|\alpha\|\|\beta\|_2< \B$, then $\Sigma(\alpha,\beta, 1/\B)$ defined in \eqref{def: sigma in low deg} is positive definite, and \eqref{def: low deg: PP alpha beta} implies
      \begin{align*}
    \MoveEqLeft  \E_{(X_i,Y_i)\sim \PP_{\alpha,\beta}}\slbt H_{\myv_i} (X_i,Y_i)\srbt\\
    =&\ \E_{\mZ\sim \QQ}\lbt H_{\myv_i} \slb\Sigma(\alpha,\beta,1/\B)^{1/2}\mZ\srb\rbt   \\
      =&\ \  \partial^{\myv_i}_{t}\lb\exp\left\{\frac{1}{2}t^T\slb \Sigma (\alpha,\beta,1/\B) -I_{p+q}\srb t\right\}\rb\bl_{t=(0,\ldots,0)}
      \end{align*}
     by Lemma~\ref{lemma: low deg: hermite poly}. Here $\Sigma(\alpha,\beta,1/\B)$ is as in \eqref{def: sigma in low deg}, and $\Sigma(\alpha,\beta,1/\B)$ is positive definite because $\|\alpha\|_2\|\beta\|_2<\B$, as discussed in Section~\ref{sec: low deg}.
    %
Therefore, we can write
\begin{align*}
  \MoveEqLeft  \E_{(X_i,Y_i)\sim \PP_{\alpha,\beta}}\slbt H_{\myv_i} (X_i,Y_i)\srbt\\
    =&\ 1\{\|\alpha\|_2\|\beta\|_2<\B\}\\
    &\ \times\partial^{\myv_i}_{t}\lb\exp\left\{\frac{1}{2}t^T\slb \Sigma (\alpha,\beta,1/\B) -I_{p+q}\srb t\right\}\rb\bl_{t=(0,\ldots,0)}
\end{align*}
 Lemma~\ref{lemma: low deg: mgf lemma} gives the form of the partial derivative in the above expression, and implies that the partial derivative is zero unless $|\myv^x_i|=|\myv^y_i|$.
 Therefore, $\langle \LL_n, H_\myv\rangle_{L^2(\QQ_n)}\neq 0$ only if $|\myv^{x}_i|=|\myv^{y}_i|$ for all $i\in[n]$. In this case, $|\myv_i|=2|\myv^{x}_i|$ is even, and by  Lemma~\ref{lemma: low deg: mgf lemma},
  \begin{align*}
    \MoveEqLeft  \langle \LL_n, H_\myv\rangle_{L^2(\QQ_n)}\\
      =& \E_{\pi}\lbt1\{\|\alpha\|_2\|\beta\|_2<\B\}\prod_{i\in[n]}\B^{-|\myv^{x}_i|}|\myv^{x}_i|!\alpha^{\myv^{x}_i}\beta^{\myv^{y}_i}\rbt\\
      =&\ \lbs\B^{-\sum_{i=1}^n|\myv^{x}_i|}\prod_{i=1}^n|\myv^{x}_i|!\rbs\\
      &\ \times\E_{\pi}\lbt 1\{\|\alpha\|_2\|\beta\|_2<\B\}\alpha^{\si\myv^{x}_i}\beta^{\si\myv^{y}_i}\rbt \\
      =&\ \B^{-|\myv|/2}\lbs\prod_{i=1}^n{|\myv^{x}_i|!}\rbs\\
      &\ \times\E_{\pi}\lbt 1\{\|\alpha\|_2\|\beta\|_2<\B\}\alpha^{\si\myv^{x}_i}\beta^{\si\myv^{y}_i}\rbt.
  \end{align*} 
  Therefore,
  \begin{align*}
    \MoveEqLeft   \langle \LL_n, \widehat H_\myv\rangle_{L^2(\QQ_n)}^2\\
       =&\ \begin{cases}\frac{\B^{-|\myv|}}{\myv!}\E_{\pi}\lbt 1\{\|\alpha\|_2\|\beta\|_2<\B\}\alpha^{\sum_{i=1}^n\myv^{x}_i}\beta^{\sum_{i=1}^{n}\myv^{y}_i}\rbt^2 \\
       \times \lbs\prod_{i=1}^n{|\myv^{x}_i|!}\rbs^2 \quad  \text{ if }|\myv^{x}_i|=|\myv^{y}_i|
  \text{ for all } i\in[n],\\
  \\
  0 \quad \text{ o.w.}\end{cases}
  \end{align*}
  
   \end{proof}

   \subsubsection{Proof of Lemma~\ref{lemma: low deg: projection}}
   \begin{proof}
 Lemma~\ref{lemma: low deg: dot product} implies that $\LL_n$ belongs to  the subspace generated by those  $H_\myv$'s whose degree-index $\myv$ has $|\myv^{x}_i|=|\myv^{y}_i|$ for all $i\in[n]$. The degree of the polynomial $H_\myv$ is $|\myv|$, which is even in the above case. Therefore, if $D_n\geq 1$ is odd,
 $\|\LL_n^{\leq D_n}\|^2_{L_2(\QQ_n)}$ equals $\|\LL_n^{\leq (D_n-1)}\|^2_{L_2(\QQ_n)}$.
 Hence, it suffices to compute the norm of  $\LL_n^{\leq 2\D_n}$, where $\D_n=\floor*{D_n/2}$. Suppose $\myv\in\ZZ^{n(p+q)}$ is such that $|\myv^{x}_i|=|\myv^{y}_i|$ for all $i\in[n]$.  Lemma~\ref{lemma: low deg: dot product} gives 
 \begin{align*}
    \MoveEqLeft \langle \LL_n, \widehat H_\myv\rangle_{L^2(\QQ_n)}^2\\
     =&\ \frac{\B^{-|\myv|}}{\myv!}\lbs\E_{\pi}\slbt1\{\|\alpha\|_2\|\beta\|_2<\B\}\alpha^{\sum_{i=1}^n\myv^{x}_i}\beta^{\sum_{i=1}^{n}\myv^{y}_i}\srbt\rbs^2\\
     &\ \times\lbs\prod_{i=1}^n{|\myv^{x}_i|!}\rbs^2.
 \end{align*}
 Consider the pair of replicas $\alpha_1,\alpha_2\iid \pi_x$ and $\beta_1,\beta_2\iid\pi_y$. Letting $\myW$ denote the indicator function of  the event
 $\{\|\alpha_1\|_2\|\beta_1\|_2<\B,\ \|\alpha_2\|_2\|\beta_2\|_2<\B\}$,
 we can then write
 \begin{align}\label{inlemma: low deg: replica}
     \langle \LL_n, \widehat H_\myv\rangle_{L^2(\QQ_n)}^2=&\ \frac{\B^{-|\myv|}}{\myv!}\E_{\pi}\slbt(\alpha_1\alpha_2)^{\sum_{i=1}^n\myv^{x}_i}(\beta_1\beta_2)^{\sum_{i=1}^{n}\myv^{y}_i}\myW\srbt\nn\\
     &\ \times\lbs\prod_{i=1}^n{|\myv^{x}_i|!}\rbs^2.
 \end{align}
 Denote by $\overline{d}=(d_1,\ldots,d_n)\in\ZZ^n$. Using \eqref{inlemma: low deg: replica}, we obtain the following expression:
 \begin{align*}
 \MoveEqLeft   \|\LL_n^{\leq 2\D_n}\|_{L_2(\QQ)}=\sum_{d=0}^{\D_n}\B^{-2d}\sum_{\overline d:\sum d_i=d}\sum_{\substack{\myv: \myv^{x}_i\in \ZZ^p,\\ \myv^{y}_i\in\ZZ^q,\\
    |\myv^{x}_i|=|\myv^{y}_i|=d_i}}T_{\overline d, w}
    \end{align*}
    where
    \[T_{\overline d, w}=\E_{\pi}\lbt \myW\prod_{i=1}^{n}\lb\frac{d_i^2}{\myv^{x}_i!\myv^{y}_i!}(\alpha_1\alpha_2)^{\myv^{x}_i}(\beta_1\beta_2)^{\myv^{y}_i}\rb\rbt.\]
 Therefore $ \|\LL_n^{\leq 2\D_n}\|_{L_2(\QQ)}$   equals
    \begin{align*}
    &\sum_{d=0}^{\D_n}\B^{-2d}\sum_{\overline d:\sum d_i=d} \E_{\pi}\lbt\myW\sum_{\substack{\myv: \myv^{x}_i\in \ZZ^p,\\ \myv^{y}_i\in\ZZ^q,\\
    |\myv^{x}_i|=|\myv^{y}_i|=d_i}}
    \lb \prod_{i=1}^{n}\frac{d_i!}{\myv^{x}_i!}(\alpha_1\alpha_2)^{\myv^{x}_i}\rb \\
    &\ \times
\lb \prod_{i=1}^{n}\frac{d_i!}{\myv^{y}_i!}(\beta_1\beta_2)^{\myv^{y}_i}\rb\rbt\\
=&\ \sum_{d=0}^{\D_n}\B^{-2d}\sum_{\overline d:\sum d_i=d} \E_{\pi}\lbt\myW \lb\sum_{\substack{w^x:\myv^{x}_i\in \ZZ^p\\ |\myv^{x}_i|=d_i}}\prod_{i=1}^{n}\frac{d_i!}{\myv^{x}_i!}(\alpha_1\alpha_2)^{\myv^{x}_i}\rb\\
&\ \times\lb\sum_{\substack{w^y:\myv^{y}_i\in \ZZ^q\\ |\myv^{y}_i|=d_i}}\prod_{i=1}^{n}\frac{d_i!}{\myv^{y}_i!}(\beta_1\beta_2)^{\myv^{y}_i}\rb\rbt
\end{align*}
In the last step, we used the variables $w^x=(w^x_1,\ldots,w^x_n)$, and $w_y=(w^y_1,\ldots,w^y_n)$. Suppose $z_i\in\ZZ^p$ for each $i\in[n]$.  For any $x\in\RR^p$ and $y\in\RR^q$, it holds that
\begin{align*}
   \MoveEqLeft \sum_{\substack{ z_i\in \ZZ^{p},|z_i|=d_i}}\prod_{i=1}^n\frac{d_i!}{z_i!}x^{z_i}y^{z_i}\\
    =&\ \prod_{i=1}^n \sum_{\substack{ z_i\in \ZZ^{p},|z_i|=d_i}}\frac{d_i!}{z_i!}x^{z_i}y^{z_i}\\
    \stackrel{(a)}{=}&\ \prod_{i=1}^n(x^Ty)^{d_i}=(x^Ty)^{\si d_i},
\end{align*}
where (a) follows from Fact~\ref{fact: multinomial theorem}.
  \begin{fact}\label{fact: multinomial theorem}[Multinomial Theorem]
 Suppose $\alpha\in\RR^p$. Then for $m\in\ZZ$, 
 \[\slb\sum_{i=1}^p \alpha_i\srb^m=\sum_{z\in\ZZ^p, |z|=m}\frac{k!\alpha^z}{z!}.\]
 \end{fact}
 Therefore it follows that
 \begin{align*}
     \MoveEqLeft \lb\sum_{\substack{w^x:\myv^{x}_i\in \ZZ^p\\ |\myv^{x}_i|=d_i}}\prod_{i=1}^{n}\frac{d_i!}{\myv^{x}_i!}(\alpha_1\alpha_2)^{\myv^{x}_i}\rb\lb\sum_{\substack{w^y:\myv^{y}_i\in \ZZ^q\\ |\myv^{y}_i|=d_i}}\prod_{i=1}^{n}\frac{d_i!}{\myv^{y}_i!}(\beta_1\beta_2)^{\myv^{y}_i}\rb\\
 =&\ (\alpha_1^T\alpha_2)^{\si d_i}(\beta_1^T\beta_2)^{\si d_i},
 \end{align*}
 which implies
\begin{align*}
\MoveEqLeft \|\LL_n^{\leq 2\D_n}\|_{L_2(\QQ)}\\
=&\ \sum_{d=0}^{\D_n}\B^{-2d}\sum_{\overline d:\sum d_i=d} \E_{\pi}\lbt \myW(\alpha_1^T\alpha_2)^{\si d_i}(\beta_1^T\beta_2)^{\si d_i}\rbt\\
\stackrel{(a)}{=}&\ \sum_{d=0}^{\D_n}\B^{-2d}{d+n-1\choose d}\E_{\pi}\lbt W(\alpha_1^T\alpha_2)^{d}(\beta_1^T\beta_2)^{d}\rbt\\
= &\ \E_{\pi}\lbt W\sum_{d=0}^{\D_n}\lbs{d+n-1\choose d}\lb \B^{-2}(\alpha_1^T\alpha_2)(\beta_1^T\beta_2)\rb^d\rbs\rbt.
 \end{align*}
 where $(a)$ follows since the number of $\overline d\in \ZZ^n$ such that $|\overline d|=d$ equals ${n+d-1\choose d}$. Noting  $\D_n=\floor*{D_n/2}$, the proof follows.
 \end{proof} 
\subsection{Proof of Technical Lemmas for Theorem~\ref{thm: CT matrix operator norm}}
\label{sec: add lemma: operator norm}

First, we introduce some additional notations and state some useful results that will be used repeatedly throughout the proof.
Suppose $A\in\RR^{p\times q}$. We can write $A$ as
\[A=\begin{bmatrix}
A_{*1} & A_{*2} &\cdots & A_{*q}
\end{bmatrix}.\]
  We define the vectorization operator as
  \[\vec(A)=\begin{bmatrix}
  A_{*1}\\ \cdots\\ A_{*q}
  \end{bmatrix}.\]
  We will use two well known operations on the vetorization operators, which follow from Section 10.2.2 of \cite{matrixcookbook}.
  \begin{fact}\label{fact: vectorize operator}
  \begin{itemize}
      \item[A.]  $Trace(A^TB)=\vec(A)^T\vec(B).$
      \item[B.] $\vec(AXB)=(B^T\otimes A)\vec(X)$
      where $\otimes$ denotes the Kronecker delta product.
  \end{itemize}
 \end{fact}
 Often times we will also use the fact that \cite[Theorem 13.12]{laub2005matrix}
  \begin{equation}
      \|A\otimes B\|_{op}=\|A\|_{op}\|B\|_{op}.
  \end{equation}
  Define the Hadamard product between vectors 
  $x=(x_1,\ldots,x_p)$ and $y=(y_1,\ldots,y_p)$ by
  \[x\circ y=(x_1y_1,\ldots,x_py_p)^T.\]
  Note that Cauchy-Schwarz inequality implies that
  \begin{equation}\label{hadamard: cauchy schwarz}
      \|x\circ y\|_2\leq \|x\|_2\|y\|_2
  \end{equation}
 We will also often use of Fact~\ref{Fact: Frobenius norm inequality}, which states $\|AB\|_F^2\leq \|A\|_{op}^2\|B\|_F^2$.

\subsubsection{Proof of Lemma~\ref{lemma: technical}}
 \begin{proof}
The first result is immediate. For the second result,  denote by $x_D$ by the projection of $x$ on $R^D$.  Note that for any $x\in \RR^{m}$ and $y\in\RR^{p}$.
\begin{align*}
\dfrac{x^T D(A) y}{\|x\|.\|y\|}=\dfrac{x_{D_1}^TAy_{D_2}}{\|x\|.\|y\|}\leq \dfrac{x_{D_1}^TAy_{D_2}}{\|x_{D_1}\|.\|y_{D_2}\|}
\end{align*}
Thus the maximum singular value of $D(A)$ is smaller than that of $A$, indicating that
\[\|D(A)\|\leq \|A\|.\]
\end{proof}

\subsubsection{Proof of Lemma~\ref{lemma: normal data matrix}} 
First, we state and prove two facts, which are used in the proof of Lemma~\ref{lemma: normal data matrix}.
\begin{fact}
\label{fact: normal data mtx}
Suppose $A\in\RR^{n\times r}$, $B\in\RR^{p\times s}$ are potentially random matrices satisfying $A^TA=I_r$ and $B^TB=I_s$. Let $\mX\in\RR^{n\times p}$ be such that $r,s\leq p$, and $\mX\mid A,B$ is distributed as a standard Gaussian data matrix.  Then the matrix
$A^T\mX B\mid A, B$ is  distributed as a standard Gaussian data matrix.
\end{fact}

\begin{proof}[Proof of Fact \ref{fact: normal data mtx}]
$\mX\in\RR^{n\times p}$ is a Gaussian data matrix with covariance $\Sigma\in\RR^{p\times p}$ if and only if
\begin{equation}
 \label{infact: data matrix} 
 \vec(\mX^T)\sim N_{np}(0,I_n\otimes\Sigma).
\end{equation}
Now  
\[\vec((A^T\mX B)^T)=\vec(B^T\mX^TA)\stackrel{(a)}{=}(A^T\otimes B^T)\vec(\mX^T)\]
where (a) follows from Fact~\ref{fact: vectorize operator}B. However, since $(A^T\otimes B^T)\in\RR^{rs\times np}$, \eqref{infact: data matrix} implies
\[(A^T\otimes B^T)\vec(\mX^T)\mid A,B\sim N_{rs}\slb 0, (A^T\otimes B^T)(A\otimes B)\srb,\]
but
\[(A^T\otimes B^T)(A\otimes B)=A^TA\otimes B^TB=I_r\otimes I_s=I_{rs}.\]
Therefore, 
\[\vec((A^T\mX B)^T)\mid A,B\sim N_{rs}(0,I_{rs}).\]
 Then the result follows from  \eqref{infact: data matrix}.
\end{proof}
In the above fact, it may appear that $A^T\mX B$ is independent of matrices $A$ and $B$ since its conditional distribution is standard Gaussian. However, $A^T\mX B$  still depends on $A$ and $B$ through $r$ and $s$, which may be random quantities. 

\begin{fact}
\label{fact: normal data proj matrix}
Suppose $A\in\RR^{n\times k}$, $\mX\in\RR^{n\times p}$, $B\in\RR^{p\times s}$ are  such that conditional on $A$ and $B$, $\mX$ is distributed as a standard Gaussian data matrix. Further suppose that the rank of $A$ and $B$  are $a$ and $b$, respectively. Then the following assertion holds:
\[\frac{\|A^T\mX B\|_{op}}{\|A\|_{op}\|B\|_{op}}\leq \|\mathbb Z\|_{op}\]
where $\mathbb Z\mid A,B$ is distributed as  a standard Gaussian data matrix in $\RR^{a\times b}$.
\end{fact}

\begin{proof}[Proof of Fact \ref{fact: normal data proj matrix}]
Suppose $P_A$ and $P_B$ are the projection matrices onto the column spaces of $A$ and $B$, respectively. Then we can write $P_A=V_AV_A^T$ and $P_B=V_BV_B^T$, where $V_A\in\RR^{n\times a}$ and $V_B\in\RR^{p\times b}$ are matrices matrices with full column rank so that $V_A^TV_A=I_a$ and $V_B^TV_B=I_b$.  Writing $A=P_A A$ and $B=P_B B$,  we obtain that
\[\|A^T\mX B\|_{op}=\|A^TV_AV_A^T\mX V_BV_B^TB\|_{op}\]
which is bounded by
\[\|A\|_{op}\|V_A\|_{op}\|V_A^T\mX V_B\|_{op}\|V_B\|_{op}\|B\|_{op}.\]
That $\|V_A\|_{op}$ and $\|V_B\|_{op}$ are one follows from the definitions of $V_A$ and $V_B$. Fact~\ref{fact: normal data mtx} implies conditional on $V_A$ and $V_B$,  $V_A^T\mX V_B\in\RR^{a\times b}$ is  distributed as a standard Gaussian data matrix. Hence, the proof follows.
\end{proof}

\begin{proof}[Proof of Lemma~\ref{lemma: normal data matrix}]
 Let us denote the rank of $\mZ_1 D$ by $a'$. Note that $a'\leq \text{rank}(D)=a$.  Letting $A=\mZ_1D$, and applying Fact~\ref{fact: normal data proj matrix}, we have the bound
\[\|D^T\mZ_1^T\mZ_2B\|_{op}\leq \|\mZ_1 D\|_{op}\|\mZ\|_{op}\|B\|_{op}\]
where $\mZ\mid \mZ_1$ is distributed as a standard Gaussian data matrix in $\RR^{a'\times b}$. Next we apply Fact~\ref{fact: normal data proj matrix} again, but now on the term $\|\mZ_1 D\|_{op}$, which leads to
\[\|\mZ_1 D\|_{op}\leq \|D\|_{op}\|\mZ '\|_{op},\]
where $\mZ'\in \RR^{n\times a}$ is a standard Gaussian data matrix. Therefore,
\[\|D^T\mZ_1^T\mZ_2B\|_{op}\leq \|A\|_{op}\|\mZ'\|_{op}\|\mZ\|_{op}\|B\|_{op}.\]
We use the Gaussian matrix concentration inequality in Fact~\ref{fact: Gaussian matrix concentration inequality} to show that with probability at least $1-\exp(-Cn)$,
$\|\mZ'\|_{op}\leq \sqrt{2}(\sqrt{n}+\sqrt{a})$.
Also,  for $\mZ\in\RR^{a'\times b}$, the first part of Fact~\ref{fact: Gaussian matrix concentration inequality} implies  
\[\mathbb P\slb \|\mZ\|_{op}\leq \sqrt{a'}+\sqrt{b}+t\mid \mZ_1\srb\geq 1-\exp(-t^2/2)\]
for any $t>0$.
Since $a'\leq a$, and $t$ is deterministic, the above implies 
\[\mathbb P\slb \|\mZ\|_{op}\leq \sqrt{a}+\sqrt{b}+t\srb\geq 1-\exp(-t^2/2).\]
  Hence, for any $t>0$, we have the following with probability at least $1-\exp(-Cn)-\exp(-t^2/2)$:
\[\|D^T\mZ_1^T\mZ_2B\|_{op}\leq \sqrt{2}\|D\|_{op}\|B\|_{op}(\sqrt{n}+\sqrt{a})(\sqrt{a}+\sqrt{b}+t).\]
Since $a\leq b\leq n$, it follows that 
\[\|D^T\mZ_1^T\mZ_2B\|_{op}\leq C\|D\|_{op}\|B\|_{op}\sqrt{n}\max\{\sqrt{b},t\}.\]
Therefore, the proof follows.
 \end{proof}

 \subsubsection{Proof of Lemma \ref{Adlemma: lowrank: premultiply: 1}}
 \begin{proof}[Proof of Lemma \ref{Adlemma: lowrank: premultiply: 1}]

Denoting
\[\mathcal T=\lbs (x',y')\in\RR^p\ :\ x'= x_{S_x}, y'=y_{S_y}, x\in T_p^\e, y\in T_q^\e\rbs,\]
we note that
 \[ \max_{x\in T^{\e}_p,y\in T^{\e}_q}|\langle x_{S_x}, \mA y_{S_y}\rangle|= | \max_{(x,y)\in T}|\langle x, \mA y\rangle|.\]
Therefore it suffices to show that
there exist absolute constants $C,c>0$ such that
\begin{align*}
\MoveEqLeft P\lbs \max_{(x,y)\in \mathcal T}|\langle x,\mA y\rangle |\geq \Delta\rbs \\
\leq &\ C \exp\lbs (p+q)\dfrac{\log (C\da)}{\da}-\dfrac{n^2\Delta^2}{4C\|\M\|_{op}^2\|\N\|_{op}^2(2n+p+q)}\rbs.\\
&\ + \dfrac{C}{\Delta^2}\|\M\|_{op}^2\|N\|_{op}^2(n(p+q))^C\lbs e^{ -c(n+q)}+ e^{ -c(n+p)}\rbs 
\end{align*}

Let us denote 
$\Z_1=\vec(\mZ_1^T)$, $\Z_2=\vec(\mZ_2^T)$, and $\Z=(\Z_1^T,\Z^T_2)^T$.
 Thus
\[\Z^T=\{(\Z_1)^T_{*1},\ldots, (\mZ_1)^T_{*p}, (\mZ_2)^T_{*1},\ldots, (\mZ_1)^T_{*q}\}.\]
Recalling
$\mQ=\frac{\M \mZ_1^T\mZ_2\N}{n}$, we define
\begin{equation}\label{inlemma: def: fx(z1,z2) in lemma 13}
  f_{x,y}(\Z_1,\Z_2)= \Big\langle x, \eta(\mQ) y\Big\rangle=\langle x, \mA y\rangle.   
 \end{equation}
To obtain a tight concentration inequality for $ f_{x,y}(\Z_1,\Z_2)$, we want to use the following Gaussian concentration lemma due to \cite{deshpande2014}

\begin{lemma}[Corollary 10 of \cite{deshpande2014}]
Let $\mathcal Z\sim N(0,I_n)$ be a vector of $n$ i.i.d. standard Gaussian variables. Suppose $\mathcal B$ is a finite set and we have functions $F_b:\RR^n\mapsto \RR$ for every $b\in\mathcal B$. Assume $\mathcal G\in\RR^n\times\RR^n$ is a Borel set such that for lebesgue-almost every $(Z,Z')\in\mathcal G:$
\[\max_{b\in\mathcal B}\sup_{t\in[0,1]}\|\grad F_b(\sqrt{t} Z+\sqrt{1-t}Z')\|_{2}\leq \mathcal L.\]
Then, there exists an absolute constant $C>0$ so that for any $\Delta>0$,
\begin{align*}
    \MoveEqLeft \mathbb P\slb \max_{b\in\mathcal B}\abs{F_b(\Z)-\E F_b(\Z)}\geq \Delta\srb\\
    \leq &\  C|\mathcal B|\exp(-\frac{\Delta^2}{C\mathcal L^2})\\
    &\ +\frac{C}{\Delta^2}\E\slbt\max_{b\in\mathcal B}(F_b(\Z)-F_b(\Z'))^4\srbt \mathbb P(\mathcal G^c)^{1/2}.
\end{align*}
Here $\Z'$ is an independent copy of $\Z$.
\end{lemma}

 In our case,  the index $b$ corresponds to $(x,y)$, the set $\mathcal B$ corresponds to $\mathcal T$, and the function $F_b(\Z)$ corresponds to $F_{x,y}(\Z)$.
 To find the centering and the Lipschitz constant $\mathcal L$, we need to compute $\E f_{x,y}(\Z_1,\Z_2)$ and $\triangledown_{\Z} f_{x,y}(\Z_1,\Z_2)$, respectively.

First, note that since $\mZ_1$ and $\mZ_2$ are independent standard Gaussian data matrices, $\mQ\stackrel{d}{=}-\mQ$. Noting $\E\eta(X)=0$ for any symmetric random variable $X$, we deduce 
\[\E\langle x, \mA y\rangle =\langle x, E [\eta(\mQ)]y\rangle=0.\]
Using Lemma~\ref{Surrogate lemma: Adlemma 1: derivative} we obtain that
\[\ns\pdv{f_{x,y}(\Z_1,\Z_2)}{\Z_1} \ns_2 \leq \|g(\mZ_2)\|_{op}\ns v\circ \grad\eta(\Vec(\mQ))\ns_2\]
and
\[\ns\pdv{f_{x,y}(\Z_1,\Z_2)}{\Z_2} \ns_2 \leq \|h(\mZ_1)\|_{op}\ns v\circ \grad\eta(\Vec(\mQ))\ns_2\]
where 
\[v=\Vec(xy^T),\quad g(\mZ_2)=\mZ_2\N\otimes \M^T/n,\]
\[  h(\mZ_1)=\mZ_1\M^T\otimes \N/n.\]
Because $|\grad \eta(x)|<1$ for each $x\in\RR$,
\begin{align*}
 \norm{ v\circ \grad\eta(\Vec(\mQ)) }_2\
\leq &\ \sup_x\grad|\eta(x)|\|v\|_2\\
\leq &\ \|v\|_2=\|x\|_2\|y\|_2
\end{align*}
since
$\|v\|_2^2=\|xy^T\|_F^2=\|x\|_2^2\|y\|^2_2$. 
Also, because $\|A\otimes B\|_{op}$ equals $\|A\|_{op}\|B\|_{op}$, we have
\begin{align}\label{inlemma: adlemma: g(z) inequality}
\|g(\mZ_2)\|^2_{op}=\ \frac{\|\mZ_2\N\otimes \M^T\|_{op}^2}{n^2}=&\ \frac{\|\mZ_2\N\|_{op}^2 \|M\|_{op}^2}{n^2}\nn\\
\leq &\  \frac{\|M\|_{op}^2 \|N\|_{op}^2 \|\mZ_2\|^2_{op}}{n^2}.
\end{align}
and similarly,
\begin{align}\label{inlemma: adlemma: h(z) inequality}
\|h(\mZ_1)\|^2_{op}\leq \frac{\|M\|_{op}^2 \|N\|_{op}^2 \|\mZ_1\|^2_{op}}{n^2}.
\end{align}
Therefore,
\[\ns\pdv{f_{x,y}(\Z_1,\Z_2)}{\Z_1} \ns_2\leq \|x\|_2\|y\|_2 \frac{\|M\|_{op} \|N\|_{op} \|\mZ_2\|_{op}}{n},\]
\[\ns\pdv{f_{x,y}(\Z_1,\Z_2)}{\Z_2} \ns_2\leq \|x\|_2\|y\|_2 \frac{\|M\|_{op} \|N\|_{op} \|\mZ_1\|_{op}}{n}.\]
Letting  $\grad f_{x,y}(\Z)$ denote $\pdv{f_{x,y}(\Z)}{\Z}$, we note that the
 above two inequalities imply 
\[\ns\grad f_{x,y}(\Z)\ns_2^2\leq \|x\|^2_2\|y\|^2_2 \frac{\|M\|^2_{op} \|N\|^2_{op} (\|\mZ_1\|^2_{op}+\|\mZ_2\|^2_{op})}{n^2}.\]
Because $\|x\|_2,\|y\|_2\leq 1$, we have
\begin{equation}\label{inlemma: ineq: del f x,y norm bound}
    \ns\grad f_{x,y}(\Z)\ns_2^2\leq  \frac{\|M\|^2_{op} \|N\|^2_{op} (\|\mZ_1\|^2_{op}+\|\mZ_2\|^2_{op})}{n^2}.
\end{equation}

We choose a good set $\mathcal{G}_1$ where the above bound is small. To that end, we take $\mathcal{G}_1$ to be
\begin{align}\label{inlemma: CT: def G1}
\G_1=\lbs &\ (\tilde \mZ_1, \tilde \mZ_1', \tilde \mZ_2, \tilde \mZ_2')\ :\  \tilde \mZ_1\in\RR^{n\times p}, \tilde \mZ_1'\in\RR^{n\times p}, \tilde \mZ_2\in\RR^{n\times q},\nn \\
&\ \tilde \mZ_2'\in \times \RR^{n\times q} ,
 \max\{\|\mZ_1\|_{op},\|\mZ_1'\|_{op}\}\leq \sqrt{2}(\sqn +\sqrt{p}),\nn\\
&\ \max\{\|\mZ_2\|_{op},\|\mZ_2'\|_{op}\}\leq \sqrt{2}(\sqn +\sqrt{q})\rbs.
\end{align}
Let us denote $\Z_i=\vec(\mZ_i^T)$ and $\tilde \Z_i=\vec(\tilde \mZ_i^T)$.
To apply Lemma~\ref{lemma: CT: deshpande}, now we define the process 
\[\Z_i(t)=\sqrt{t}\tilde \Z_i+\sqrt{1-t}\tilde \Z_i',\quad t\in[0,1], i=1,2.\]
Equation \ref{inlemma: ineq: del f x,y norm bound}
 implies that on $\G_1$,
\[\norm{\triangledown_\Z f_{x,y}\slb \Z_1(t),\Z_2(t)\srb}_{2}^2\leq \dfrac{4\|M\|_{op}^2\|N\|_{op}^2(2n+p+q)}{n^2}={\Lip}.\]

We are now in a position to apply Lemma~\ref{lemma: CT: deshpande}, which yields that
\begin{align}\label{inlemma: left: ss: bound}
\MoveEqLeft P\lb \max_{(x,y)\in \mathcal T}|f_{x,y}(\Z_1,\Z_2)|\geq \Delta \rb\nn\\
\leq &\ C|\mathcal T|\exp\lb-\dfrac{\Delta^2}{C\Lip^2}\rb+\dfrac{C}{\Delta^2}E\lbt \max_{(x,y)\in T} f_{x,y}(\Z_1,\Z_2)^4\rbt^{1/2}\nn\\
&\ \times P(\mathcal{G}^c_1)^{1/2}.
\end{align}
From equation 79 of \cite{deshpande2014} it follows that $C$ can be chosen so large such that
\[|\mathcal T|\leq \exp\lb (p+q)\dfrac{\log (C\da)}{\da}\rb.\]
Thus, after plugging in the value of ${\Lip}$, the first term on the right hand side of \eqref{inlemma: left: ss: bound} can be bounded above by
\[C \exp\lbs (p+q)\dfrac{\log (C\da)}{\da}-\dfrac{n^2\Delta^2}{4C\|\M\|^2\|\N\|^2(2n+p+q)}\rbs.\]
To bound the second term in \eqref{inlemma: left: ss: bound}, notice that
Lemma~\ref{suurogate lemma: operator 1} yields the bound
\[\E\lbt \max_{(x,y)\in \mathcal T} f_{x,y}(\Z_1,\Z_2)^4\rbt \leq C\|\M\|_{op}^4\|\N\|_{op}^4(n(p+q))^C,\]
whereas Fact~\ref{fact: Gaussian matrix concentration inequality} leads to the bound
\begin{equation}\label{inlemma: left: ss: G1c prob}
P(\G_1^c)^{1/2}\leq 2\slb\exp(-c(n+p))+\exp(-c(n+q))\srb.
\end{equation}
Therefore
the proof follows.
\end{proof}

\begin{lemma}\label{Surrogate lemma: Adlemma 1: derivative}
\label{lemma: CT: deshpande}
Suppose $f_{x,y}$ is as defined in \eqref{inlemma: def: fx(z1,z2) in lemma 13} and\\ $\mQ=M\mZ_1^T\mZ_2N/n$. Then
\[\ns\pdv{f_{x,y}(\Z_1,\Z_2)}{\Z_1} \ns_2 \leq \|g(\mZ_2)\|_{op}\ns v\circ \grad\eta(\Vec(\mQ)) \ns_2,\]
\[\ns\pdv{f_{x,y}(\Z_1,\Z_2)}{\Z_2} \ns_2 \leq\|h(\mZ_1)\|_{op}\ns v\circ \grad\eta(\Vec(\mQ))\ns_2\]
where 
$v=\Vec(xy^T)$, $ g(\mZ_2)=\mZ_2\N\otimes \M^T/n$,  and $h(\mZ_1)=\mZ_1M^T\otimes N/n$.
\end{lemma}
\begin{proof}
Using $v=\Vec(xy^T)$, and the fact that $Tr(AB)=Tr(BA)$,  we calculate that
\begin{align*}
f_{x,y}(\Z_1,\Z_2)=&\ Tr\lb yx^T\eta(\mQ)\rb
    \\
    =&\  Tr\lb (xy^T)^T\eta\slb \frac{\M \mZ_1^T\mZ_2\N}{n}\srb\rb\\
    =&\ \vec(xy^T)^T\vec\lb \eta\slb \frac{\M \mZ_1^T\mZ_2\N}{n}\srb\rb\\
    =&\ v^T\eta\lb \vec\slb \frac{\M \mZ_1^T\mZ_2\N}{n}\srb \rb.
\end{align*}
Fact~\ref{fact: vectorize operator} implies
\begin{equation}\label{inlemma: CT: Qmn form}
   \Vec( \mQ)=  \frac{(\N^T\mZ_2^T\otimes \M)}{n}\Z_1 \\
    =  g(\mZ_2)^T\Z_1,
\end{equation}
which yields $f_{x,y}(\Z_1,\Z_2)=v^T\eta(g(\mZ_2)^T\Z_1)$. 
 Noting $v\in\RR^{pq}$, we can hence write $f_{x,y}(\Z_1,\Z_2)$ as 
\[f_{x,y}(\Z_1,\Z_2)=\sum_{i=1}^{pq}v_i \eta\lb [g(\mZ_2)_i]^T\Z_1\rb.\]
Let us denote by $\grad \eta(x)$ the derivative of $\eta(x)$ evaluated at $x\in\RR$. For $A\in\RR^{p\times q}$, we denote by  $\grad \eta(A)$ the matrix whose $(i,j)$-th entry equals $\grad\eta(A_{i,j})$.
Then we obtain that for $j\in [np]$, 
\[\pdv{f_{x,y}(\Z_1,\Z_2)}{(\Z_1)_j}=\sum_{i=1}^{pq}v_i \grad\eta\lb [g(\mZ_2)_i]^T\Z_1\rb g(\mZ_2)_{ij},\]
indicating that
\begin{align*}
  \pdv{f_{x,y}(\Z_1,\Z_2)}{\Z_1}
 =&\ \sum_{i=1}^{pq}v_i \grad\eta\lb [g(\mZ_2)_i]^T\Z_1\rb g(\mZ_2)_{i}\\
 = &\ g(\mZ_2)\lbt v\circ \grad\eta\lb g(\mZ_2)^T\Z_1\rb\rbt
\end{align*}
where $\circ$ implies the Hadamard product.  It follows that
\[\ns\pdv{f_{x,y}(\Z_1,\Z_2)}{\Z_1} \ns_2 \leq \|g(\mZ_2)\|_{op}\ns v\circ \grad\eta\lb g(\mZ_2)^T\Z_1\rb \ns_2.\]
Then the first part of the proof follows from \eqref{inlemma: CT: Qmn form}. The proof of the second part follows similarly, and hence, skipped.

 Writing $v'=Vec(yx^T)$, we have
\begin{align*}
    f_{x,y}(\Z_1,\Z_2)=&\ Tr\lb \eta\slb \frac{\N^T\mZ_2^T\mZ_1\M^T}{n}\srb xy^T\rb\\
    = &\ Tr\lb xy^T\eta\slb \frac{\N^T\mZ_2^T\mZ_1\M^T}{n}\srb \rb,
     \end{align*}
    which equals
    \begin{align*}
  \MoveEqLeft  Tr\lb (yx^T)^T\eta\slb \frac{\N^T\mZ_2^T\mZ_1\M^T}{n}\srb\rb\\
    = &\ \vec(yx^T)^T\vec\lb \eta\slb \frac{\N^T \mZ_2^T\mZ_1\M^T}{n}\srb\rb\\
    =&\ (v')^T\eta\lb \vec\slb \frac{\N^T \mZ_2^T\mZ_1\M^T}{n}\srb \rb.
    \end{align*}
   Fact~\ref{fact: vectorize operator} implies that the above equals
    \begin{align*}
   (v')^T \eta \lb \frac{(\M \mZ_1^T\otimes \N^T)}{n}\Z_2 \rb
    =&\ (v')^T \eta \lb h(\mZ_1)^T\Z_2 \rb.
\end{align*}
where $h(\mZ_1)=\frac{\mZ_1M^T\otimes N}{n}$.
Thus, similarly we can show that
\begin{align*}
 \MoveEqLeft \ns\pdv{f_{x,y}(\Z_1,\Z_2)}{\Z_2}\ns_2\\
 \leq &\ \|h(\mZ_1)\|_{op} \ns v'\circ \grad\eta \lb h(\mZ_1)^T\Z_2 \rb\ns_2\\
 = &\ \|h(\mZ_1)\|_{op}\ns\vec\slb (xy^T)^T\srb\circ \vec\lb\grad\eta \lb \slbt\frac{\M \mZ_1^T\mZ_2\N}{n} \srbt\rb^T\rb\ns_2\\
 =&\ \|h(\mZ_1)\|_{op}\ns\vec\slb xy^T\srb\circ \vec\lb\grad\eta \lb \slbt\frac{\M \mZ_1^T\mZ_2\N}{n} \srbt\rb\rb\ns_2\\
 =&\ \|h(\mZ_1)\|_{op}\ns v\circ \grad\eta\lb g(\mZ_2)^T\Z_1\rb\ns_2.
\end{align*}
Therefore, the proof follows.
\end{proof}

\begin{lemma}\label{suurogate lemma: operator 1}
There exists an absolute constant $C$ so that the function $f_{x,y}$ defined in \eqref{inlemma: def: fx(z1,z2) in lemma 13} satisfies
\[\E\lbt \max_{\|x\|_2\leq 1, \|y\|_2\leq 1} f_{x,y}(\Z_1,\Z_2)^4\rbt \leq C\|\M\|_{op}^4\|\N\|_{op}^4(n(p+q))^C.\]
\end{lemma}

\begin{proof}
As usual, we let $\mQ=M\mZ_1^T\mZ_2N/n$.
Since $\|x\|_2,\|y\|_2\leq 1$, we have
\begin{align*}
 \MoveEqLeft f_{x,y}(\Z_1,\Z_2)^4\leq \norm{\eta(\mQ)}_{op}^4\stackrel{(a)}{\leq} \|\eta(\mQ)\|^4_F\\
&\   \stackrel{(b)}{\leq} \|\mQ\|_F^4\stackrel{(c)}{\leq} n^{-4}\|\M\|_{op}^4\|N\|_{op}^4\|\mZ_1\|_F^4\|\mZ_2\|_F^4.  
\end{align*}
Here (a) follows because the operator norm is smaller than the Frobenius norm, $(b)$ follows because  $|\eta(x)|\leq |x|$, and 
(c) follows from Fact~\ref{Fact: Frobenius norm inequality}.
Since $\mZ_1$ and $\mZ_2$ are independent,
\begin{align*}
  \MoveEqLeft  \E\lbt \max_{ \|x\|_2\leq 1, \|y\|_2\leq 1} f_{x,y}(\Z_1,\Z_2)^4\rbt \\
    \leq &\ n^{-4}\|\M\|_{op}^4\|\N\|_{op}^4\E[\|\mZ_1\|_F^4]\E[\|\mZ_2\|_F^4].
\end{align*}
Now note that since $\mZ_1$ and $\mZ_2$ are standard Gaussian data matrices, 
\[\E[\|\mZ_1\|_F^4]\leq \E \slbt Tr(\mZ_1^T\mZ_1)^2\srbt\leq k_1(n+p)^{k_2}\]
for some absolute constants $k_1$ and $k_2$. We can choose $C$ so large such that $k_1(n+p)^{k_2}\leq C(n+p)^{C}$. 
Similarly, we can show that 
\[\E[\|\mZ_2\|_F^4]\leq \E \slbt Tr(\mZ_2^T\mZ_2)^2\srbt\leq C(n+q)^{C},\]
implying
\[\E\lbt \max_{\|x\|_2\leq 1, \|y\|_2\leq 1} f_{x,y}(\Z_1,\Z_2)^4\rbt \leq C\|\M\|_{op}^4\|\N\|_{op}^4(n(p+q))^C\]
for sufficiently large $C$. 
\end{proof}

\subsubsection{Proof of Lemma~\ref{Adlemma: lowrank: premultiply: 2}}
\begin{proof}

The framework will be same as the proof of Lemma \ref{Adlemma: lowrank: premultiply: 1}.
Define $\T=\T_1\times  \mathcal T_2$
where
\[\T_1=\lbs x'\in \RR^p\ :\ x'=x_{S_x},\  x\in T_p^{\e} \rbs.\]
Let $\Z_1$, $\Z_2$, $\Z$, and $f_{x,y}$ be as in Lemma \ref{Adlemma: lowrank: premultiply: 1}. In this case, the main difference from Lemma~\ref{Adlemma: lowrank: premultiply: 1} is that  $|\T|$ is much larger. Eventually we will arrive at \eqref{inlemma: left: ss: bound} using the concentration inequality in Lemma~\ref{lemma: CT: deshpande}, but large $|\T|$ makes the right hand side of the   inequality in \eqref{inlemma: left: ss: bound} much larger. Therefore, we require a tighter bound on $\Lip$, which is  the bound on the Lipschitz constant of $\grad f_{x,y}(\Z)$  on the good set, so that the concentration inequality in \eqref{inlemma: left: ss: bound} is still useful. To bound the Lipschitz constant, as before, we bound $\|\triangledown f_{x,y}(\Z_1,\Z_2)\|^2$ using 
 Lemma~\ref{Surrogate lemma: Adlemma 1: derivative}, which implies that
\[\ns\pdv{f_{x,y}(\Z_1,\Z_2)}{\Z_1} \ns_2 \leq \|g(\mZ_2)\|_{op}\ns v\circ \grad\eta(\Vec(\mQ)) \ns_2,\]
where $v=\vec(xy^T)$ and $g(\mZ_2)=\mZ_2\N\otimes \M^T/n$.
From \eqref{inlemma: adlemma: g(z) inequality}
it follows that
\begin{equation}\label{inlemma: CT: g(Z2)}
   \|g(\mZ_2)\|_{op}^2\leq \frac{\|\M\|_{op}^2\|\N\|_{op}^2\|\mZ_2\|_{op}^2}{n^2}. 
\end{equation}
In Lemma~\ref{Adlemma: lowrank: premultiply: 1}, we bounded $\norm{ v\circ \grad\eta( \Vec(\mQ)) }_2$ by $\|v\|_2$, which was later bounded by $1$. 
 We require  a tighter bound on  $\norm{ v\circ \grad\eta( \Vec(\mQ)) }_2$ this time. Note that $\grad \eta(z)\leq 1\{|z|\geq \tau/\sqn\}$ at all $z\in\RR$ for  any directional derivative of $\eta$. Noting  $\|x\|_\infty\leq \sqrt{\da/p}$ for $x\in\T_1$, we deduce that any $A\in\RR^{p\times q}$ and  $(x,y)\in\T$ satisfy
 \begin{align*}
     \norm{ v\circ \grad\eta( \Vec(A)) }_2^2= &\ \sum_{i=1}^p\sum_{j=1}^qx_i^2y_j^2\eta(A_{i,j})^2\\
     \leq &\ \frac{\da}{p}\sum_{j=1}^qy_j^2\sup_{j\in[q]}\sum_{i=1}^p \eta(A_{i,j})^2\\
  =   &\  \frac{\da}{p}\|y\|_2^2\sup_{j\in[q]}\sum_{i=1}^p1\{|A_{i,j}|>\tau/\sqn\},
 \end{align*}
which is not greater than \[{\da}\sup_{j\in[q]}\sum_{i=1}^p1\{|A_{i,j}|>\tau/\sqn\}/p\]
because $\|y\|_2^2\leq 1$ for $y\in \mathcal T_2$.
Thus, it follows that
\begin{align*}
    \ns\pdv{f_{x,y}(\Z_1,\Z_2)}{\Z_1} \ns^2_2\leq &\  \frac{2\da\|\M\|_{op}^2\|\N\|_{op}^2\|\mZ_2\|_{op}^2}{pn^2}\\
    &\ \times\sup_{j\in[q]}\lbs \sum_{i=1}^p1\{\abs{(\mQ)_{i,j}}>\tau/\sqn\}\rbs.
\end{align*}
Similarly, we can show that
\begin{align*}
   \MoveEqLeft \ns\pdv{f_{x,y}(\Z_1,\Z_2)}{\Z_2} \ns^2_2\leq \frac{2\da \|\M\|_{op}^2\|\N\|_{op}^2\|\mZ_1\|_{op}^2}{pn^2}\\
    &\ \times\sup_{j\in[q]}\lbs \sum_{i=1}^p1\{\abs{(\mQ)_{i,j}}>\tau/\sqn\}\rbs.
\end{align*}
Thus,
\begin{align*}
    \ns\grad{f_{x,y}(\Z)} \ns^2_2\leq &\  \frac{2\da\|\M\|_{op}^2\|\N\|_{op}^2(\|\mZ_1\|_{op}^2+\|\mZ_2\|_{op}^2)}{n^2}\\
    &\ \times\sup_{j\in[q]}\frac{\lbs \sum_{i=1}^p1\{\abs{(\mQ)_{i,j}}>\tau/\sqn\}\rbs}{p}.
\end{align*}

We want to define the good set $\G_2$ of $(\tilde \mZ_1,\tilde \mZ_2,\tilde \mZ_1', \tilde \mZ_2')$ such that 
\[Z_i(t)=\sqrt{t}\tilde \mZ_i+\sqrt{1-t}\tilde \mZ_i',\quad t\in[0,1], i=1,2,\]
satisfies both
$\|\mZ_1(t)\|^2_{op}+\|\mZ_2(t)\|^2_{op}\leq 4(2n+p+q)$ and 
\[\sup_{j\in[q]} \sum_{i=1}^p1\{|(M\mZ_1(t)^T\mZ_2(t)N)_{i,j}|>\tau\sqn\} \leq 4pe^{-\tau^2/K}.\]
 We claim that the above holds if   $(\tilde \mZ_1,\tilde \mZ_2, \tilde \mZ_1', \tilde \mZ_2')\in\G_1$ defined in \eqref{inlemma: CT: def G1}, and for all $j\in[q]$,
 \begin{align}\label{inlemma: left: ssc: G2}
&\ \sum_{i=1}^{p}1\{|(M\tilde \mZ_1^T\tilde \mZ_2N)_{i,j}|>\tau\sqn/2\},\nn\\
&\ \sum_{i=1}^{p}1\{|(M (\tilde\mZ'_1)^T\tilde \mZ'_2N)_{i,j}|>\tau\sqn/2\}\leq 2p e^{-\tau^2/K}\nonumber \nn\\
&\ \sum_{i=1}^{p}1\{|(M (\tilde\mZ'_1)^T\tilde \mZ_2N)_{i,j}|>\tau\sqn/2\},\nn\\ 
&\ \sum_{i=1}^{p}1\{|(M\tilde (\mZ_1)^T\tilde \mZ'_2N)_{i,j}|>\tau\sqn/2\}\leq 2p e^{-\tau^2/K}. 
 \end{align}
The above claim follows from (89) and (90) of \cite{deshpande2014}. Therefore we define the good set $\mathcal G_2$ to be the subset of $\mathcal G_1$ where \eqref{inlemma: left: ssc: G2} is satisfied. 
Defining $\Z_1(t)=\vec(\mZ_1(t)^T)$ and $\Z_2(t)=\vec(\mZ_2(t)^T)$, we obtain that for some absolute constant $C>0$, it holds that
\begin{align*}
 \MoveEqLeft   \|\grad f_{x,y}(\Z_1(t),\Z_2(t))\|_2^2\\
 \le &\ q C\underbrace{\dfrac{\da(2n+p+q)\|\M\|_{op}^2\|N\|^2_{op}e^{-\tau^2/K_0}}{n^2}}_{\Lip^2}=C\Lip^2
\end{align*}
provided $\tilde \mZ_1$, $\tilde \mZ'_1$, $\tilde \mZ_2$, $\tilde \mZ_2'\in\mathcal G_2$.
 Similar to the proof of Lemma~\ref{Adlemma: lowrank: premultiply: 1}, using Lemma~\ref{lemma: CT: deshpande}, we obtain that there exists an absolute constant $C$ so that
 \begin{align}\label{inlemma: deshpande: 2nd}
 \MoveEqLeft P\lbs \max_{(x,y)\in \T}|f_{x,y}(\Z_1,\Z_2)|\geq \Delta \rbs\nn\\
  \leq &\  C|\T|\exp\lb-\dfrac{\Delta^2}{C\Lip^2}\rb\nn\\
  &\ +\dfrac{C}{\Delta^2}E\slbt \max_{(x,y)\in \T} f_{x,y}(\mZ_1,\mZ_2)^4\srbt^{1/2}P(\G_2^c)^{1/2}.  \end{align}
Now since $|\T|\leq|T_p^{\e}|\times |T_q^{\e}|$, and for any $k\in\NN$,  the $\e$-net  $T_k^{\e}$ is chosen so as to satisfy 
$|T_k^{\e}|\leq (1+2/\e)^{k}$, we have
 $|\T|\leq (1+2/\e)^{p+q}$. Therefore, we conclude that the first term of the bound in \eqref{inlemma: deshpande: 2nd} is not larger than
 \[C\exp\lb-\dfrac{\Delta^2}{C\Lip^2}+(p+q)\log(1+2/\e)\rb.\]

Rest of the proof is devoted to bounding the second term of the bound in \eqref{inlemma: deshpande: 2nd}. The expectation term can be bounded easily using
 Lemma~\ref{suurogate lemma: operator 1}, which yields 
\[\E\lbt \max_{(x,y)\in T} f_{x,y}(\mZ_1,\mZ_2)^4\rbt\leq C\|\M\|_{op}^4\|\N\|^4_{op}\{n(p+q)\}^C.\]
We will now show that $P(\G_2^c)$
is small.
Note that by definition,  $\G_2=\G_1\cap \mathcal V$, where $\mathcal V$ is the set of $(\tilde \mZ_1,\tilde \mZ_2, \tilde \mZ_1',\tilde \mZ_2')$, which satisfies the equation system \eqref{inlemma: left: ssc: G2}. Notice that by \eqref{inlemma: left: ss: G1c prob}, we already have $P(\G_1^c)\leq e^{-c(n+p)}+e^{-c(n+q)}$ for some $c>0$. Thus it suffices to show that  $P(\mathcal{V}^c)$ is small.
To this end, note that since $\tilde \mZ_1,\tilde \mZ'_1,\tilde \mZ_2,\tilde \mZ_2'$ are independent,  \eqref{inlemma: left: ssc: G2} implies
\begin{align*}
    P(\mathcal{V}^c)\leq 4P\lb & \sum_{i=1}^{p}1\Big\{|M^T_{i*}\tilde Z_1^T\tilde Z_2 N_{j}|>\tau\sqn/4\Big\}>2p e^{-\tau^2/K}\\
    &\ \text{ for all }j\in[q]\rb.
\end{align*}
Defining the set
$\mathcal{A}_j=\lbs\| \tilde  \mZ_2\N_{*j}\|_2\leq 2\sqn\|\N_{*j}\|_2\rbs$, we bound the above probability as follows:
\begin{align}\label{inlemma: CT: concentration: Vc}
 P(\mathcal{V}^c)\leq  4\sum_{j=1}^{q}P\lb &  \sum_{i=1}^{p}1\{|M^T_{i*}\tilde Z_1^T\tilde Z_2 N_{j}|>\tau\sqn/4\}\nn\\
 & >2p\exp(-\tau^2/K_0)\ \bl \ \tilde \mZ_2\in \mathcal A_j\rb\nn \\
 +&\ 4\sum_{j=1}^q P\slb \tilde \mZ_2\in\mathcal A_j^c \srb.
\end{align}
Now note that $\tilde \mZ_2\N_{*j}\sim N_n\slb 0,\|\N_{*j}\|_2^2I_n\srb$, or
${\tilde \mZ_2\N_{*j}}/{\|\N_{*j}\|_2}\sim N_n(0,I_n)$.
Therefore, there exists a universal constant $c>0$ so that  \begin{align}\label{inlemma: CT: concentration: Z2}
    \sum_{j=1}^qP(\tilde \mZ_2\in \mathcal{A}_j^c)=&\ \sum_{j=1}^q P\lb  \|\N_{*j}\|_2^{-1}\|\tilde  \mZ_2\N_{*j}\|_2>2\sqn \rb\nn\\
    \leq &\  q\exp(-cn),
\end{align}
where the last bound is due to the Chi-square tail bound in Fact~\ref{fact: Chi square tail probability} (see also Lemma 1 of \cite{laurent2000} and Lemma 12 of \cite{deshpande2014}).
Therefore, it only remains to bound the first term in \eqref{inlemma: CT: concentration: Vc}.
 We begin with an expansion of $|M^T_{i*}\tilde Z_1^T\tilde Z_2 N_{j}|$ as follows
\begin{align*}
   |M^T_{i*}\tilde Z_1^T\tilde Z_2 N_{j}|= &\ \bl\sum_{l=1}^{p}\sum_{k=1}^n\M_{il}(\tilde \mZ_1)_{kl}(\tilde \mZ_2\N)_{kj}\bl\\
   =&\ \bl\sum_{l=1}^{p}\M_{il}\underbrace{\sum_{k=1}^n(\tilde \mZ_1)_{kl}(\tilde \mZ_2\N)_{kj}}_{\Psi^j_{l}}\bl.
\end{align*}
Since $\tilde \mZ_1$ and $\tilde \mZ_2$ are independent, $\tilde \mZ_1$ conditioned on $\tilde \mZ_2$ is still a  standard Gaussian data matrix.
Hence, for $l\in[p]$, conditional on $\tilde \mZ_2$, $\Psi_{l}^j$'s are independent $N(0,\|\tilde Z_2N_{*j}\|^2_2)$ random variables. As a result, for each $l\in[p]$ and $j\in[q]$, $\Psi_l^j$ can be written as $\|\tilde Z_2N_{*j}\|_2\mathbb Z_l$, where $\mathbb Z_l=\Psi_l^j/\|\tilde \mZ_2 N_{*j}\|_2$, and  $\mathbb \mZ_1,\ldots, \mathbb Z_p\mid \tilde \mZ_2\iid N(0,1)$. 
Noting $\|N_j\|_2\leq\|N\|_{op}$ for every $j\in[q]$,  we derive the following bound provided $\tilde \mZ_2\in \mathcal A_j$: 
\begin{align*}
  \MoveEqLeft  \sum_{i=1}^{p}1\{|M^T_{i*}\tilde Z_1^T\tilde Z_2 N_{j}|> \tau\sqn/4\}\\
    =&\ \sum_{i=1}^p 1\slbt \|\tilde Z_2N_{*j}\|_2\sbl\sum_{l=1}^pM_{il}\mathbb Z_l\sbl>\tau\sqrt n/4\srbt\\
    \leq &\ \sum_{i=1}^p 1\slbt \sqrt 2\|N\|_{op}\sbl\sum_{l=1}^pM_{il}\mathbb Z_l\sbl>\tau/4\srbt.
\end{align*}
Defining
\begin{equation}\label{inlemma: CT: def: f}
    f(x)\equiv f(x_1,\ldots,x_p)=\sum_{i=1}^p\frac{1\slbt |\sum_{l=1}^pM_{il}x_l|>\tau/(4\sqrt{2}\|N\|_{op})\srbt}{p},
\end{equation}
we notice that the above calculations implies conditional on $\tilde \mZ_2\in\mathcal A_j$,
\[
   \frac{\sum_{i=1}^{p}|M^T_{i*}\tilde Z_1^T\tilde Z_2 N_{j}|> \tau\sqn/4]}{p}\leq f(\mathbb Z_1,\ldots,\mathbb Z_p). 
\]
Therefore,
\begin{align}\label{inlemma: CT: concentration: bound by f}
   \MoveEqLeft P\lb \sum_{i=1}^{p}1\{|M^T_{i*}\tilde Z_1^T\tilde Z_2 N_{j}|>\tau\sqn/4\}>2p e^{-\tau^2/K}\mid \tilde \mZ_2\in \mathcal A_j\rb\nn\\
    \leq&\ P\slb f(\mathbb Z_1,\ldots,\mathbb Z_p)>2\exp(-\tau^2/K) \mid \tilde \mZ_2\in \mathcal A_j\srb,
\end{align}
which is is bounded by $\exp(-2\sqrt{p})$
by Lemma~\ref{lemma: CT: Chernoff}.  Therefore, \eqref{inlemma: CT: concentration: Vc}, \eqref{inlemma: CT: concentration: Z2}, and \eqref{inlemma: CT: concentration: bound by f} jointly imply that
\[P(\mathcal V^c)\leq 4 q\exp(-cn)+4q\exp(-2\sqrt{p}).\]
Therefore $P(\mathcal G_2^c)$ is bounded by
\begin{align*}
  \MoveEqLeft   \exp(-c(n+p))+\exp(-c(n+q))+4 q\exp(-cn)\\
  &\ +4q\exp(-2\sqrt{p})
     \leq 4q\exp(-c\min\{n,\sqrt{p}\}),
\end{align*}
which completes the proof.

\end{proof}

\begin{lemma}
\label{lemma: CT: Chernoff}
Suppose  $160\|M\|^2_{op}\|N\|^2_{op}<K,\tau^2$ and $\tau<\sqrt{K\log p}/2$. Further suppose $\mathbb Z_1,\ldots \mathbb Z_p$  are independent standard Gaussian random variables. Then the function $f$  defined in \eqref{inlemma: CT: def: f} satisfies
\[\mathbb P\slb f(\mathbb Z_1,\ldots,\mathbb Z_p)>2\exp(-\tau^2/K)\srb\leq \exp(-2\sqrt{p}).\]
\end{lemma}

\begin{proof}[Proof of Lemma~\ref{lemma: CT: Chernoff}]
Note that $pf(\mathbb Z_1,\ldots,\mathbb Z_p)$ is a sum of dependent Bernoulli random variables. Therefore the traditional Chernoff’s or Hoeffding's  bound for inependent Bernoulli random variables will not apply.
We use a generalized version of Chernoff’s inequality, originally due to \cite{panconesi1997} (also discussed by \cite{pelekis2015, linial2014} among others), which applies to weakly dependent Bernoulli random variables.
\begin{lemma}[\cite{panconesi1997}]
\label{lemma: Linial2014}
Let $X_1,\ldots, X_p$ be Bernoulli random variables and $\e\in(0,1)$. Suppose there exists $\delta\in(0,\e)$ such that for any $\mathcal B\subset[p]$, the following assertion holds:
\begin{equation}\label{cond: linial}
    \E\slbt \prod_{i\in \mathcal B}X_i\srbt \leq \delta^{|\mathcal B|}.
\end{equation}
For $x,y\in(0,1)$, we denote
\[D(x\mm y)=y\log(y/x)+(1-y)\log((1-y)/(1-x)).\]
Then we have
\[\mathbb P\lbt \frac{\sum_{i=1}^pX_i}{p}\geq \e \rbt\leq \exp\slb-pD(\delta\mm \e)\srb.\]
\end{lemma}

Note that if we take $X_i=1\{ |\sum_{l=1}^pM_{il}\mathbb Z_l|>\tau/(4\sqrt{2}\|N\|_{op})\}$ and $\e=2\exp(-\tau^2/K)$, then the above  lemma can be applied to bound $P( f(\mathbb Z_1,\ldots,\mathbb Z_p)>2\exp(-\tau^2/K))$ provided \eqref{cond: linial} holds, which will be referred as the weak dependence Condition from now on. Suppose $|\mathcal B|=k$. For the sake of simplicity, we take $\mathcal B=\{1,\ldots, k\}$. The  arguments, which are to follow, would hold for any other choice of $\mathcal B$ as well as long as $\|\mathcal B|=k$. Denote by $M_k$ the submatrix of $M$ containing only the first $k$ rows of $M$. Let us denote $\mathbb Z_{1:k}=(\mathbb Z_1,\ldots,\mathbb Z_k)$. Letting $t=\tau/(4\sqrt{2}\|N\|_{op})$, we observe that for our choice of $X_i$'s, $\E[\prod_{i\in \mathcal B}X_i]$ equals
\begin{align*}
      P\slb |M_{i*}^T\mathbb Z_{1:k}|>t,\  l\in[k]\srb
     \leq &\  P\slb \mathbb Z_{1:k}^T M_k^TM_k \mathbb Z_{1:k}>kt^2\srb\\
     \leq &\ P\slb \|M_k^TM_k\|_{op} \sum_{i=1}^k \mathbb \mZ^2_{l}  >kt^2\srb.
\end{align*}
The operator norm $\|M_k^TM_k\|_{op}$ equals $\|M_k\|_{op}^2$, which is bounded by $\|M\|_{op}^2$ by Lemma~\ref{lemma: technical}B. Therefore, the right hand side of the last display is bounded by $P( \sum_{l=1}^k\mathbb Z_l^2>kt^2/\|M\|_{op}^2)$. 
By Chi-square tail bounds (see for instance Fact~\ref{fact: Chi square tail probability}), the latter probability is bounded above by $\exp(-kt^2/(5\|M\|_{op}^2))$ for all $t>\sqrt{5}\|M\|_{op}$. Since $t=\tau/(4\sqrt{2}\|N\|_{op})$, note that $\tau>\sqrt{160}\|M\|_{op}\|N\|_{op}$ suffices. For such $\tau$, we have thus shown that
\[\E\slbt\prod_{i\in \mathcal B}X_i\srbt\leq \exp(-|\mathcal B|\frac{\tau^2}{160\|M\|_{op}^2\|N\|_{op}^2}).\]
Thus our 
$\delta= \exp(-\frac{\tau^2}{160\|M\|_{op}^2\|N\|_{op}^2})$,
which is less than $\e/2=\exp(-\tau^2/K)$ because  $K>160\|M\|_{op}^2\|N\|_{op}^2$. Thus our $(\delta,\e)$ pair satisfies the weak dependence condition. Therefore by 
Lemma~\ref{lemma: Linial2014}, it follows that
\[\mathbb P\slb f(\mathbb Z_1,\ldots,\mathbb Z_p)>2\exp(-\tau^2/K)\srb\leq \exp(-p D(\delta\mm \epsilon)).\]
We will now use the lower bound $D(x\mm y)\geq 2(x-y)^2$ for $x,y\in(0,1)$. Because $|\delta-\e|\leq \e/2$, $D(\delta\mm\e)\geq \e^2/2$, indicating
$p D(\delta\mm \epsilon)\geq 2p\exp(-{2\tau^2}/{K}),$
which is greater than $2\sqrt p$ if $2\tau^2/K\leq \log p/2$, or equivalently $\tau^2\leq (K\log p)/4$. 
 Therefore, the current lemma follows.
\end{proof}

\FloatBarrier
\bibliographystyle{IEEEtran}
\bibliography{sparseCCA,biblio_cca_support}
\vskip 0pt plus -1fil

\begin{IEEEbiographynophoto}{Nilanjana Laha}
received  a Bachelor of Statistics in 2012 and a Master of Statistics in 2014 from the Indian Statistical Institute, Kolkata. Then she received a Ph.D. in statistics in 2019 from the University of Washington, Seattle. She was a postdoctoral research fellow at the  department of Biostatistics at Harvard university from 2019 to 2022. She is currently an assistant professor in Statistics at Texas A \& M University. Her
research interests include dynamic treatment regimes, high dimensional association, and shape constrained inference.
\end{IEEEbiographynophoto}
\vskip 0pt plus -1fil
\begin{IEEEbiographynophoto}{Rajarshi Mukherjee}
received  a Bachelor of Statistics in 2007 and a Master of Statistics in 2009 from the Indian Statistical Institute, Kolkata. He received his Ph.D. degree in Bisostatistics from Harvard University in  2014. He was a Stein fellow in the department of Statistics at Stanford University from 2014 to 2017. He was an assistant professor at the division of Biostatistics at the University of California, Berkeley, from 2017 to 2018. Since 2018, he has been an assistant professor at the department of Biostatistics at Harvard University. His research interests primarily lie in  structured signal detection problems in high dimensional and network models, and functional estimation and adaptation theory in nonparametric statistics.
\end{IEEEbiographynophoto}
\end{document}